\documentclass[a4paper,english]{article}

\usepackage{amsmath,amssymb}
\usepackage{tikz}
\usepackage{subcaption,graphicx}
\usepackage{color}
\usepackage{bm}
\usepackage{url}

\usepackage{diffcoeff}
\usepackage{listings}
\usepackage{spverbatim}
\newcommand{\bbR}{\mathbb{R}}
\newcommand{\bbA}{\mathbb{A}}
\newcommand{\bbB}{\mathbb{B}}

\newcommand{\inner}[3][]{\ensuremath{\left\langle #2, \; #3 \right\rangle_{#1}}}
\newcommand{\bilprod}[2]{\left\langle \left\langle \, #1, #2 \, \right\rangle \right\rangle}
\newcommand{\where}{\qquad \text{where} \qquad}
\newcommand*{\norm}[1]{\ensuremath{\left\|#1\right\|}}

\DeclareMathOperator*{\grad}{grad}
\DeclareMathOperator*{\Grad}{Grad}
\DeclareMathOperator*{\Div}{Div}
\renewcommand{\div}{\operatorname{div}}

\DeclareMathOperator{\Tr}{Tr}
\DeclareMathOperator*{\esssup}{ess\,sup}

\makeatletter \renewcommand\d[1]{\ensuremath{%
		\;\mathrm{d}#1\@ifnextchar\d{\!}{}}}
\makeatother

\def\onedot{$\mathsurround0pt\ldotp$}
\def\cddot{
	\mathbin{\vcenter{\baselineskip.67ex
			\hbox{\onedot}\hbox{\onedot}}%
}}

\newtheorem{assumption}{Assumption}
\newtheorem{conjecture}{Conjecture}

\newtheorem{remark}{Remark}
\newtheorem{proposition}{Proposition}
\newtheorem{definition}{Definition}
\newtheorem{proof}{Proof}

    \usepackage[breakable]{tcolorbox}
    \tcbset{nobeforeafter} 
    \usepackage{float}
    \usepackage[T1]{fontenc}

    \usepackage{graphicx}
    \makeatletter
    \usepackage{caption}

    \usepackage{adjustbox} 
    \usepackage{xcolor} 
    \usepackage{enumerate} 
    \usepackage{geometry} 
    \usepackage{amsmath} 
    \usepackage{amssymb} 
    \usepackage{textcomp} 
    \AtBeginDocument{%
        \def\PYZsq{\textquotesingle}
    }
    \usepackage{upquote} 
    \usepackage{eurosym} 
    \usepackage[mathletters]{ucs} 
    \usepackage[utf8x]{inputenc} 
    \usepackage{fancyvrb} 
    \usepackage{grffile} 
    \usepackage{hyperref}
    \usepackage{longtable} 
    \usepackage{booktabs}  
    \usepackage{mathrsfs}

    \definecolor{urlcolor}{rgb}{0,.145,.698}
    \definecolor{linkcolor}{rgb}{.71,0.21,0.01}
    \definecolor{citecolor}{rgb}{.12,.54,.11}

    \definecolor{ansi-black}{HTML}{3E424D}
    \definecolor{ansi-black-intense}{HTML}{282C36}
    \definecolor{ansi-red}{HTML}{E75C58}
    \definecolor{ansi-red-intense}{HTML}{B22B31}
    \definecolor{ansi-green}{HTML}{00A250}
    \definecolor{ansi-green-intense}{HTML}{007427}
    \definecolor{ansi-yellow}{HTML}{DDB62B}
    \definecolor{ansi-yellow-intense}{HTML}{B27D12}
    \definecolor{ansi-blue}{HTML}{208FFB}
    \definecolor{ansi-blue-intense}{HTML}{0065CA}
    \definecolor{ansi-magenta}{HTML}{D160C4}
    \definecolor{ansi-magenta-intense}{HTML}{A03196}
    \definecolor{ansi-cyan}{HTML}{60C6C8}
    \definecolor{ansi-cyan-intense}{HTML}{258F8F}
    \definecolor{ansi-white}{HTML}{C5C1B4}
    \definecolor{ansi-white-intense}{HTML}{A1A6B2}
    \definecolor{ansi-default-inverse-fg}{HTML}{FFFFFF}
    \definecolor{ansi-default-inverse-bg}{HTML}{000000}

    \DefineVerbatimEnvironment{Highlighting}{Verbatim}{commandchars=\\\{\}}

    \let\Oldtex\TeX
    \let\Oldlatex\LaTeX
    \renewcommand{\TeX}{\textrm{\Oldtex}}
    \renewcommand{\LaTeX}{\textrm{\Oldlatex}}

\makeatletter
\def\PY@reset{\let\PY@it=\relax \let\PY@bf=\relax%
    \let\PY@ul=\relax \let\PY@tc=\relax%
    \let\PY@bc=\relax \let\PY@ff=\relax}
\def\PY@tok#1{\csname PY@tok@#1\endcsname}
\def\PY@toks#1+{\ifx\relax#1\empty\else%
    \PY@tok{#1}\expandafter\PY@toks\fi}
\def\PY@do#1{\PY@bc{\PY@tc{\PY@ul{%
    \PY@it{\PY@bf{\PY@ff{#1}}}}}}}
\def\PY#1#2{\PY@reset\PY@toks#1+\relax+\PY@do{#2}}

\expandafter\def\csname PY@tok@w\endcsname{\def\PY@tc##1{\textcolor[rgb]{0.73,0.73,0.73}{##1}}}
\expandafter\def\csname PY@tok@c\endcsname{\let\PY@it=\textit\def\PY@tc##1{\textcolor[rgb]{0.25,0.50,0.50}{##1}}}
\expandafter\def\csname PY@tok@cp\endcsname{\def\PY@tc##1{\textcolor[rgb]{0.74,0.48,0.00}{##1}}}
\expandafter\def\csname PY@tok@k\endcsname{\let\PY@bf=\textbf\def\PY@tc##1{\textcolor[rgb]{0.00,0.50,0.00}{##1}}}
\expandafter\def\csname PY@tok@kp\endcsname{\def\PY@tc##1{\textcolor[rgb]{0.00,0.50,0.00}{##1}}}
\expandafter\def\csname PY@tok@kt\endcsname{\def\PY@tc##1{\textcolor[rgb]{0.69,0.00,0.25}{##1}}}
\expandafter\def\csname PY@tok@o\endcsname{\def\PY@tc##1{\textcolor[rgb]{0.40,0.40,0.40}{##1}}}
\expandafter\def\csname PY@tok@ow\endcsname{\let\PY@bf=\textbf\def\PY@tc##1{\textcolor[rgb]{0.67,0.13,1.00}{##1}}}
\expandafter\def\csname PY@tok@nb\endcsname{\def\PY@tc##1{\textcolor[rgb]{0.00,0.50,0.00}{##1}}}
\expandafter\def\csname PY@tok@nf\endcsname{\def\PY@tc##1{\textcolor[rgb]{0.00,0.00,1.00}{##1}}}
\expandafter\def\csname PY@tok@nc\endcsname{\let\PY@bf=\textbf\def\PY@tc##1{\textcolor[rgb]{0.00,0.00,1.00}{##1}}}
\expandafter\def\csname PY@tok@nn\endcsname{\let\PY@bf=\textbf\def\PY@tc##1{\textcolor[rgb]{0.00,0.00,1.00}{##1}}}
\expandafter\def\csname PY@tok@ne\endcsname{\let\PY@bf=\textbf\def\PY@tc##1{\textcolor[rgb]{0.82,0.25,0.23}{##1}}}
\expandafter\def\csname PY@tok@nv\endcsname{\def\PY@tc##1{\textcolor[rgb]{0.10,0.09,0.49}{##1}}}
\expandafter\def\csname PY@tok@no\endcsname{\def\PY@tc##1{\textcolor[rgb]{0.53,0.00,0.00}{##1}}}
\expandafter\def\csname PY@tok@nl\endcsname{\def\PY@tc##1{\textcolor[rgb]{0.63,0.63,0.00}{##1}}}
\expandafter\def\csname PY@tok@ni\endcsname{\let\PY@bf=\textbf\def\PY@tc##1{\textcolor[rgb]{0.60,0.60,0.60}{##1}}}
\expandafter\def\csname PY@tok@na\endcsname{\def\PY@tc##1{\textcolor[rgb]{0.49,0.56,0.16}{##1}}}
\expandafter\def\csname PY@tok@nt\endcsname{\let\PY@bf=\textbf\def\PY@tc##1{\textcolor[rgb]{0.00,0.50,0.00}{##1}}}
\expandafter\def\csname PY@tok@nd\endcsname{\def\PY@tc##1{\textcolor[rgb]{0.67,0.13,1.00}{##1}}}
\expandafter\def\csname PY@tok@s\endcsname{\def\PY@tc##1{\textcolor[rgb]{0.73,0.13,0.13}{##1}}}
\expandafter\def\csname PY@tok@sd\endcsname{\let\PY@it=\textit\def\PY@tc##1{\textcolor[rgb]{0.73,0.13,0.13}{##1}}}
\expandafter\def\csname PY@tok@si\endcsname{\let\PY@bf=\textbf\def\PY@tc##1{\textcolor[rgb]{0.73,0.40,0.53}{##1}}}
\expandafter\def\csname PY@tok@se\endcsname{\let\PY@bf=\textbf\def\PY@tc##1{\textcolor[rgb]{0.73,0.40,0.13}{##1}}}
\expandafter\def\csname PY@tok@sr\endcsname{\def\PY@tc##1{\textcolor[rgb]{0.73,0.40,0.53}{##1}}}
\expandafter\def\csname PY@tok@ss\endcsname{\def\PY@tc##1{\textcolor[rgb]{0.10,0.09,0.49}{##1}}}
\expandafter\def\csname PY@tok@sx\endcsname{\def\PY@tc##1{\textcolor[rgb]{0.00,0.50,0.00}{##1}}}
\expandafter\def\csname PY@tok@m\endcsname{\def\PY@tc##1{\textcolor[rgb]{0.40,0.40,0.40}{##1}}}
\expandafter\def\csname PY@tok@gh\endcsname{\let\PY@bf=\textbf\def\PY@tc##1{\textcolor[rgb]{0.00,0.00,0.50}{##1}}}
\expandafter\def\csname PY@tok@gu\endcsname{\let\PY@bf=\textbf\def\PY@tc##1{\textcolor[rgb]{0.50,0.00,0.50}{##1}}}
\expandafter\def\csname PY@tok@gd\endcsname{\def\PY@tc##1{\textcolor[rgb]{0.63,0.00,0.00}{##1}}}
\expandafter\def\csname PY@tok@gi\endcsname{\def\PY@tc##1{\textcolor[rgb]{0.00,0.63,0.00}{##1}}}
\expandafter\def\csname PY@tok@gr\endcsname{\def\PY@tc##1{\textcolor[rgb]{1.00,0.00,0.00}{##1}}}
\expandafter\def\csname PY@tok@ge\endcsname{\let\PY@it=\textit}
\expandafter\def\csname PY@tok@gs\endcsname{\let\PY@bf=\textbf}
\expandafter\def\csname PY@tok@gp\endcsname{\let\PY@bf=\textbf\def\PY@tc##1{\textcolor[rgb]{0.00,0.00,0.50}{##1}}}
\expandafter\def\csname PY@tok@go\endcsname{\def\PY@tc##1{\textcolor[rgb]{0.53,0.53,0.53}{##1}}}
\expandafter\def\csname PY@tok@gt\endcsname{\def\PY@tc##1{\textcolor[rgb]{0.00,0.27,0.87}{##1}}}
\expandafter\def\csname PY@tok@err\endcsname{\def\PY@bc##1{\setlength{\fboxsep}{0pt}\fcolorbox[rgb]{1.00,0.00,0.00}{1,1,1}{\strut ##1}}}
\expandafter\def\csname PY@tok@kc\endcsname{\let\PY@bf=\textbf\def\PY@tc##1{\textcolor[rgb]{0.00,0.50,0.00}{##1}}}
\expandafter\def\csname PY@tok@kd\endcsname{\let\PY@bf=\textbf\def\PY@tc##1{\textcolor[rgb]{0.00,0.50,0.00}{##1}}}
\expandafter\def\csname PY@tok@kn\endcsname{\let\PY@bf=\textbf\def\PY@tc##1{\textcolor[rgb]{0.00,0.50,0.00}{##1}}}
\expandafter\def\csname PY@tok@kr\endcsname{\let\PY@bf=\textbf\def\PY@tc##1{\textcolor[rgb]{0.00,0.50,0.00}{##1}}}
\expandafter\def\csname PY@tok@bp\endcsname{\def\PY@tc##1{\textcolor[rgb]{0.00,0.50,0.00}{##1}}}
\expandafter\def\csname PY@tok@fm\endcsname{\def\PY@tc##1{\textcolor[rgb]{0.00,0.00,1.00}{##1}}}
\expandafter\def\csname PY@tok@vc\endcsname{\def\PY@tc##1{\textcolor[rgb]{0.10,0.09,0.49}{##1}}}
\expandafter\def\csname PY@tok@vg\endcsname{\def\PY@tc##1{\textcolor[rgb]{0.10,0.09,0.49}{##1}}}
\expandafter\def\csname PY@tok@vi\endcsname{\def\PY@tc##1{\textcolor[rgb]{0.10,0.09,0.49}{##1}}}
\expandafter\def\csname PY@tok@vm\endcsname{\def\PY@tc##1{\textcolor[rgb]{0.10,0.09,0.49}{##1}}}
\expandafter\def\csname PY@tok@sa\endcsname{\def\PY@tc##1{\textcolor[rgb]{0.73,0.13,0.13}{##1}}}
\expandafter\def\csname PY@tok@sb\endcsname{\def\PY@tc##1{\textcolor[rgb]{0.73,0.13,0.13}{##1}}}
\expandafter\def\csname PY@tok@sc\endcsname{\def\PY@tc##1{\textcolor[rgb]{0.73,0.13,0.13}{##1}}}
\expandafter\def\csname PY@tok@dl\endcsname{\def\PY@tc##1{\textcolor[rgb]{0.73,0.13,0.13}{##1}}}
\expandafter\def\csname PY@tok@s2\endcsname{\def\PY@tc##1{\textcolor[rgb]{0.73,0.13,0.13}{##1}}}
\expandafter\def\csname PY@tok@sh\endcsname{\def\PY@tc##1{\textcolor[rgb]{0.73,0.13,0.13}{##1}}}
\expandafter\def\csname PY@tok@s1\endcsname{\def\PY@tc##1{\textcolor[rgb]{0.73,0.13,0.13}{##1}}}
\expandafter\def\csname PY@tok@mb\endcsname{\def\PY@tc##1{\textcolor[rgb]{0.40,0.40,0.40}{##1}}}
\expandafter\def\csname PY@tok@mf\endcsname{\def\PY@tc##1{\textcolor[rgb]{0.40,0.40,0.40}{##1}}}
\expandafter\def\csname PY@tok@mh\endcsname{\def\PY@tc##1{\textcolor[rgb]{0.40,0.40,0.40}{##1}}}
\expandafter\def\csname PY@tok@mi\endcsname{\def\PY@tc##1{\textcolor[rgb]{0.40,0.40,0.40}{##1}}}
\expandafter\def\csname PY@tok@il\endcsname{\def\PY@tc##1{\textcolor[rgb]{0.40,0.40,0.40}{##1}}}
\expandafter\def\csname PY@tok@mo\endcsname{\def\PY@tc##1{\textcolor[rgb]{0.40,0.40,0.40}{##1}}}
\expandafter\def\csname PY@tok@ch\endcsname{\let\PY@it=\textit\def\PY@tc##1{\textcolor[rgb]{0.25,0.50,0.50}{##1}}}
\expandafter\def\csname PY@tok@cm\endcsname{\let\PY@it=\textit\def\PY@tc##1{\textcolor[rgb]{0.25,0.50,0.50}{##1}}}
\expandafter\def\csname PY@tok@cpf\endcsname{\let\PY@it=\textit\def\PY@tc##1{\textcolor[rgb]{0.25,0.50,0.50}{##1}}}
\expandafter\def\csname PY@tok@c1\endcsname{\let\PY@it=\textit\def\PY@tc##1{\textcolor[rgb]{0.25,0.50,0.50}{##1}}}
\expandafter\def\csname PY@tok@cs\endcsname{\let\PY@it=\textit\def\PY@tc##1{\textcolor[rgb]{0.25,0.50,0.50}{##1}}}

\def\PYZbs{\char`\\}
\def\PYZus{\char`\_}

\def\PYZlt{\char`\<}

\def\PYZhy{\char`\-}
\def\PYZsq{\char`\'}

\makeatother

    \makeatletter
        \newbox\Wrappedcontinuationbox 
        \newbox\Wrappedvisiblespacebox 
        \newcommand*\Wrappedvisiblespace {\textcolor{red}{\textvisiblespace}} 
        \newcommand*\Wrappedcontinuationsymbol {\textcolor{red}{\llap{\tiny$\m@th\hookrightarrow$}}} 
        \newcommand*\Wrappedcontinuationindent {3ex } 
        \newcommand*\Wrappedafterbreak {\kern\Wrappedcontinuationindent\copy\Wrappedcontinuationbox} 
        \newcommand*\Wrappedbreaksatspecials {%
            \def\PYGZus{\discretionary{\char`\_}{\Wrappedafterbreak}{\char`\_}}%
            \def\PYGZob{\discretionary{}{\Wrappedafterbreak\char`\{}{\char`\{}}%
            \def\PYGZcb{\discretionary{\char`\}}{\Wrappedafterbreak}{\char`\}}}%
            \def\PYGZca{\discretionary{\char`\^}{\Wrappedafterbreak}{\char`\^}}%
            \def\PYGZam{\discretionary{\char`\&}{\Wrappedafterbreak}{\char`\&}}%
            \def\PYGZlt{\discretionary{}{\Wrappedafterbreak\char`\<}{\char`\<}}%
            \def\PYGZgt{\discretionary{\char`\>}{\Wrappedafterbreak}{\char`\>}}%
            \def\PYGZsh{\discretionary{}{\Wrappedafterbreak\char`\#}{\char`\#}}%
            \def\PYGZpc{\discretionary{}{\Wrappedafterbreak\char`\%}{\char`\%}}%
            \def\PYGZdl{\discretionary{}{\Wrappedafterbreak\char`\$}{\char`\$}}%
            \def\PYGZhy{\discretionary{\char`\-}{\Wrappedafterbreak}{\char`\-}}%
            \def\PYGZsq{\discretionary{}{\Wrappedafterbreak\textquotesingle}{\textquotesingle}}%
            \def\PYGZdq{\discretionary{}{\Wrappedafterbreak\char`\"}{\char`\"}}%
            \def\PYGZti{\discretionary{\char`\~}{\Wrappedafterbreak}{\char`\~}}%
        } 
        \newcommand*\Wrappedbreaksatpunct {%
            \lccode`\~`\.\lowercase{\def~}{\discretionary{\hbox{\char`\.}}{\Wrappedafterbreak}{\hbox{\char`\.}}}%
            \lccode`\~`\,\lowercase{\def~}{\discretionary{\hbox{\char`\,}}{\Wrappedafterbreak}{\hbox{\char`\,}}}%
            \lccode`\~`\;\lowercase{\def~}{\discretionary{\hbox{\char`\;}}{\Wrappedafterbreak}{\hbox{\char`\;}}}%
            \lccode`\~`\:\lowercase{\def~}{\discretionary{\hbox{\char`\:}}{\Wrappedafterbreak}{\hbox{\char`\:}}}%
            \lccode`\~`\?\lowercase{\def~}{\discretionary{\hbox{\char`\?}}{\Wrappedafterbreak}{\hbox{\char`\?}}}%
            \lccode`\~`\!\lowercase{\def~}{\discretionary{\hbox{\char`\!}}{\Wrappedafterbreak}{\hbox{\char`\!}}}%
            \lccode`\~`\/\lowercase{\def~}{\discretionary{\hbox{\char`\/}}{\Wrappedafterbreak}{\hbox{\char`\/}}}%
            \catcode`\.\active
            \catcode`\,\active 
            \catcode`\;\active
            \catcode`\:\active
            \catcode`\?\active
            \catcode`\!\active
            \catcode`\/\active 
            \lccode`\~`\~ 	
        }
    \makeatother

    \let\OriginalVerbatim=\Verbatim
    \makeatletter
    \renewcommand{\Verbatim}[1][1]{%
        \sbox\Wrappedcontinuationbox {\Wrappedcontinuationsymbol}%
        \sbox\Wrappedvisiblespacebox {\FV@SetupFont\Wrappedvisiblespace}%
        \def\FancyVerbFormatLine ##1{\hsize\linewidth
            \vtop{\raggedright\hyphenpenalty\z@\exhyphenpenalty\z@
                \doublehyphendemerits\z@\finalhyphendemerits\z@
                \strut ##1\strut}%
        }%
        \def\FV@Space {%
            \nobreak\hskip\z@ plus\fontdimen3\font minus\fontdimen4\font
            \discretionary{\copy\Wrappedvisiblespacebox}{\Wrappedafterbreak}
            {\kern\fontdimen2\font}%
        }%
        
        \Wrappedbreaksatspecials
        \OriginalVerbatim[#1,codes*=\Wrappedbreaksatpunct]%
    }
    \makeatother

    \definecolor{incolor}{HTML}{303F9F}
    \definecolor{outcolor}{HTML}{D84315}
    \definecolor{cellborder}{HTML}{CFCFCF}
    \definecolor{cellbackground}{HTML}{F7F7F7}
    
    \hypersetup{
      breaklinks=true,  
      colorlinks=true,
      urlcolor=urlcolor,
      linkcolor=linkcolor,
      citecolor=citecolor,
      }
    
	\title{\textbf{Numerical approximation of port-Hamiltonian systems for hyperbolic or parabolic PDEs with boundary control}\thanks{%
This work is supported by the project ANR-16-CE92-0028,
	entitled {\em Interconnected Infinite-Dimensional systems for Heterogeneous Media}, INFIDHEM, funded by the French National
	Research Agency (ANR) and the Deutsche Forschungsgemeinschaft (DFG). Further information is available at {\url{https://websites.isae-supaero.fr/infidhem/the-project.}}}}
	\author{\textsc{Brugnoli}, Andrea\footnote{Andrea.Brugnoli@isae-supaero.fr} \and \textsc{Haine}, Ghislain\footnote{Ghislain.Haine@isae-supaero.fr} \and \textsc{Serhani}, Anass\footnote{Anass.Serhani@isae-supaero.fr} \and \textsc{Vasseur}, Xavier\footnote{Xavier.Vasseur@isae-supaero.fr}}
	\date{ISAE-SUPAERO, Universit\'e de Toulouse, France.\\
	\vspace{2mm} {10 Avenue Edouard Belin, BP-54032, 31055 Toulouse Cedex 4.}}

\begin{document}

\maketitle

\begin{abstract}
We consider the design of structure-preserving discretization methods for the solution of systems of boundary controlled Partial Differential Equations (PDEs) thanks to the port-Hamiltonian formalism. We first provide a novel general structure of infinite-dimensional port-Hamiltonian systems (pHs) for which the Partitioned Finite Element Method (PFEM) straightforwardly applies. The proposed strategy is applied to abstract multidimensional linear hyperbolic and parabolic systems of PDEs. Then we show that instructional model problems based on the wave equation, Mindlin equation and heat equation fit within this unified framework. Secondly we introduce the ongoing project {\sc{SCRIMP}} (Simulation and ContRol of Interactions in Multi-Physics) developed for the numerical simulation of infinite-dimensional pHs. {\sc{SCRIMP}} notably relies on the {\sc{FEniCS}} open-source computing platform for the finite element spatial discretization. Finally, we illustrate how to solve the considered model problems within this framework by carefully explaining the methodology. As additional support, companion interactive {\sc{Jupyter}} notebooks are available.

{\it Keywords}: port-Hamiltonian systems; Partial differential equations; Boundary control;  Structure-preserving discretization; Finite Element Method

{\it Mathematical classification (2010)}: 65M60; 35L90; 35K90
\end{abstract}
	


\section{Introduction}

The efficient numerical simulation of complex multiphysics systems is ubiquitous in Computational Science and Engineering. Although a wide range of methods exists to tackle specific problems, they often lack of versatility and adaptability, especially when the modelling is of increasing complexity as in real-world applications.

Infinite-dimensional port-Hamiltonian systems (pHs) have been first introduced in \cite{vanderschaft2002} using the language of differential geometry. They provide a powerful tool to model complex multiphysics open systems (whether or not being linear) for control purpose. A wide range of physical systems has been written within this formalism, see {\it e.g.} \cite{LeGorrec2013a,VanderSchaft2018a,Vu2016}. This twenty-year-old framework \cite{pHsReview} enjoys nice properties, such as the relevant physical meaning of the variables, a useful underlying linear structure (namely Stokes-Dirac structure) which encodes the power balance satisfied by the Hamiltonian (often chosen as an energy), and last but not least: the interconnection of multiple pHs remains a pHs. This allows a ``modular'' modelling of complex multiphysics systems.

Since then, many researchers have developed numerical methods to discretize these systems in a structure-preserving manner, hence keeping the advantages of the infinite-dimensional pHs. Such methods aim at constructing an approximate finite-dimensional pHs at the discrete level. Our aim is to show the versatility of PFEM thanks to a new unified framework, and to introduce the ongoing project SCRIMP together with companion {\sc{Jupyter}} notebooks ~\cite{bhsv:20z}. Each particular example discussed here has been treated  previously \cite{brugnoli2019mindlin,cardoso2018pfem,SerMatHai19b,SerMatHai19c}. Here, the existence of an underlying common structure for many pHs is highlighted. Obtaining such a general scheme for infinite-dimensional pHs is of major importance for control purposes \cite{toledo2020automatica}, and for coupling atomic elements into a more complex system with the guarantee of well-preserved energy exchanges between subsystems \cite{Krug2020}. 

The first proposed structure-preserving scheme for pHs dates back to \cite{golo2004}, where the authors proposed a mixed finite element spatial discretization for hyperbolic systems of conservation laws. Pseudo-spectral methods relying on higher-order global polynomial approximations were studied in \cite{moulla2012}. Unfortunately this method seems to be limited to the one-dimensional case. A finite difference method with staggered grids was developed in \cite{trenchant2018} for two-dimensional domains, but complex geometries are then difficult to tackle. Weak formulations leading to Galerkin numerical approximations began to be explored in the past few years. In \cite{kotyczka2018} the prototypical example of hyperbolic systems of two conservation laws has been discretized by a weak formulation. However the construction of the necessary power-preserving mappings is not straightforward on arbitrary meshes. All these methods require {\it ad hoc} implementations, and are usually restricted to particular cases of pHs. Furthermore, since they do not rely on well-established and versatile numerical libraries, using such techniques remains confined within a small community of experts. We refer the reader for a more complete overview of structure-preserving discretization for pHs to \cite{pHsReview,KotHDR} and the references therein.

Thanks to \cite{cardoso2018pfem} it has become clear that there exists a deep relation between structure-preserving discretization of pHs and Mixed Finite Element Method (MFEM). Indeed velocity-stress formulations for the wave dynamics \cite{kirby2015} and elastodynamics problems are of Hamiltonian type and their mixed discretization preserves this structure for \emph{closed} systems. This leads to the intuition that a MFEM may be used to discretize the underlying geometric structure of pHs in a unified way, even for \emph{open} systems, translating the infinite dimensional Stokes-Dirac structure into a finite Dirac structure. The discretization strategy relies on the partitioned structure of the problem and for this reason goes under the name of Partitioned Finite Element method (PFEM). This method proves nice convergence properties, see {\it e.g.} \cite{HaiMatSer20} for a recent proof on the wave equation, that does not require the fulfilment of the usual \emph{inf--sup} condition for MFEM, generalizing the results cited in \cite[Remark~6]{Jol03}.

It has to be pointed out that the core idea of PFEM, i.e. performing an integration by parts on a \emph{partition} of the weak formulation of the system of equations, has already been proposed for \emph{closed} hyperbolic systems in \cite{Jol03}. Therein, the formulations are called either  \emph{primal--dual} or \emph{dual--primal}, depending on the chosen partition of the system.

The major difference between MFEM and PFEM relies on the choice of the test functions in the weak formulation, hence on the finite element form functions. Indeed, in PFEM, they never carry homogeneous conditions. In \textit{e.g.} \cite[Section~7.1]{BofBreFor13}, it is shown that for a Dirichlet control, test functions are taken in the kernel of the Dirichlet trace. As already mentioned, in~\cite{Jol03}, the proposed \emph{primal--dual} and \emph{dual--primal} discretizations are then suitable for the structure-preserving discretization of \emph{closed} systems. Nevertheless, by keeping these homogeneous conditions in the test functions, not only the application of Dirichlet control is difficult, but the definition of the Neumann observation, necessary for the discrete power balance, would be more complex. PFEM aims at easing the mimicking of the continuous power balance at the discrete level, by relaxing the test functions used in~\cite{Jol03}. To the best of our knowledge, this relaxation has not been investigated yet in the case of boundary controlled wave-like systems, and this probably comes from the fact that MFEM have been first developed for elliptic systems. Indeed, in this case, such a boundary condition is mandatory for well-posedness (especially to obtain the \emph{ellipticity in the kernel} condition~\cite[Eq.~(5.1.7)]{BofBreFor13}), while PFEM is made for evolution systems, and more especially for pHs. Furthermore, we are not aware of applications of the studied scheme, being called MFEM or PFEM, to the boundary controlled heat equation, beside the first attempt presented in~\cite{SerMatHai19b,SerMatHai19c}. %

In our opinion, the driving forces of PFEM are threefold: first, PFEM takes collocated boundary controls and observations into account in a simple manner; secondly, PFEM is structure-preserving, meaning in particular that the discrete power balance perfectly mimics the continuous one; thirdly, the implementation of PFEM only relies on existing finite element libraries, such as \textsc{FEniCS} \cite{AlnaesBlechta2015a}, selected in the ongoing project SCRIMP for its robustness and efficiency. Last but not least, the pHs point of view allows us to separate axioms of physics (such as conservation laws) from constitutive laws and equations of state (such as Hooke's law and ideal gas law). PFEM is based on this separation, providing the possibility to tackle parabolic or nonlinear systems, at the price of solving a finite-dimensional \emph{port-Hamiltonian Differential Algebraic Equation (pHDAE)}. In the particular case of linear hyperbolic PDE, as shown in Section~\ref{sec:generalFramework}, the constitutive laws can be easily ({\it i.e.} without matrix inversions) taken into account in order to recover an \emph{Ordinary Differential Equation (ODE)}.

PFEM could also be named \textit{e.g.} \emph{extended MFEM} or \emph{relaxed MFEM}. Since only evolution systems are considered (not necessarily of hyperbolic type, see \textit{e.g.} Section~\ref{sec:heat}), relaxed conditions for the selection of the test functions hold, hence for the finite elements as well. We choose to follow the terminology introduced in~\cite{cardoso2018pfem} and widely used since then. Furthermore, it emphasizes the pH formulation of the initial system to discretize. 

\paragraph*{Main contributions}

We first aim at presenting the strategy of the structure-preserving discretization PFEM, in a new unified abstract framework, allowing for an easy application to a wide class of boundary controlled partial differential equations. Then, in order to show the versatility of our approach, we successively apply PFEM to the boundary controlled wave equation, the boundary controlled Mindlin plate model, and the boundary controlled heat equation with a thermodynamically well-founded Hamiltonian (namely the \emph{internal energy}, instead of the quadratic functional commonly used). Taking advantage of the strong underlying structure, we finally describe a unified object-oriented implementation of these models {\it via} PFEM. Companion interactive {\sc{Jupyter}} notebooks \cite{bhsv:20z} are discussed to illustrate our methodology.

\paragraph*{Structure of the manuscript}

The manuscript is organized as follows. In Section~\ref{sec:generalFramework}  the abstract pHs framework is introduced, with a particular focus on both hyperbolic and parabolic linear systems of partial differential equations. In Section~\ref{sec:pfem} the general  structure-preserving discretization is presented, and then specialized on the two cases previously mentioned. In Section~\ref{sec:scrimp} the ongoing  environment SCRIMP is described in detail. In Section~\ref{sec:simus} the three companion interactive {\sc{Jupyter}} notebooks \cite{bhsv:20z} are thoroughly explained. Conclusions and perspectives are finally drawn in Section~\ref{sec:conclusions}.

\section{Definition of the general framework}
\label{sec:generalFramework}
{In this section, we introduce an abstract class of pHs and their underlying geometric structure: the Dirac structure for the finite dimensional case and the Stokes-Dirac structure for the infinite dimensional case. For the infinite-dimensional case it is shown how hyperbolic and parabolic systems easily fit into this framework.}

\subsection{Finite dimensional port-Hamiltonian systems}

\paragraph*{State representation}

Let us begin with a classical definition of a pHs in finite dimension. Consider the time-invariant dynamical system \cite{SchJel14}:
\begin{equation}
\label{eq:finitePH}
\begin{cases}
\displaystyle \diff{ \mathbf{x} }{t} &= \left( \mathbf{J}(\mathbf{x}) - \mathbf{R}(\mathbf{x}) \right) \nabla H(\mathbf{x}) + \mathbf{B}(\mathbf{x})\mathbf{u}, \\
\mathbf{y} &= \mathbf{B}(\mathbf{x})^\top ~ \nabla H(\mathbf{x}),
\end{cases}
\end{equation}
where $H(\mathbf{x}) : X \simeq \mathbb{R}^n \rightarrow \mathbb{R}$, the Hamiltonian, is a real-valued function of the vector of \emph{energy variables} $\mathbf{x}$, bounded from below. Matrix-valued functions $\mathbf{J}(\mathbf{x})$ (the \emph{structure operator}) and $\mathbf{R}(\mathbf{x})$ (the \emph{dissipative} or \emph{resistive operator}) are skew-symmetric and symmetric positive semi-definite respectively. The control $\mathbf{u} \in \mathbb{R}^m$ is applied thanks to the matrix-valued control function $\mathbf{B}(\mathbf{x})$ of size $m \times n$. Variable $\mathbf{y} \in \bbR^m$ is the power conjugated output to the input.

Such a system is called a \emph{port-Hamiltonian system}, as it arises from the Hamiltonian modelling of a physical system and it interacts with the environment {\it via} the input $\mathbf{u}$ and the output $\mathbf{y}$, included in the formulation. The vector $\nabla H(\mathbf{x})$ is made of the \emph{co-energy variables}.

Due to the structural properties of $\mathbf{J}(\mathbf{x})$ and $\mathbf{R}(\mathbf{x})$, the port-Hamiltonian system enjoys the nice following \emph{power balance}:
\begin{equation}
\label{eq:finitePH-PowerBalance}
\diff{H}{t} = - \left( \nabla H(\mathbf{x}) \right)^\top ~ \mathbf{R}(\mathbf{x}) \nabla H(\mathbf{x}) + \mathbf{u}^\top ~ \mathbf{y} ~ \le ~ \mathbf{u}^\top ~ \mathbf{y},
\end{equation}
meaning that $\mathbf{R}(\mathbf{x})$ accounts for dissipation, and that the input--output product corresponds to the power supplied to (or took from) the system, through the control $\mathbf{u}$.

\paragraph*{Flow--effort representation}
Consider two finite dimensional vector spaces ${E}={F}\simeq\mathbb{R}^n$. The elements of ${F}$ are called \emph{flows}, while the elements of ${E}$ are called \emph{efforts}. Those are port variables and their combination gives the power flowing inside the system. The space ${B} = {F} \times {E}$ is called the bond space of power variables. Therefore, identifying ${E}$ as the dual of ${F}$, the power is defined as $\left\langle \mathbf{e}, \mathbf{f} \right\rangle := \mathbf{e}(\mathbf{f})$.
\begin{definition}[\cite{courant1990}, Def. 1.1.1]\label{def:dirstr}
Given the finite-dimensional space ${F}$ and its dual ${E}$ with respect to the inner product $\left\langle \cdot , \cdot \right\rangle_{{F}} : {F} \times {F} \rightarrow \mathbb{R}$, consider the symmetric bilinear form:
$$
\bilprod{(\mathbf{f}_1, \mathbf{e}_1)}{(\mathbf{f}_2, \mathbf{e}_2)} := {\inner{\mathbf{e}_1}{\mathbf{f}_2}} + {\inner{\mathbf{e}_2}{\mathbf{f}_1}}, \where (\mathbf{f}_i, \mathbf{e}_i) \in {B}, \; i = 1, 2.
$$
A Dirac structure on ${B} := {F} \times {E}$ is a subspace ${D} \subset {B}$, which is maximally isotropic under $\left\langle \left\langle \cdot, \cdot \right\rangle \right\rangle$.	Equivalently, a Dirac structure on ${B} := {F} \times {E}$ is a subspace ${D} \subset {B}$ which equals its orthogonal companion with respect to $\left\langle \left\langle \cdot, \cdot \right\rangle \right\rangle$, {\it i.e.} ${D} ={D}^{[\perp]}$, where:
$$
{D}^{[\perp]} := \left\{ (\mathbf{f}, \mathbf{e}) \in {B} ~ \mid ~ \bilprod{(\mathbf{f}, \mathbf{e})}{(\widetilde{\mathbf{f}}, \widetilde{\mathbf{e}})} = 0, ~ \forall ~ (\widetilde{\mathbf{f}}, \widetilde{\mathbf{e}}) \in {D} \right\}.
$$
\end{definition}

The connection between the concept of Dirac structure and pHs in its canonical form~\eqref{eq:finitePH} is achieved by considering the following \emph{ports}:
\begin{itemize}
\item
the \emph{storage ports} $(\mathbf{f}_{\mathbf{x}}, \mathbf{e}_{\mathbf{x}}) := \left( \diff{\mathbf{x}}{t}, \nabla H(\mathbf{x}) \right) \in \mathbb{R}^n \times \mathbb{R}^n$, made of the \emph{storage  flow} $\mathbf{f}_{\mathbf{x}}$ (time-derivative of the energy variables) and \emph{storage  effort} $\mathbf{e}_{\mathbf{x}}$ (the co-energy variables);
\item
the \emph{resistive (or dissipative) ports} $(\mathbf{f}_{\mathbf{R}}, \mathbf{e}_{\mathbf{R}}) \in \mathbb{R}^k \times \mathbb{R}^k$, made of the \emph{resistive (or dissipative) flow} $\mathbf{f}_{\mathbf{R}}$ and \emph{resisitive (or dissipative) effort} $\mathbf{e}_{\mathbf{R}}$;
\item
the \emph{interconnection ports} $(\mathbf{f}_{\mathbf{u}}, \mathbf{e}_{\mathbf{u}}) := ( -\mathbf{y}, \mathbf{u} ) \in \mathbb{R}^m \times \mathbb{R}^m$, made of the \emph{interconnection flow} $\mathbf{f}_{\mathbf{u}}$ and \emph{interconnection effort} $\mathbf{e}_{\mathbf{u}}$.
\end{itemize}

Assuming that the matrix $\mathbf{R}(\mathbf{x})$ has constant rank, from classical matrix factorizations there exist matrices $\mathbf{G}$ (not necessarily square, of size $k \times n$) and $\mathbf{K}$ symmetric positive semi-definite of size $k \times k$ such that $\mathbf{R} = \mathbf{G}^\top ~ \mathbf{K} \mathbf{G}$. These notations at hand, the  pHs~\eqref{eq:finitePH} rewrites:
\begin{equation}
\label{eq:finitePH-Dirac}
\begin{pmatrix}
\mathbf{f}_{\mathbf{x}}(\mathbf{x}) \\
\mathbf{f}_{\mathbf{R}}(\mathbf{x}) \\
\mathbf{f}_{\mathbf{u}}(\mathbf{x})
\end{pmatrix}
=
\begin{bmatrix}
\mathbf{J}(\mathbf{x}) 			& -\mathbf{G}(\mathbf{x})^\top	& \mathbf{B}(\mathbf{x}) \\
\mathbf{G}(\mathbf{x})			& \mathbf{0}				& \mathbf{0} \\
- \mathbf{B}(\mathbf{x})^\top	& \mathbf{0}				& \mathbf{0} \\
\end{bmatrix}
\begin{pmatrix}
\mathbf{e}_{\mathbf{x}}(\mathbf{x}) \\
\mathbf{e}_{\mathbf{R}}(\mathbf{x}) \\
\mathbf{e}_{\mathbf{u}}(\mathbf{x})
\end{pmatrix},
\end{equation}
together with the \emph{(dissipative or resistive) constitutive relation}:
\begin{equation}
\label{eq:finitePH-Constitutive}
\mathbf{e}_{\mathbf{R}}(\mathbf{x}) = \mathbf{K}(\mathbf{x}) \mathbf{f}_{\mathbf{R}}(\mathbf{x}).
\end{equation}

It is clear that the \emph{extended structure operator} $\mathbf{J}_e$ appearing in~\eqref{eq:finitePH-Dirac} is skew-symmetric of size $(n+k+m) \times (n+k+m)$. Its graph is a Dirac structure with respect to the Euclidean inner product, as a kernel representation, see~\cite{SchJel14}. Hence, it comes:
$$
\inner[\mathbb{R}^n]{\mathbf{e}_{\mathbf{x}}(\mathbf{x})}{\mathbf{f}_{\mathbf{x}}(\mathbf{x})} 
+ \inner[\mathbb{R}^k]{\mathbf{e}_{\mathbf{R}}(\mathbf{x})}{\mathbf{f}_{\mathbf{R}}(\mathbf{x})} 
+ \inner[\mathbb{R}^m]{\mathbf{e}_{\mathbf{u}}(\mathbf{x})}{\mathbf{f}_{\mathbf{u}}(\mathbf{x})}
= 0.
$$
Noting that $\diff{H}{t} = \inner[\mathbb{R}^n]{\mathbf{e}_{\mathbf{x}}(\mathbf{x})}{\mathbf{f}_{\mathbf{x}}(\mathbf{x})} $ (by definition of the storage port) leads to:
$$
\diff{H}{t} = - \inner[\mathbb{R}^k]{\mathbf{e}_{\mathbf{R}}(\mathbf{x})}{\mathbf{f}_{\mathbf{R}}(\mathbf{x})} 
- \inner[\mathbb{R}^m]{\mathbf{e}_{\mathbf{u}}(\mathbf{x})}{\mathbf{f}_{\mathbf{u}}(\mathbf{x})},
$$
and thanks to the symmetry of $\mathbf{R}$,~\eqref{eq:finitePH-Constitutive} gives:
$$
\diff{H}{t} = - \inner[\mathbb{R}^m]{\mathbf{f}_{\mathbf{R}}(\mathbf{x})}{\mathbf{K}( \mathbf{x}) \mathbf{f}_{\mathbf{R}}(\mathbf{x})} 
- \inner[\mathbb{R}^m]{\mathbf{e}_{\mathbf{u}}(\mathbf{x})}{\mathbf{f}_{\mathbf{u}}(\mathbf{x})}.
$$
Finally, from~\eqref{eq:finitePH-Dirac} and the definition of the storage and interconnection ports, the power balance~\eqref{eq:finitePH-PowerBalance} is recovered.

The relation between~\eqref{eq:finitePH} and~\eqref{eq:finitePH-Dirac}--\eqref{eq:finitePH-Constitutive} can be understood as follows: the power balance~\eqref{eq:finitePH-PowerBalance} is \emph{encoded} in the Dirac structure (obtained from the extended structure operator $\mathbf{J}_e$) together with the resistive constitutive relation.

\begin{remark}\label{rem:MassMatrices}
The canonical Euclidean inner product has been used here, but other inner products are allowed to take into account \emph{mass matrices} (symmetric positive definite) on the left-hand side of~\eqref{eq:finitePH},~\eqref{eq:finitePH-Dirac}, and~\eqref{eq:finitePH-Constitutive}. This is crucial after the spatial discretization procedure. This corresponds to a \emph{kernel representation} of Dirac structure \cite{SchJel14}.
\end{remark}

System \ref{eq:finitePH} is a pHs in canonical form. Recently, finite-dimensional differential algebraic port-Hamiltonian systems (pHDAE) have been introduced both for linear \cite{beattie2018linear} and nonlinear systems \cite{morandin2019}. This enriched description shares not only all the crucial features of ordinary pHs, but also easily accounts for algebraic constraints, time-dependent transformations and explicit dependence on time in the Hamiltonian. The application of the proposed discretization method naturally leads  to pHDAEs. Indeed, a constitutive relation between $\mathbf{f}_{\mathbf{x}} := \diff{\mathbf{x}}{t}$ and $\mathbf{e}_{\mathbf{x}} := \nabla H(\mathbf{x})$ is needed to be well-defined. But PFEM takes into account constitutive relations apart from~\eqref{eq:finitePH-Dirac} as constraints. However, as shown later in Sections~\ref{sec:wave} and~\ref{sec:min}, the method simplifies in the case, for instance, of linear hyperbolic systems.

\subsection{Infinite-dimensional port-Hamiltonian systems}

In this section an infinite-dimensional generalization of pHs is presented. For sake of readability,  the (Stokes-)Dirac structure is first defined, and secondly, infinite-dimensional pHs are then described in both hyperbolic and parabolic cases. A more general framework can be designed, but this goes beyond the aim of this present work.

\paragraph*{Structure operator}

As to avoid functional difficulties, the analogue of the extended structure operator will not be written as in~\eqref{eq:finitePH-Dirac}. {More precisely, the control operator will not be included in an extended structure unbounded operator, but given apart. The Stokes-Dirac will be then obtained thanks to a structure operator related to the boundary control operator through an abstract Green formula.} However, like in finite dimension, the aim is to establish a link between flow and effort variables. Most importantly, the underlying Stokes-Dirac structure must encode the power balance of the dynamical system under study.

Consider a Lipschitz domain $\Omega \subset \mathbb{R}^d, \; d \in \{1,2,3\}$, and the relation:
\begin{equation}\label{eq:stdir}
 \begin{pmatrix}
 \bm{f}_1 \\ \bm{f}_2
 \end{pmatrix} = 
 \begin{bmatrix}
 \bm{0} & -\mathcal{L}^* \\
 \mathcal{L} & \bm{0} \\
 \end{bmatrix}
 \begin{pmatrix}
 \bm{e}_1 \\ \bm{e}_2
 \end{pmatrix}, \quad
 \begin{aligned}
 \bm{e}_1 \in H^{\mathcal{L}} : \{\bm{e}_1 \in L^2(\Omega, \bbA) ~ \mid ~ \mathcal{L} \bm{e}_1 \in L^2(\Omega, \mathbb{B}) \},\\ 
 \bm{e}_2 \in H^{-\mathcal{L}^*} : \{\bm{e}_2 \in L^2(\Omega, \bbB) ~ \mid ~ -\mathcal{L}^* \bm{e}_2 \in L^2(\Omega, \mathbb{A}) \}.\\ 
 \end{aligned}{}
\end{equation}
By $L^2(\Omega, \mathbb{X})$ we denote the space of square integrable $\mathbb{X}$-valued functions.
Symbol $\mathbb{A}, \mathbb{B}$ denote either the space of scalars $\bbR$, vectors $\bbR^d$, symmetric tensors $\bbR^{d \times d}_{\text{sym}} =: \mathbb{S}$ or a Cartesian product of those, depending on the particular example. The operator $\mathcal{L}: H^{\mathcal{L}} \rightarrow L^2(\Omega, \mathbb{B})$ is a generic differential, and therefore linear but unbounded, operator. 
The notation $\mathcal{L}^*: H^{-\mathcal{L}^*} \rightarrow L^2(\Omega, \mathbb{A})$ denotes the formal adjoint of $\mathcal{L}$, defined by the relation:
\begin{equation}\label{eq:formalAdj}
\inner[L^2(\Omega, \bbB)]{\mathcal{L}\bm{e}_1}{\bm{e}_2} = \inner[L^2(\Omega, \bbA)]{\bm{e}_1}{\mathcal{L}^*\bm{e}_2}, \qquad \bm{e}_1 \in C_0^\infty(\Omega, \bbA), ~ \bm{e}_2 \in C_0^\infty(\Omega, \bbB).
\end{equation}
Of course, for~\eqref{eq:stdir} to be well-defined, constitutive relations are needed. Only physical laws will be taken into account when constructing the above relation on concrete examples. As pointed out in the introduction, PFEM aims at both preserving this relation and discretizing constitutive laws to close the system.

\begin{remark}
One can be confused by the lack of evolution in time in~\eqref{eq:stdir}. However, this emphasizes an important paradigm in the proposed point of view: this relation translates the \emph{time-independent} geometric structure of the pHs as an equation with differential operator (ill-posed on its own), while constitutive relations will bring back the time dependency of the problem. In particular, some flows must be the time derivative of the energy variables.
\end{remark} 

Throughout the paper, $\inner[X]{\cdot}{\cdot}$ denotes the inner product of the Hilbert space $X$. Definition \eqref{eq:formalAdj} is analogous to Definition 5.80 in \cite{rogers2004pde}. In Section \ref{sec:wave}, the operator $\mathcal{L}$ is the gradient, denoted by $\grad$, and its formal adjoint is the divergence, denoted by $\div$, from the so-called Green's formula (integration by parts). In Section \ref{sec:min}, the operator $\mathcal{L}$ contains both $\grad$ and $\Grad$. This latter corresponds to the symmetric part of the gradient and represents the deformation tensor in continuum mechanics:
\begin{equation*}
\Grad(\bm{e}) := \frac{1}{2} \left(\nabla \bm{e} + \nabla^\top \bm{e} \right), \qquad \bm{e} \in C^\infty(\bbR^d).
\end{equation*}
The formal adjoint of $\Grad$ is the tensor divergence $\Div$. For a tensor field $\bm{E}: \Omega \rightarrow \mathbb{M}:=\bbR^{d\times d}$, with components $e_{ij}$, the divergence is a vector, defined columnwise as:
\begin{equation*}
\Div(\bm E) := \left( \sum_{i = 1}^d \partial_{x_i} e_{ij} \right)_{j = 1, \dots, d}.
\end{equation*}
Finally, in Section~\ref{sec:heat}, $\mathcal{L}$ is made of the gradient and the identity operator.

\paragraph*{Stokes-Dirac structure}

Definition~\ref{def:dirstr} still remains valid in infinite dimension. Nevertheless, as stated above, the structure operator in~\eqref{eq:stdir} is not extended to include the control operator. Hence an additional assumption has to be made for $\begin{bmatrix} \bm{0} & -\mathcal{L}^* \\ \mathcal{L} & \bm{0} \\ \end{bmatrix}$ to define a Dirac structure in relation with a pHs coming from boundary control of partial differential equations. In other words, a Stokes-Dirac structure requires the specification of boundary variables in order to express a general power conservation property for \textit{open} physical systems. This assumption is based on the so-called Stokes' theorem (also known as the divergence theorem, Gauss's theorem or Ostrogradsky's theorem) and its corollaries, as the Green's formula.

\begin{assumption}[Abstract Green's formula]
The operator $\mathcal{L}$ is assumed to satisfy the abstract Green's formula:
\begin{equation}\label{eq:intbypar}
 \inner[L^2(\Omega, \bbB)]{\mathcal{L}\bm{e}_1}{\bm{e}_2} - \inner[L^2(\Omega, \bbA)]{\bm{e}_1}{\mathcal{L}^*\bm{e}_2} 
 = \inner[V_\partial,(V_\partial)']{\Gamma_0 \bm{e}_1 }{\Gamma_\perp \bm{e}_2 }, 
 \quad \forall ~ \bm{e}_1 \in H^{\mathcal{L}}, ~ \bm{e}_2 \in H^{-\mathcal{L}^*},
\end{equation}
where the right-hand side is the duality bracket at the boundary, on a well-suited boundary functional space $V_\partial$ for some trace operators $\Gamma_0, \; \Gamma_\perp$. From now on, this duality bracket will be denoted by $\inner[\partial\Omega]{\cdot}{\cdot}$ with a slight abuse of notation. 
\end{assumption}

\begin{remark}
This abstract formula is well-known in the boundary control systems theory, see {\it e.g.} \cite[Chapter~10]{TucWei09}.
\end{remark}

\begin{remark}
In practice, equation~\eqref{eq:intbypar} dictates the \emph{causalities}, {\it i.e.} the possible choices for the \emph{boundary control} $\bm{u}_\partial$ and the \emph{boundary observation} $\bm{y}_\partial$, {\it via} the equality $\inner[\partial\Omega]{\Gamma_0 \bm{e}_1 }{\Gamma_\perp \bm{e}_2 } = \inner[\partial\Omega]{\bm{u}_\partial }{\bm{y}_\partial }$ (with a slight abuse of notation for the right-hand side to make sense). Of course, the admissible causalities are also related to the well-posedness of the system under study, and in particular to the definitions of the boundary functional spaces.
\end{remark}

For sake of simplicity, a focus on the two following causalities will be made. Let the boundary variables associated to system \eqref{eq:stdir} be defined by:
\begin{equation}\label{eq:causalityWaveMin}
\bm{e}_\partial = \Gamma_\perp \bm{e}_2 \in (V_\partial)', \qquad \bm{f}_\partial = -\Gamma_0 \bm{e}_1 \in V_\partial,
\end{equation}
or the other way:
\begin{equation}\label{eq:causalityHeat}
\bm{e}_\partial = \Gamma_0 \bm{e}_1 \in V_\partial, \qquad \bm{f}_\partial = -\Gamma_\perp \bm{e}_2 \in (V_\partial)'.
\end{equation}
In light of~\eqref{eq:intbypar}, systems:
\begin{equation}\label{eq:stdir_bound}
 \begin{aligned}
 \begin{pmatrix}
 \bm{f}_1 \\ \bm{f}_2
 \end{pmatrix} &= 
 \begin{bmatrix}
 \bm{0} & -\mathcal{L}^* \\
 \mathcal{L} & \bm{0} \\
 \end{bmatrix}
 \begin{pmatrix}
 \bm{e}_1 \\ \bm{e}_2
 \end{pmatrix}, \qquad
 \begin{pmatrix}
 \bm{e}_\partial \\ \bm{f}_\partial
 \end{pmatrix} &= 
 \begin{bmatrix}
 \bm{0} & \Gamma_\perp \\
 -\Gamma_0 & \bm{0} \\
 \end{bmatrix}
 \begin{pmatrix}
 \bm{e}_1 \\ \bm{e}_2
 \end{pmatrix},
 \end{aligned}
\end{equation}
and:
\begin{equation}\label{eq:stdir_bound_bis}
 \begin{aligned}
 \begin{pmatrix}
 \bm{f}_1 \\ \bm{f}_2
 \end{pmatrix} &= 
 \begin{bmatrix}
 \bm{0} & -\mathcal{L}^* \\
 \mathcal{L} & \bm{0} \\
 \end{bmatrix}
 \begin{pmatrix}
 \bm{e}_1 \\ \bm{e}_2
 \end{pmatrix}, \qquad
 \begin{pmatrix}
 \bm{e}_\partial \\ \bm{f}_\partial
 \end{pmatrix} &= 
 \begin{bmatrix}
 \Gamma_0 & \bm{0} \\
 \bm{0} & -\Gamma_\perp \\
 \end{bmatrix}
 \begin{pmatrix}
 \bm{e}_1 \\ \bm{e}_2
 \end{pmatrix},
 \end{aligned}
\end{equation}
define Stokes-Dirac structures with respect to the bilinear pairing:
\begin{multline*}
 \bilprod{(\bm{f}_1^1, \bm{f}_2^1, \bm{f}_\partial^1, \bm{e}_1^1, \bm{e}_2^1, \bm{e}_\partial^1)}{(\bm{f}_1^2, \bm{f}_2^2, \bm{f}_\partial^2, \bm{e}_1^2, \bm{e}_2^2, \bm{e}_\partial^2)} \\
 = \inner[L^2(\Omega, \bbA)]{\bm{f}_1^1}{\bm{e}_1^2} + \inner[L^2(\Omega, \bbB)]{\bm{f}_2^1}{\bm{e}_2^2} 
 + \inner[L^2(\Omega, \bbA)]{\bm{f}_1^2}{\bm{e}_1^1} + \inner[L^2(\Omega, \bbB)]{\bm{f}_2^2}{\bm{e}_2^1} \\
 + \inner[\partial\Omega]{\bm{f}_\partial^1}{\bm{e}_\partial^2} + \inner[\partial\Omega]{\bm{f}_\partial^2}{\bm{e}_\partial^1}.
\end{multline*}
Obviously, for systems~\eqref{eq:stdir_bound} and~\eqref{eq:stdir_bound_bis} to be well-defined, constitutive relations are needed.

In the remaining part of this section, only~\eqref{eq:stdir_bound} will be treated in details. The other canonical causality \eqref{eq:stdir_bound_bis}, figuring in Section~\ref{sec:heat}, straightforwardly follows with the same strategy.

\paragraph{Hyperbolic systems}
In the hyperbolic case, both flows represent the dynamics of the independent energy variables $\bm{\alpha}_1, \; \bm{\alpha}_2$. The Hamiltonian is a generic functional of these variables $H = H(\bm{\alpha}_1, \bm{\alpha}_2)$. The co-energy variables are by definition the \textit{variational derivatives} (see {\it e.g.}~\cite{SchMas02}) of $H$ with respect to the energy variables:
\begin{equation}
 \bm{f}_1 = \partial_t\bm{\alpha}_1, \quad 
 \bm{e}_1 :=\delta_{\bm{\alpha}_1} H, \qquad 
 \bm{f}_2 = \partial_t\bm{\alpha}_2, \quad
 \bm{e}_2 :=\delta_{\bm{\alpha}_2} H.
\end{equation}
Then system \eqref{eq:stdir_bound} possesses the equivalent state representation:
\begin{equation}\label{eq:pHhyp}
 \begin{aligned}
 \begin{pmatrix}
 \partial_t \bm{\alpha}_1 \\ \partial_t \bm{\alpha}_2
 \end{pmatrix} &= 
 \begin{bmatrix}
 \bm{0} & -\mathcal{L}^* \\
 \mathcal{L} & \bm{0} \\
 \end{bmatrix}
 \begin{pmatrix}
 \delta_{\bm{\alpha}_1}{H} \\ \delta_{\bm{\alpha}_2}{H}
 \end{pmatrix}, &\qquad
 \begin{pmatrix}
 \bm{u}_\partial \\ \bm{y}_\partial
 \end{pmatrix} &= 
 \begin{bmatrix}
 \bm{0} & \Gamma_\perp \\
 \Gamma_0 & \bm{0} \\
 \end{bmatrix}
 \begin{pmatrix}
 \delta_{\bm{\alpha}_1}{H} \\ \delta_{\bm{\alpha}_2}{H}
 \end{pmatrix}.
 \end{aligned}
\end{equation}
It holds $\bm{u}_\partial = \bm{e}_\partial, \; \bm{y}_\partial=- \bm{f}_\partial$. The power balance is naturally embedded in the Stokes-Dirac structure defined by~\eqref{eq:stdir_bound}:
\begin{equation}\label{eq:powbal_hyp}
    \diff{}{t}H(\bm{\alpha}_1, \bm{\alpha}_2) = \inner[\partial\Omega]{\bm{y}_\partial}{\bm{u}_\partial}.
\end{equation}

\paragraph{Linear hyperbolic systems}
The system is linear when the Hamiltonian has the form:
$$
 {H}(\bm{\alpha}_1, \bm{\alpha}_2) = \frac{1}{2} \inner[L^2(\Omega, \bbA)]{\bm{\alpha}_1}{\mathcal{Q}_1 \bm{\alpha}_1} + \frac{1}{2} \inner[L^2(\Omega, \bbB)]{\bm{\alpha}_2}{\mathcal{Q}_2 \bm{\alpha}_2},
$$
where $\mathcal{Q}_1$, $\mathcal{Q}_2$ are positive symmetric operators, bounded from below and above:
$$
	m_1 \bm{I}_{\bbA} \le\mathcal{Q}_1 \le M_1 \bm{I}_{\bbA}, \qquad m_2 \bm{I}_{\bbB} \le \mathcal{Q}_2 \le M_2 \bm{I}_{\bbB}, \qquad m_1>0, ~ m_2>0, ~ M_1>0, ~ M_2>0,
$$
with $\bm{I}_\bbA$ and $\bm{I}_\bbB$ the identity operators in $L^2(\Omega,\bbA)$ and $L^2(\Omega,\bbB)$ respectively. In this case, the co-energy variables are given by:
\begin{equation}\label{eq:e_lin}
	\bm{e}_1 := \delta_{\alpha_1} H = \mathcal{Q}_1 \bm{\alpha}_1, \qquad \bm{e}_2 := \delta_{\alpha_2} H = \mathcal{Q}_2 \bm{\alpha}_2.
\end{equation}
Since $\mathcal{Q}_1, \, \mathcal{Q}_2$ are positive and bounded from below and above, it is possible to invert them to obtain:
\begin{equation}\label{eq:alpha_lin}
\bm{\alpha}_1 = \mathcal{Q}_1^{-1}\bm{e}_1 =: \mathcal{M}_1\bm{e}_1, \qquad \bm{\alpha}_2 = \mathcal{Q}_2^{-1} \bm{e}_2 =: \mathcal{M}_2 \bm{e}_2,
\end{equation}
giving rise to the \emph{co-energy formulation}. The Hamiltonian is rewritten as:
\begin{equation}\label{eq:H_coenergy}
H(\bm{e}_{1},\bm{e}_{2}) = \frac{1}{2} \inner[L^2(\Omega, \mathbb{A})]{\bm{e}_{1}}{\mathcal{M}_1\bm{e}_{1}} + \frac{1}{2} \inner[L^2(\Omega, \mathbb{B})]{\bm{e}_{2}}{\mathcal{M}_2\bm{e}_{2}},
\end{equation}
and a linear hyperbolic pHs~\eqref{eq:stdir_bound} can be expressed as:
\begin{equation}\label{eq:pHlinhyp}
 \begin{aligned}
 \begin{bmatrix}
	\mathcal{M}_1 & \bm{0} \\
	\bm{0} & \mathcal{M}_2 \\
	\end{bmatrix}
	\begin{pmatrix}
	\partial_t \bm{e}_1 \\ \partial_t \bm{e}_2
	\end{pmatrix} &= \begin{bmatrix}
	\bm{0} & - \mathcal{L}^* \\
	\mathcal{L} & \bm{0} \\
	\end{bmatrix}\begin{pmatrix}
	\bm{e}_1 \\ \bm{e}_2
	\end{pmatrix}, \qquad
 \begin{pmatrix}
 \bm{u}_\partial \\ \bm{y}_\partial
 \end{pmatrix} &= 
 \begin{bmatrix}
 \bm{0} & \Gamma_\perp \\
 \Gamma_0 & \bm{0} \\
 \end{bmatrix}
 \begin{pmatrix}
	\bm{e}_1 \\ \bm{e}_2
	\end{pmatrix}.
 \end{aligned}
\end{equation}
In this particular case, the constitutive relations needed for system~\eqref{eq:stdir_bound} to be well-defined are given by~\eqref{eq:e_lin}, and then directly included in~\eqref{eq:pHlinhyp}. In Sections~\ref{sec:wave} and~\ref{sec:min}, it will be shown that PFEM leads directly to a finite-dimensional pHs of the form~\eqref{eq:finitePH} with $\mathbf{R(\mathbf{x})} = 0$. This simplification considerably facilitates the solution in time, as~\eqref{eq:finitePH} is an Ordinary Differential Equation (ODE).

\paragraph{Parabolic systems}
In this case, the first flow $\bm{f}_1$ still represents a dynamics $\partial_t \bm{\alpha}_1$ of the energy variable $\bm{\alpha}_1$. The Hamiltonian then reads $H = H(\bm{\alpha}_1)$, and its variational derivative gives the co-energy variable $\bm{e}_1=\delta_{\bm{\alpha}_1} H$.

The second flow $\bm{f}_2$ represents an extra flow related to the effort variable $\bm{e}_2$ appearing in the dynamics of the energy variable $\bm{\alpha}_1$. The relation is given implicitly by a mapping $\mathcal{G}$ as $\mathcal{G}(\bm{f}_2,\bm{e}_2) = 0$. Then, pHs~\eqref{eq:stdir_bound} of parabolic type is expressed as:
\begin{equation}\label{eq:pHpar}
 \begin{aligned}
 \begin{pmatrix}
 \partial_t \bm{\alpha}_1 \\ \bm{f}_2
 \end{pmatrix} &= 
 \begin{bmatrix}
 \bm{0} & -\mathcal{L}^* \\
 \mathcal{L} & \bm{0} \\
 \end{bmatrix}
 \begin{pmatrix}
 \delta_{\bm{\alpha}_1}{H} \\ \bm{e}_2 
 \end{pmatrix}, \qquad
 \begin{pmatrix}
 \bm{u}_\partial \\ \bm{y}_\partial
 \end{pmatrix} &= 
 \begin{bmatrix}
 \bm{0} & \Gamma_\perp \\
 \Gamma_0 & \bm{0} \\
 \end{bmatrix}
 \begin{pmatrix}
 \delta_{\bm{\alpha}_1}{H} \\ \bm{e}_2
 \end{pmatrix}, \qquad
 \mathcal{G}(\bm{f}_2,\bm{e}_2) = 0.
 \end{aligned}
\end{equation}
In Section~\ref{sec:heat}, an example of a parabolic-type pHs~\eqref{eq:stdir_bound_bis} is studied. It will be shown that the PFEM structure-preserving discretization of such a system naturally leads to a finite-dimensional pHDAE. Again, the power balance is naturally embedded in the Stokes-Dirac structure defined by~\eqref{eq:stdir_bound}:
\begin{equation}\label{eq:powbal_par}
    \diff{}{t}H(\bm{\alpha}_1) = -\inner[L^2(\Omega)]{\bm{f}_2}{\bm{e}_2} + \inner[\partial\Omega]{\bm{y}_\partial}{\bm{u}_\partial}.
\end{equation}
In practice, this becomes explicit with the constitutive relation $\mathcal{G}(\bm{f}_2,\bm{e}_2) = 0$ as it will be seen in Section~\ref{sec:heat} (and more generally in~\cite{SerMatHai19b,SerMatHai19c}). Note that this latter relation has to be accurately discretized to ensure that the discretized power balance mimics the continuous one.

\begin{remark}
By adding resistive port(s), dissipation(s) can easily be taken into account (both internal or at the boundary), as done in the finite-dimensional case {\it via} $\mathbf{R}(\mathbf{x})$ playing the role of a output feedback gain matrix. In this case, the system becomes a \emph{parabolic system}, the \emph{dissipative constitutive relation} being represented by $\mathcal{G}$. See~\cite{SerMatHai19d,SerMatHai19a} for a detailed discussion about structure-preserving discretization of dissipative systems.
\end{remark}

\section{The Partitioned Finite Element Method (PFEM)}\label{sec:pfem}

We are now in a position to introduce a general methodology to discretize infinite-dimensional pHs in a structure-preserving manner. The main contribution in this section is the application of PFEM to a general abstract class of pHs, unifying the previously published results. This generality is notably of particular interest for the development of a well-structured software for the numerical simulations of physics-based models. The power balances~\eqref{eq:powbal_hyp}, \eqref{eq:powbal_par} are deeply linked to a linear underlying Stokes-Dirac. The main idea of PFEM is to mimic this structure, in order to obtain a discretized copy of these power balances as~\eqref{eq:finitePH-PowerBalance}.
This systematically translates the Stokes-Dirac structure into a finite-dimensional Dirac structure.   The compatible discretization, with respect to this Dirac structure, of the constitutive relations allows to mimic the continuous power-balance. This method goes under the name Partitioned Finite Element Method (PFEM), and was originally presented in \cite{cardoso2018pfem}. The procedure is a natural extension of {MFEM} to pHs and boils down to these three simple steps:
\begin{enumerate}
	\item System \eqref{eq:stdir} is written in weak form; 
	\item The integration by parts~\eqref{eq:intbypar} is carried out on a partition of the system \eqref{eq:stdir} to make the appropriate boundary control appear;
	\item A Galerkin method is employed to obtain a finite-dimensional system. For the approximation basis, the finite element method is used here but spectral methods can be chosen as well.
\end{enumerate}

This strategy of structured discretization in order to mimic the continuous power balance at the discrete level has been addressed for closed abstract linear hyperbolic systems in \cite{Jol03}. This pioneering work already proposed the key point of PFEM: the integration by parts on a partition of the weak formulation of the system. The author called the obtained systems \emph{primal--dual} or \emph{dual--primal} formulation, depending on which line is integrated by parts. In the port-Hamiltonian formalism, systems are opened with control and observation. It appears that \cite{Jol03} admits PFEM as a generalization for structure-preserving space discretization. The choice of a control in the pHs community is called a \emph{causality}, and \emph{primal--dual} or \emph{dual--primal} correspond in this work to the \emph{canonical} causalities~\eqref{eq:stdir_bound} and~\eqref{eq:stdir_bound_bis} respectively.

\subsection{General strategy}

Consider smooth test functions $\bm{v}_1$ and $\bm{v}_2$ and the weak form of \eqref{eq:stdir}:
\begin{equation}\label{eq:weak_dyn}
\begin{aligned}
\inner[L^2(\Omega, \bbA)]{\bm{v}_1}{\bm{f}_1} &= \inner[L^2(\Omega, \bbA)]{\bm{v}_1}{-\mathcal{L}^*\bm{e}_2}, \\
\inner[L^2(\Omega, \bbB)]{\bm{v}_2}{\bm{f}_2} &= \inner[L^2(\Omega, \bbB)]{\bm{v}_2}{\mathcal{L}\bm{e}_1}.
\end{aligned}
\end{equation}
Next the integration by parts is performed either on the first or on the second line (the system is \emph{partitioned}), depending on the causality.

\paragraph{Integration by parts of the term \texorpdfstring{$\inner[L^2(\Omega, \bbA)]{\bm{v}_1}{-\mathcal{L}^*\bm{e}_2}$}{}}
In this case case, using \eqref{eq:intbypar}, it is obtained:
$$
-\inner[L^2(\Omega, \bbA)]{\bm{v}_1}{\mathcal{L}^*\bm{e}_2} = -\inner[L^2(\Omega, \bbB)]{\mathcal{L}\bm{v}_1}{\bm{e}_2} + \inner[\partial \Omega]{\Gamma_0 \bm{v}_1}{\Gamma_\perp \bm{e}_2}.
$$
The boundary variable $\bm{e}_\partial := \Gamma_\perp \bm{e}_2$ in~\eqref{eq:stdir_bound} explicitly appears. Then the equation defining the corresponding $\bm{f}_\partial := \Gamma_0 \bm{e}_1$ is put into weak form to obtain the final system for all smooth test functions $\bm{v}_1$, $\bm{v}_2$, and $\bm{v}_\partial$:
\begin{equation}\label{eq:weak_intJ1}
\begin{aligned}
\inner[L^2(\Omega, \bbA)]{\bm{v}_1}{\bm{f}_1} &= -\inner[L^2(\Omega, \bbB)]{\mathcal{L}\bm{v}_1}{\bm{e}_2} + \inner[\partial \Omega]{\Gamma_0 \bm{v}_1}{ \bm{e}_\partial}, \\
\inner[L^2(\Omega, \bbB)]{\bm{v}_2}{\bm{f}_2} &= \inner[L^2(\Omega, \bbB)]{\bm{v}_2}{\mathcal{L}\bm{e}_1}, \\
\inner[\partial \Omega]{\bm{v}_\partial}{\bm{f}_\partial} &= -\inner[\partial \Omega]{\bm{v}_\partial}{\Gamma_0\bm{e}_1}.
\end{aligned}
\end{equation}
Now, a Galerkin discretization is introduced. Test, energy and co-energy functions with the same subscript are discretized using the same basis, for all $t\ge0$:
\begin{equation}\label{eq:approx_vfeb}
\begin{aligned}
\bm{\square}_1(t,\bm{x}) &\approx \bm{\square}_1^d(t,\bm{x}) := \sum_{i=1}^{N_1} \bm{\varphi}_1^i(\bm{x}) \square_1^i(t), &\qquad \forall \bm{x} \in \Omega, \\
\bm{\square}_2(t,\bm{x}) &\approx \bm{\square}_2^d(t,\bm{x}) := \sum_{i=1}^{N_2} \bm{\varphi}_2^i(\bm{x}) \square_2^i(t), &\qquad \forall \bm{x} \in \Omega, \\
\bm{\square}_\partial(t,\bm{s}) &\approx \bm{\square}_\partial^d(t,\bm{s}) := \sum_{i=1}^{N_\partial} \bm{\psi}_\partial^i(\bm{s}) \square_\partial^i(t), &\qquad \forall \bm{s} \in \partial\Omega,
\end{aligned}
\end{equation}
where $\square$ stands for $v$, $f$, and $e$ and $\bm{\varphi}_1^i \in H^{\mathcal{L}}$, $\bm{\varphi}_2^i \in L^2(\Omega,\mathbb{B})$, and $\bm{\psi}_\partial^i \in L^2(\partial\Omega,\mathbb{R}^m)$.

\begin{remark}
In general, a discretization in the same basis of either $(\mathbf{f}_1, \mathbf{e}_1)$ and $(\mathbf{f}_2, \mathbf{e}_2)$ or $(\mathbf{f}_1, \mathbf{f}_2)$ and $(\mathbf{e}_1, \mathbf{e}_2)$ (as done in \cite{kotyczka2018}) must be performed. The former is our choice since it directly leads to square mass matrices, while the latter may be more appropriate when dealing for instance with Maxwell's equations for electromagnetics, see~\cite{PayMatHai20} and references therein for details on the difficulties that may then occur.
\end{remark} 

Then plugging the approximations into \eqref{eq:weak_intJ1}, it is computed:
\begin{equation}\label{eq:stdir_findim_1}
\begin{bmatrix}
\mathbf{M}_1 & \mathbf{0} & \mathbf{0} \\
\mathbf{0} & \mathbf{M}_2 & \mathbf{0} \\
\mathbf{0} & \mathbf{0} & \mathbf{M}_\partial
\end{bmatrix}
\begin{pmatrix}
{\mathbf{f}}_{1} \\
{\mathbf{f}}_{2} \\
{\mathbf{f}_\partial}
\end{pmatrix}
= \begin{bmatrix}
\mathbf{0} & - \mathbf{D}_{\mathcal{L}}^\top &  \mathbf{B}_\perp\\
\mathbf{D}_{\mathcal{L}} & \mathbf{0} & \mathbf{0} \\
- \mathbf{B}_\perp^\top & \mathbf{0} & \mathbf{0}
\end{bmatrix} 
\begin{pmatrix}
\mathbf{e}_{1} \\
\mathbf{e}_{2} \\
\mathbf{e}_\partial
\end{pmatrix},
\end{equation}
where vectors $\mathbf{f}_{1}$, $\mathbf{f}_{2}$, $\mathbf{e}_{1}$, $\mathbf{e}_{2}$, $\mathbf{f}_\partial$, and $\mathbf{e}_\partial$ are given by the column-wise concatenation of the respective degrees of freedom of $\bm{f}_{1}^d$, $\bm{f}_{2}^d$, $\bm{e}_{1}^d$, $\bm{e}_{2}^d$, $\bm{f}_\partial^d$, and $\bm{e}_\partial^d$, and where the matrices are defined as follows:
\begin{equation}
\begin{aligned}
M_1^{ij} &= \inner[L^2(\Omega, \mathbb{A})]{\bm{\varphi}_1^i}{\bm{\varphi}_1^j}, \\
M_2^{mn} &= \inner[L^2(\Omega, \mathbb{B})]{\bm{\varphi}_2^m}{\bm{\varphi}_2^n}, \\
M_\partial^{lk} &= \inner[\partial \Omega]{\bm{\psi}_\partial^l}{\bm{\psi}_\partial^k},
\end{aligned}
\quad
\begin{aligned}
{D}_{\mathcal{L}}^{mi} &= \inner[L^2(\Omega, \mathbb{B})]{\bm{\varphi}_2^m}{\mathcal{L}\bm{\varphi}_1^i}, \\
{B}_\perp^{ik} &= \inner[\partial \Omega]{\Gamma_0\bm{\varphi}_1^i}{ \bm{\psi}_\partial^k},
\end{aligned}
\end{equation}
where $1 \le i, j \le N_1$, $1 \le m, n \le N_2$, and $1 \le l, k \le N_\partial$. System~\eqref{eq:stdir_findim_1} is a kernel representation of a Dirac structure as in~\eqref{eq:finitePH-Dirac} (see Remark~\ref{rem:MassMatrices}).

\begin{remark}
Note that matrices $\mathbf{D}_{\mathcal{L}}$ and $\mathbf{B}_\perp$ are not square.
\end{remark}

The discrete Hamiltonian is naturally defined as the continuous one evaluated in the discrete energy variables. As done in Section~\ref{sec:generalFramework}, it is easier to distinguish the linear hyperbolic from the parabolic case.

\subparagraph{Hyperbolic case}

In this setting, the flows $\bm{f}_i$, $i=1, 2$, are given by the time derivative of the energy variables $\bm{\alpha}_i$. Hence, the discretization of these energy variables is given by:
$$
\bm{\alpha}_i^d(t,\bm{x}) = \bm{\alpha}_i^d(0,\bm{x}) + \int_0^t \bm{f}_i^d(s,\bm{x}) {\rm~d} s, \qquad i = 1, 2.
$$
The discrete Hamiltonian is then defined by $H^d(\underline{\alpha}_1,\underline{\alpha}_2) := H(\bm{\alpha}_1^d, \bm{\alpha}_2^d)$, where $\underline{\alpha}_1$ and $\underline{\alpha}_2$ are the column-wise concatenation of the time varying coefficients of $\bm{\alpha}_1^d$ and $\bm{\alpha}_2^d$ in their respective basis.

\begin{definition}
The discretization of the constitutive relations is said to be \emph{compatible} if and only if:
$$
\nabla H^d = \begin{pmatrix} \nabla_{\underline{\alpha}_1} H^d \\ \nabla_{\underline{\alpha}_2} H^d \end{pmatrix} 
= \begin{bmatrix}
\mathbf{M}_1 & \bm{0} \\ \bm{0} & \mathbf{M}_2
\end{bmatrix}
\begin{pmatrix} \mathbf{e}_1 \\ \mathbf{e}_2 \end{pmatrix}.
$$
\end{definition}

\begin{proposition}\label{prop:PBhyperbolic}
If the discretisation of the constitutive relations is compatible, the discrete power balance reads at the discrete level:
$$
\diff{}{t} H^d = - ( \mathbf{e}_\partial )^\top ~ \mathbf{M}_\partial \mathbf{f}_\partial,
$$
which perfectly mimics the continuous identity.
\end{proposition}

\begin{proof}
A straightforward computation gives:
$$
\begin{array}{rl}
\diff{}{t} H^d &= ( \nabla_{\underline{\alpha}_1} H^d )^\top \diff{}{t} \underline{\alpha}_1 + ( \nabla_{\underline{\alpha}_2} H^d )^\top \diff{}{t} \underline{\alpha}_2, \\
&= ( \mathbf{M}_1 \mathbf{e}_1 )^\top ~ \mathbf{f}_1 + ( \mathbf{M}_2 \mathbf{e}_2 )^\top ~ \mathbf{f}_2, \\
&= ( \mathbf{e}_1 )^\top ~ \mathbf{M}_1 \mathbf{f}_1 + ( \mathbf{e}_2 )^\top ~ \mathbf{M}_2 \mathbf{f}_2, \\
&= - ( \mathbf{e}_\partial )^\top ~ \mathbf{M}_\partial \mathbf{f}_\partial,
\end{array}
$$
where the symmetry of the mass matrices and the Dirac structure have been used.
\end{proof}

\begin{remark}
In the special case of linear hyperbolic systems, it has been seen that the \emph{co-energy formulation} allows to take the constitutive relations into account directly in the differential equations. Applying PFEM to~\eqref{eq:pHlinhyp} then leads to an ODE, and the constitutive relations are then automatically discretized in a compatible manner.
\end{remark}

\subparagraph{Parabolic case}

In this setting, only the flow $\bm{f}_1$ is the time derivative of the energy variable $\bm{\alpha}_1$. This energy variable is discretized as in the hyperbolic case. The discrete Hamiltonian is then defined by $H^d(\underline{\alpha}_1) := H(\bm{\alpha}_1^d)$.

\begin{definition}
The discretization of the constitutive relation is said to be \emph{compatible} if and only if:
$
\nabla H^d = \mathbf{M}_1 \mathbf{e}_1.
$
\end{definition}

\begin{proposition}\label{prop:PBparabolic}
If the discretisation of the constitutive relations is compatible, the discrete power balance reads at the discrete level:
$$
\diff{}{t} H^d = - \mathbf{e}_2^\top ~ \mathbf{M}_2 \mathbf{f}_2 - \mathbf{e}_\partial^\top ~ \mathbf{M}_\partial \mathbf{f}_\partial,
$$
which perfectly mimics the continuous identity.
\end{proposition}

\begin{proof}
The proof can be derived similarly as in the hyperbolic case.
\end{proof}

\begin{remark}
Of course, an accurate discretization of the implicit constitutive relation $\mathcal G(\bm{f}_2,\bm{e}_2) = 0$ is also required to conclude. This will be illustrated in Section~\ref{sec:heat}.
\end{remark}

\paragraph{Integration by parts of the term \texorpdfstring{$\inner[L^2(\Omega, \bbB)]{\bm{v}_2}{\mathcal{L}\bm{e}_1}$}{}}
Using~\eqref{eq:intbypar}, it comes:
$$
\inner[L^2(\Omega, \bbB)]{\bm{v}_2}{\mathcal{L}\bm{e}_1} = \inner[L^2(\Omega, \bbA)]{\mathcal{L}^*\bm{v}_2}{\bm{e}_1} + \inner[\partial \Omega]{\Gamma_\perp \bm{v}_2}{\Gamma_0 \bm{e}_1}.
$$
Now the boundary variable $\bm{e}_\partial := \Gamma_0 \bm{e}_1$ explicitly appears, {\it i.e.} the causality considered in~\eqref{eq:stdir_bound_bis}. The weak formulation then reads:
\begin{equation}\label{eq:weak_intJ2}
\begin{aligned}
\inner[L^2(\Omega, \bbA)]{\bm{v}_1}{\bm{f}_1} &= -\inner[L^2(\Omega, \bbA)]{\bm{v}_1}{\mathcal{L}^*\bm{e}_2}, \\
\inner[L^2(\Omega, \bbB)]{\bm{v}_2}{\bm{f}_2} &= \inner[L^2(\Omega, \bbA)]{\mathcal{L}^*\bm{v}_2}{\bm{e}_1} - \inner[\partial \Omega]{\Gamma_\perp \bm{v}_2}{\bm{e}_\partial}, \\
\inner[\partial \Omega]{\bm{v}_\partial}{\bm{f}_\partial} &= \inner[\partial \Omega]{\bm{v}_\partial}{ \Gamma_\perp\bm{e}_2}.
\end{aligned}
\end{equation}
Plugging the approximations~\eqref{eq:approx_vfeb} into~\eqref{eq:weak_intJ2}, this time with $\bm{\varphi}_1^i \in L^2(\Omega,\mathbb{A})$, $\bm{\varphi}_2^i \in H^{-\mathcal{L}^*}$, and $\bm{\psi}_\partial^i \in L^2(\partial\Omega,\mathbb{R}^m)$, gives the following kernel representation of a finite-dimensional Dirac structure:
\begin{equation}\label{eq:stdir_findim_2}
\begin{bmatrix}
\mathbf{M}_1 & \mathbf{0} & \mathbf{0} \\
\mathbf{0} & \mathbf{M}_2 & \mathbf{0} \\
\mathbf{0} & \mathbf{0} & \mathbf{M}_\partial \\
\end{bmatrix}
\begin{pmatrix}
{\mathbf{f}}_{1} \\
{\mathbf{f}}_{2} \\
{\mathbf{f}_\partial} \\
\end{pmatrix}
= \begin{bmatrix}
\mathbf{0} & \mathbf{D}_{-\mathcal{L}^*} & \mathbf{0}\\
-\mathbf{D}_{-\mathcal{L}^*}^\top & \mathbf{0} & \mathbf{B}_0 \\
\mathbf{0} & -\mathbf{B}_0^\top & \mathbf{0}
\end{bmatrix} 
\begin{pmatrix}
\mathbf{e}_{1} \\
\mathbf{e}_{2} \\
\mathbf{e}_\partial \\
\end{pmatrix},
\end{equation}
where the matrices $\mathbf{D}_{-\mathcal{L}^*}$ and $\mathbf{B}_0$ are defined by:
\begin{equation}
{D}_{-\mathcal{L}^*}^{im} = \inner[L^2(\Omega, \mathbb{A})]{\bm{\varphi}_1^i}{\mathcal{-L^*}\bm{\varphi}_2^m}, \quad B_0^{mk} = \inner[\partial\Omega]{\Gamma_\perp\bm{\varphi}_2^m}{\bm{\psi}_{\partial}^k}.
\end{equation}
The power balances proven above still hold true with this causality, where the role played by $\bm{e}_\partial$ and $\bm{f}_\partial$ have been switched. \\

In the sequel, this methodology is applied to the wave equation, the Mindlin-Reissner plate model and the heat equation. These models have been chosen to demonstrate the versatility of our methodology. The wave equation is the prototype of linear hyperbolic systems, and the first example treated by PFEM \cite{cardoso2018pfem}. The Mindlin model combines wave dynamics and plane elastodynamics, and requires the introduction of tensor-valued variables. Finally, the heat equation is the prototype of parabolic systems, and leads to a pHs with intrinsic algebraic constraint, namely, to a pHDAE.

\subsection{The wave equation}\label{sec:wave}

The wave equation is a well-known model, used as the first example of linear hyperbolic systems in many lecture notes and books. This work is no exception to the rule. However, to account for more realistic physics, let us consider the heterogeneous and anisotropic multidimensional wave equation. The equation reads (see~\cite{KurZwa15}):
\begin{equation}
\label{eq:clWave}
\rho \diffp[2]{w}{t} = \mathrm{div}(\bm{\mathcal{T}} \grad(w)) + f, \quad (\bm{x}, t) \in \Omega \times [0, t_f], \qquad \Omega \subset \mathbb{R}^N,
\end{equation}
where $\rho$ is the mass density (bounded from above and below), $\bm{\mathcal{T}}$ is the \emph{tensorial} Young modulus (symmetric and positive definite) and $w$ is the deflection from the equilibrium. The field $f$ accounts for distributed force, such as gravity.

Let us denote $\alpha_p := \rho \diffp{w}{t}$ the linear momentum and $\bm{\alpha}_q := \grad(w)$ the strain, as energy variables. Hence the Hamiltonian is given as the total energy (summing kinetic and potential energies) by:
\begin{equation}
 \label{eq:HamiltonianWave}
 H = \frac{1}{2} \int_\Omega \left\{ \frac{\alpha_p^2}{\rho} + \bm{\alpha}_q \cdot (\bm{\mathcal{T}} \bm{\alpha}_q) \right\} \rm{d} \Omega.
\end{equation}
The co-energy variables are by definition the variational derivatives of $H$ with respect to the energy variables, {\it i.e.}:
\begin{equation}
 \label{eq:constitutiveWave}
e_p := \delta_{\alpha_p} H = \frac{\alpha_p}{\rho} = \diffp{w}{t}, \qquad \bm{e}_q := \delta_{\bm{\alpha}_q} H = \bm{\mathcal{T}} \bm{\alpha}_q = \bm{\mathcal{T}} \grad(w),
\end{equation}
the velocity and stress.
With these notations, equation~\eqref{eq:clWave} rewrites:
\begin{equation}
\label{eq:pHsWave}
\diffp{}{t}
\begin{pmatrix}
\alpha_p \\ \bm{\alpha}_q
\end{pmatrix}
=
\begin{bmatrix}
0 & \div \\ \grad & 0
\end{bmatrix}
\begin{pmatrix}
e_p \\
\bm{e}_q
\end{pmatrix},
\end{equation}
together with the constitutive relations given in~\eqref{eq:constitutiveWave}.

Let us denote:
$$
\bm{e}_1 := e_p, \quad \bm{e}_2 := \bm{e}_q, \quad \mathcal{M}_1 := \rho, \quad \mathcal{M}_2 := \bm{\mathcal{T}}^{-1}, \quad \mathcal{L} := \grad, \quad \Gamma_\perp := \gamma_\perp = \gamma_0 \cdot \bm{n}, \quad \Gamma_0 := \gamma_0,
$$
where $\gamma_0$ is the Dirichlet trace operator. Then the Hamiltonian~\eqref{eq:HamiltonianWave} rewrites as~\eqref{eq:H_coenergy}.  The wave equation~\eqref{eq:clWave} with Neumann boundary control $\bm{u}_\partial = \gamma_\perp (\bm{e}_q)$ is given by~\eqref{eq:pHlinhyp}.

The application of PFEM directly gives:
\begin{equation}
\label{eq:PFEMWave}
\begin{array}{l}
\begin{bmatrix}
\mathbf{M}_\rho & 0 \\ 0 & \mathbf{M}_{\bm{\mathcal{T}}^{-1}}
\end{bmatrix}
\displaystyle\diff{}{t}
\begin{pmatrix}
\mathbf{e}_1 \\ \mathbf{e}_2
\end{pmatrix}
=
\begin{bmatrix}
0 & -\mathbf{D}_{\grad}^\top \\ \mathbf{D}_{\grad} & 0
\end{bmatrix}
\begin{pmatrix}
\mathbf{e}_1 \\
\mathbf{e}_2
\end{pmatrix}
+
\begin{bmatrix}
\mathbf{B}_\perp \\ 0
\end{bmatrix} \mathbf{u}_\partial \vspace{3pt}\\
\mathbf{M}_\partial ~ \mathbf{y}_\partial
=
\begin{bmatrix}
\mathbf{B}_\perp^\top & 0
\end{bmatrix}
\begin{pmatrix}
\mathbf{e}_1 \\
\mathbf{e}_2
\end{pmatrix},
\end{array}
\end{equation}
where:
$$
M_\rho^{ij} := \inner[L^2(\Omega,\mathbb{R})]{\bm{\varphi}_1^i}{\rho ~ \bm{\varphi}_1^j}, \quad 
M_{\bm{\mathcal{T}}^{-1}}^{kl} := \inner[L^2(\Omega,\mathbb{R}^N)]{\bm{\varphi}_2^k}{{\bm{\mathcal{T}}^{-1}} \bm{\varphi}_2^l},
$$
are the discretizations of the operators $\mathcal{M}_1$ and $\mathcal{M}_2$ respectively.

\subsection{The Mindlin plate model}\label{sec:min}
The Mindlin model is a generalization to the 2D case of the Timoshenko beam model and is expressed by a system of two coupled PDEs (see \cite{timoshenko1959theory}):
\begin{equation}
\label{eq:clMin}
\begin{cases}
\displaystyle \rho b \diffp[2]{w}{t} &= \mathrm{div}(\bm{q}) + f, \quad (\bm{x}, t) \in \Omega \times [0, t_f], \qquad \Omega \subset \mathbb{R}^2, \vspace{3pt}\\
\displaystyle \frac{\rho b^3}{12} \diffp[2]{\bm \theta}{t} &= \bm{q} + \mathrm{Div}(\bm M) + \bm{\tau}, \\
\end{cases}
\end{equation}
where $\rho$ is the mass density, $b$ the plate thickness, $w$ the vertical displacement, $\bm \theta = (\theta_x, \theta_y)^\top$ collects the deflection of the cross section along axes $x$ and $y$ respectively. The fields $f, \bm{\tau}$ represent distributed forces and torques. Variables $\bm{M}, \bm{q}$ represent the momenta tensor and the shear stress. Hooke's law relates those to the curvature tensor and shear deformation vector:
\begin{equation*}
\begin{aligned}
\bm{M} &:= \bm{\mathcal{D}}_b \bm{K} \in \mathbb{S}, \\ \bm{q} &:= {D}_s \bm{\gamma},
\end{aligned} \qquad
\begin{aligned}
\bm{K} &:= \mathrm{Grad}(\bm{\theta}) \in \mathbb{S}, \\ \bm{\gamma} &:= \mathrm{grad}(w) - \bm{\theta}.
\end{aligned}
\end{equation*}
${D}_s:=\frac{E_Y b K_{\text{sh}}}{2(1+\nu)}$ is the shear rigidity coefficient, where $E_Y$ is the Young modulus, $\nu$ is the Poisson modulus, $K_{\text{sh}}$ is the shear correction factor. Tensor $\bm{\mathcal{D}}_b$ is the bending stiffness:
\begin{equation}
\label{eq:bend_rig_tensor}
\bm{\mathcal{D}}_b (\cdot) = \frac{E_Y b^3}{12 (1 - \nu^2)}[(1-\nu)(\cdot) + \nu \Tr(\cdot)].
\end{equation}
An appropriate selection of the energy variables is the following \cite{brugnoli2019mindlin,brug:20}:
\begin{equation}
\alpha_w= \rho b\partial_{t} w, \qquad
\bm{\alpha}_{\theta} := \frac{\rho b^3}{12}\partial_{t} \bm{\theta}, \qquad
\bm{A}_{\kappa} := \bm{K}, \qquad
\bm{\alpha}_{\gamma} := \bm{\gamma}. 
\end{equation}
The Hamiltonian $H$ (total energy) is expressed in terms of energy variables as:
\begin{equation}
H = \frac{1}{2} \int_{\Omega} \left\{ \frac{1}{\rho b} \alpha_w^2 + \frac{12}{\rho b^3} \norm{\bm{\alpha}_\theta}^2 + \bm{A}_\kappa \cddot (\bm{\mathcal{D}}_b \bm{A}_\kappa) + D_s \norm{\bm{\alpha}_\gamma}^2\right\} \Omega,
\end{equation} 
where $\bm{A} \cddot \bm{B}$ denotes the tensor contraction. The co-energy variables are defined as follows:
\begin{equation}
e_w := \delta_{\alpha_w} H = \partial_{t} w, \qquad
\bm{e}_{\theta} := \delta_{\bm\alpha_\theta} H =\partial_{t} \bm{\theta}, \qquad
\bm{E}_{\kappa} := \delta_{\bm{A}_\kappa} H =\bm{M}, \qquad
\bm{e}_{\gamma} := \delta_{\bm\alpha_\gamma} H =\bm{q}. 
\end{equation}
System \eqref{eq:clMin} is then expressed in port-Hamiltonian form as \cite{brugnoli2019mindlin} (forces and torques have been omitted for simplicity):
\begin{equation}\label{eq:phlinsys_Min1}
\diffp{}{t}
\begin{pmatrix}
\alpha_w \\
\bm{\alpha}_\theta \\
\bm{A}_\kappa \\
\bm{\alpha}_{\gamma} \\
\end{pmatrix} = 
\begin{bmatrix}
	0 & 0 & 0 & \div \\
	\bm{0} & \bm{0} & \Div & \bm{I}_{2 \times 2}\\
	\bm{0} & \Grad & \bm{0} & \bm{0}\\
	\grad & -\bm{I}_{2 \times 2} & \bm{0} & \bm{0} \\
	\end{bmatrix}
\begin{pmatrix}
e_w \\
\bm{e}_{\theta} \\
\bm{E}_{\kappa} \\
\bm{e}_{\gamma} \\
\end{pmatrix}.
\end{equation}
By applying the divergence theorem, the energy rate is expressed as the duality product of the boundary variables:
\begin{equation}
\begin{aligned}
\diff{H}{t} &= \inner[L^2(\Omega, \bbR)]{\partial_t \alpha_w}{e_w} + \inner[L^2(\Omega, \bbR^2)]{\partial_t \bm\alpha_\theta}{\bm{e}_\theta} +  \inner[L^2(\Omega, \mathbb{S})]{\partial_t \bm{A}_\kappa}{\bm{E}_\kappa} +  \inner[L^2(\Omega, \bbR^2)]{\partial_t \bm\alpha_\gamma}{\bm{e}_\gamma}, \\
&= \inner[\partial\Omega]{\gamma_0 e_w}{\gamma_\perp \bm{e}_\gamma} + \inner[\partial\Omega]{\bm\gamma_0 \bm{e}_\theta}{\bm\gamma_\perp \bm{E}_\kappa}
= \inner[\partial\Omega]{\bm{y}_{\partial, 1}}{\bm{u}_{\partial, 1}} + \inner[\partial\Omega]{\bm{y}_{\partial, 2}}{\bm{u}_{\partial, 2}},
\end{aligned}
\end{equation}
where:
\begin{equation*}
\begin{aligned}
    \bm{u}_{\partial, 1} &= \gamma_\perp \bm{e}_\gamma, \\
    \bm{u}_{\partial, 2} &= \bm\gamma_\perp \bm{E}_\kappa,
\end{aligned} \qquad 
\begin{aligned}
    \bm{y}_{\partial, 1} &= \gamma_0 e_w, \\
    \bm{y}_{\partial, 2} &= \bm\gamma_0 \bm{e}_\theta.
\end{aligned}
\end{equation*}
The traces $\bm\gamma_0 \bm{u} = \bm{u}\vert_{\partial\Omega}, \; \bm\gamma_\perp \bm{U} = \bm{U} \cdot \bm{n}\vert_{\partial\Omega}$ correspond to the Dirichlet trace for $\bbR^d$ vectors and to the normal trace for $\bbR^{d\times d}$ tensors. The mass operators are given by:
\begin{equation}
\mathcal{M}_1 = \begin{bmatrix}
\rho b & 0 \\
\bm{0} & \frac{\rho b^3}{12}
\end{bmatrix}, \qquad 
\mathcal{M}_2 = \begin{bmatrix}
\bm{\mathcal{D}}_b^{-1} & \bm{0} \\
\bm{0} & D_s^{-1}
\end{bmatrix}.
\end{equation} 
The $\mathcal{L}$, $\Gamma_0$, and $\Gamma_\perp$ operators are:
\begin{equation}
\mathcal{L} = \begin{bmatrix}
\bm{0} & \Grad\\
\grad & -\bm{I}_{2 \times 2}\\
\end{bmatrix}, \qquad
\Gamma_0 = \begin{bmatrix}
\gamma_0 & 0 \\ 
0 & \bm\gamma_0 \\
\end{bmatrix}, \qquad 
\Gamma_\perp = \begin{bmatrix}
0 & \gamma_{\perp} \\
\bm\gamma_{\perp} & 0 \\
\end{bmatrix}.
\end{equation}
Introducing the approximations for the test and co-energy variables:
\begin{equation}{}
\triangle_w = \sum_{i = 1}^{N_w} \phi_w^i \triangle_w^i, \qquad
\bm{\triangle}_\theta = \sum_{i = 1}^{N_\theta} \bm\phi_\theta^i \triangle_\theta^i, \qquad
\text{\Large$\bm{\triangle}$}_\kappa = \sum_{i = 1}^{N_\kappa} \bm\Phi_\kappa^i \triangle_\kappa^i, \qquad
\bm{\triangle}_\gamma = \sum_{i = 1}^{N_\gamma} \bm\phi_\gamma^i \triangle_\gamma^i, 
\end{equation} 
where $\triangle = \{v, \, e\}$, and for the boundary controls:
\begin{equation}
{\square}_{\partial, 1} = \sum_{i = 1}^{N_{\partial, 1}} \psi_{\partial, 1}^i \square_{\partial, 1}^i, \qquad \bm{\square}_{\partial, 2} = \sum_{i = 1}^{N_{\partial, 2}} \bm\psi_{\partial, 2}^i \square_{\partial, 2}^i, \qquad \square = \{v, \, u, \, y\}, \quad \psi_{\partial, 1}^i \in \bbR, \; \bm\psi_{\partial, 2}^i \in \bbR^2,
\end{equation}{}
PFEM can be applied to obtain:
\begin{equation}\label{eq:findim_min}
\begin{aligned}
\mathrm{Diag}\begin{bmatrix}
\mathbf{M}_{\rho h}\\
\mathbf{M}_{I_\theta}\\
\mathbf{M}_{\bm{\mathcal{D}}_b^{-1}}\\
\mathbf{M}_{{D}_s^{-1}} \\
\end{bmatrix}
\begin{pmatrix}
\dot{\mathbf{e}}_{w} \\
\dot{\mathbf{e}}_{\theta} \\
\dot{\mathbf{e}}_{\kappa} \\
\dot{\mathbf{e}}_{\gamma} \\
\end{pmatrix}
&= \begin{bmatrix}
\mathbf{0} & \mathbf{0} & \mathbf{0} & -\mathbf{D}_{\grad}^\top \\
\mathbf{0} & \mathbf{0} & -\mathbf{D}_{\Grad}^\top & -\mathbf{D}_{0}^\top \\
\mathbf{0} & \mathbf{D}_{\Grad} & \mathbf{0} & \mathbf{0}\\
\mathbf{D}_{\grad} & \mathbf{D}_{0} & \mathbf{0} & \mathbf{0}\\
\end{bmatrix} 
\begin{pmatrix}
{\mathbf{e}}_{w} \\
{\mathbf{e}}_{\theta} \\
{\mathbf{e}}_{\kappa} \\
{\mathbf{e}}_{\gamma} \\
\end{pmatrix} + 
\begin{bmatrix}
\mathbf{B}_{\perp, w} & \mathbf{0} \\
\mathbf{0} & \mathbf{B}_{\perp, \theta}\\
\mathbf{0} & \mathbf{0} \\
\mathbf{0} & \mathbf{0} \\
\end{bmatrix}
\begin{pmatrix}
\mathbf{u}_{\partial, 1} \\
\mathbf{u}_{\partial, 2} \\
\end{pmatrix},
\\
\mathrm{Diag}\begin{bmatrix}
\mathbf{M}_{\partial, 1} \\
\mathbf{M}_{\partial, 2} \\
\end{bmatrix}{}
\begin{pmatrix}
\mathbf{y}_{\partial, 1} \\
\mathbf{y}_{\partial, 2} \\
\end{pmatrix} &= \begin{bmatrix}
\mathbf{B}_{\perp, w}^\top & \mathbf{0} & \mathbf{0} & \mathbf{0} \\
\mathbf{0} & \mathbf{B}_{\perp, \theta}^\top & \mathbf{0} & \mathbf{0}\\ 
\end{bmatrix}
\begin{pmatrix}
{\mathbf{e}}_{w} \\
{\mathbf{e}}_{\theta} \\
{\mathbf{e}}_{\kappa} \\
{\mathbf{e}}_{\gamma} \\
\end{pmatrix}.
\end{aligned}
\end{equation}
The notation $\mathrm{Diag}$ denotes a block diagonal matrix. The mass matrices $\mathbf{M}_{\rho h}, \; \mathbf{M}_{I_\theta}, \; \mathbf{M}_{\bm{\mathcal{C}}_b}, \; \mathbf{M}_{{C}_s}$ are computed as:
\begin{equation}
\begin{aligned}
M_{\rho h}^{ij} &= \inner[L^2(\Omega)]{\phi_w^i}{\rho h \phi_w^j}, \\
M_{I_\theta}^{mn} &= \inner[L^2(\Omega, \mathbb{R}^{2})]{\phi_\kappa^m}{I_\theta \phi_\kappa^n},
\end{aligned} \qquad
\begin{aligned}
\quad M_{\bm{\mathcal{D}}_b^{-1}}^{pq} &= \inner[L^2(\Omega, \mathbb{R}^{2\times 2}_{\text{sym}})]{\bm\Phi_\kappa^p}{\bm{\mathcal{D}}_b^{-1} \bm\Phi_\kappa^q}, \\
M_{D_s^{-1}}^{rs} &= \inner[L^2(\Omega, \mathbb{R}^{2})]{\bm\phi_\gamma^r}{D_s^{-1} \bm\phi_\gamma^s}, 
\end{aligned}
\end{equation}
where $i, j \in \{1, N_w\}, \; m, n \in \{1, N_\theta\}, \, p, q \in \{1, N_\kappa\}, \; r, s \in \{1, N_\gamma\}$. Matrices $\mathbf{D}_{\grad}, \; \mathbf{D}_{\Grad}, \; \mathbf{D}_{0}$ assume the form:
\begin{equation}
\begin{aligned}
D_{\grad}^{rj} &= \inner[L^2(\Omega, \mathbb{R}^2)]{\bm\phi_\gamma^r}{\grad \phi_w^j}, \\ D_{\Grad}^{pn} &= \inner[L^2(\Omega, \mathbb{R}^{2\times 2}_{\text{sym}})]{\bm\Phi_\kappa^p}{\Grad \bm\phi_\theta^n},
\end{aligned} \quad
D_{0}^{rn} = -\inner[L^2(\Omega, \mathbb{R}^2)]{\bm\phi_\gamma^r}{\bm\phi_\theta^n}.
\end{equation}
Matrices $\mathbf{B}_w, \; \mathbf{B}_{\bm\theta}, \; \mathbf{M}_{\partial, 1}, \; \mathbf{M}_{\partial, 2}$ are computed as:
\begin{equation}
\begin{aligned}{}
\mathbf{B}_{\perp, w}^{ig} &= \inner[\partial\Omega]{\gamma_0 \phi_w^i}{\psi_{\partial, 1}^g}, \\
\mathbf{B}_{\perp, \theta}^{mg} &= \inner[\partial\Omega]{\bm\gamma_0 \bm\phi_\theta^m}{\bm\psi_{\partial, 2}^g}, \\
\end{aligned} \qquad
\begin{aligned}{}
M_{\partial, 1}^{fg} &= \inner[L^2(\partial \Omega, \mathbb{R}^m)]{\psi_{\partial, 1}^f}{\psi_{\partial, 1}^g}, \\
M_\partial^{lk} &= \inner[L^2(\partial \Omega, \mathbb{R}^m)]{\bm\psi_{\partial, 2}^l}{\bm\psi_{\partial, 2}^k}, \\ 
 \end{aligned} \qquad
\begin{aligned}{}
f, g &\in \{1, N_{\partial, 1}\}, \\
l, k &\in \{1, N_{\partial, 2}\}
\end{aligned}
\end{equation}
The discrete Hamiltonian is then computed as:
\begin{equation}
    H_d = \frac{1}{2}\mathbf{e}_w^\top \mathbf{M}_{\rho h}{\mathbf{e}}_{w} + \frac{1}{2}{\mathbf{e}}_{\theta}^\top \mathbf{M}_{I_\theta} {\mathbf{e}}_{\theta} + \frac{1}{2} {\mathbf{e}}_{\kappa}^\top \mathbf{M}_{\bm{\mathcal{D}}_b^{-1}} \mathbf{e}_{\kappa} + \frac{1}{2} {\mathbf{e}}_{\gamma}^\top \mathbf{M}_{{D}_s^{-1}} {\mathbf{e}}_{\gamma}.
\end{equation}
From system \eqref{eq:findim_min} the discrete energy rate is readily obtained:
\begin{equation}
    \diff{H_d}{t} = \mathbf{y}_{\partial, 1}^\top \mathbf{M}_{\partial, 1} \mathbf{u}_{\partial, 1} + \mathbf{y}_{\partial, 2}^\top \mathbf{M}_{\partial, 2} \mathbf{u}_{\partial, 2}.
\end{equation}
The discrete energy rate then mimics its infinite dimensional counterpart.

\begin{remark}
 Equivalently a purely mixed formulation can be obtained by integrating by parts the third and fourth lines of \eqref{eq:phlinsys_Min1}. In this case, the system of equations gathers together a plane elasticity problem \cite{arnold2014elastodynamics} and a wave equation in mixed form. Conforming finite elements for the plane elasticity system on simplicial meshes have been constructed in \cite{arnold2002mixed}. The simpler PEERS elements based on a weak symmetry formulation have been proposed in \cite{arnold1984peers}. The PEERS elements have been used in \cite{veiga2013} to construct a stable locking-free mixed formulation for the static Mindlin problem.
\end{remark}

\subsection{The heat equation}\label{sec:heat}
The heat equation is the simplest example of parabolic system. Instead of rewriting the well-known PDE under a pHs, a direct pHs modelling is presented, as done in~\cite{SerMatHai19b,SerMatHai19c}. The model is constructed in order to keep apart thermodynamical principles from equations of state. Indeed, the pHs formalism allows to modify the latter, by keeping the structure of the former.

Let $\Omega \subset \mathbb{R}^N$ be a bounded open connected set. Assume that this domain models a rigid body: its volume does not change over time and no chemical reaction is to be found. Let us denote: $\rho$ the mass density, $u$ the internal energy density, $\bm{J}_Q$ the heat flux, $T$ the local temperature, $\beta := \dfrac{1}{T}$ the reciprocal temperature, $s$ the entropy density, $\bm{J}_S := \beta \bm{J}_Q$ the entropy flux, $C_V := \left(\diff{u}{T}\right)_V$ the isochoric heat capacity.

The first law of thermodynamic reads:
\begin{equation}\label{eq:energy}
\rho \diffp{u}{t} = - \div ( \bm{J}_Q ).
\end{equation}
Under the hypothesis of an inert rigid solid, Gibbs formula reads ${\rm d} u = T {\rm d} s$, giving:
\begin{equation}\label{eq:Gibbs}
\diffp{u}{t} = T \diffp{s}{t}.
\end{equation}
Defining $\sigma := \grad \left( \beta \right) \bm{J}_Q$, and seeing $u$ as a function of the entropy density $s$, Gibbs formula~\eqref{eq:Gibbs} gives:
\begin{equation}\label{eq:entropy}
\rho \diffp{s}{t} = - \div ( \bm{J}_S ) + \sigma.
\end{equation}
Then $\sigma$ is the irreversible entropy production.

In this work, the following constitutive equations of state will be assumed:
\begin{itemize}
\item
The rigid body is at room temperature: the Dulong-Petit model is supposed to be satisfied, {\it i.e.} $u = C_V T$, with time-invariant $C_V$;
\item
The thermal conduction is given by Fourier's law, with a symmetric positive tensor $\bm{\lambda}$: $\bm{J}_Q = - \bm{\lambda} \grad(T)$.
\end{itemize}
Thanks to~\eqref{eq:energy} and the equations of state, we easily recover the classical PDE for the temperature $T$: $\rho C_V \diffp{T}{t} = \div( \bm{\lambda} \grad(T) )$. 

The ``$L^2$-energy'' $\left( \int_\Omega \rho C_V T^2 \rm{d} \Omega \right)^\frac{1}{2}$ is commonly used as Hamiltonian for the heat equation. However, it lacks of a thermodynamical meaning. The internal energy would be more accurate for this physical problem, even though it will rise some difficulties. Nevertheless, the pHs formalism allows dealing with it, and PFEM proves to be powerful enough to discretize the system in a structure-preserving manner even for this choice of Hamiltonian.

Let the internal energy be seen as a functional of the local entropy as energy variable: $\alpha_s := \rho s$, then:
$$
H := \int_\Omega \rho u(\alpha_s) \rm{d} \Omega.
$$
The co-energy variable is given by $e_s := \delta_{\alpha_s} H = \frac{{\rm d} \rho u}{{\rm d} \rho s} = T$, the local temperature. Denoting $\bm{e}_S := \bm{J}_S$, one can introduce a new flow variable $\bm{f}_S$ such that:
$$
\begin{pmatrix}
\diffp{\alpha_s}{t} \\ \bm{f}_S
\end{pmatrix}
=
\begin{bmatrix}
0 & -\div \\ -\grad & 0
\end{bmatrix}
\begin{pmatrix}
e_s \\
\bm{e}_S
\end{pmatrix}
+
\begin{pmatrix}
\sigma \\ 0
\end{pmatrix}.
$$
Obviously, $\bm{f}_S = -\grad(e_s)$. In order to get a formally skew-symmetric operator, let us also introduce an \emph{entropy} port $(f_\sigma, e_\sigma)$, such that $e_\sigma = - \sigma$. Then:
\begin{equation}
\label{eq:pHsHeat}
\begin{pmatrix}
\diffp{\alpha_s}{t} \\ \bm{f}_S \\ f_\sigma
\end{pmatrix}
=
\begin{bmatrix}
0 & -\div & -1 \\ -\grad & 0 & 0 \\ 1 & 0 & 0
\end{bmatrix}
\begin{pmatrix}
e_s \\
\bm{e}_S \\
e_\sigma
\end{pmatrix}.
\end{equation}


\begin{remark}\label{rem:Fourier}
As surprising as it can be, in this setting, Fourier's law appears to be stated in a \emph{nonlinear} way: $e_s \bm{e}_S - \bm{\lambda} \bm{f}_S = 0$. This comes from the necessity to express the constitutive relations in function of the flows and efforts appearing in the equation defining the Stokes-Dirac structure.
\end{remark}

\begin{remark}\label{rem:ConstRelEntropy}
Two variables have been added to obtain~\eqref{eq:pHsHeat}, but only one equation naturally appears: $f_\sigma = e_s$. Thus, another equation is needed to close the system: $\mathcal{G}(\bm{f}_2,\bm{e}_2) = 0$. Here, it is given by the definition of the irreversible entropy production $\sigma := \grad(\beta) \bm{J}_Q$, rewritten in the flows and efforts variables. This leads to the \emph{nonlinear} constitutive relation: $\bm{f}_S \bm{e}_S + f_\sigma e_\sigma = 0$.
\end{remark}

\begin{remark}
Usually the system energy is taken to be $\left( \int_\Omega \rho C_V T^2 \rm{d} \Omega \right)^\frac{1}{2}$, that gives rise to the well-known linear diffusive system. At first glance, our approach may be surprising since it leads to a lossless nonlinear differential-algebraic system. There is indeed a major advantage in doing so, in view of the modelling and discretization of complex systems by interconnection of several pHs. For instance, if one wants to interconnect both thermal and mechanical processes, for the energy exchanges to be consistent, physics must be coherent.
When dissipation occurs, through \textit{e.g.} friction or viscosity, the kinetic energy is converted into internal energy. Hence, for a physically meaningful system, the internal energy is indeed the one to consider.
\end{remark}

Let us define:
$$
\bm{f}_1 := \diffp{\alpha_s}{t}, \quad \bm{f}_2 := \begin{pmatrix} \bm{f}_S \\ f_\sigma \end{pmatrix}, \quad \bm{e}_1 := e_s, \quad \bm{e}_2 := \begin{pmatrix} \bm{e}_S \\ e_\sigma \end{pmatrix},
$$
and:
$$
\mathcal{L} := \begin{pmatrix} -\grad \\ 1 \end{pmatrix}, \quad \Gamma_\perp := \begin{pmatrix} -\gamma_\perp & {0} \end{pmatrix}, \quad \Gamma_0 := \gamma_0.
$$
Then, the heat equation~\eqref{eq:pHsHeat} with boundary control $\bm{e}_\partial := \bm{u}_\partial = \Gamma_0 \bm{e}_1 = \gamma_0(e_s)$ and boundary observation $\bm{f}_\partial := -\bm{y}_\partial = \Gamma_\perp \bm{e}_2 = \gamma_\perp(\bm{e}_S)$ rewrites under the form~\eqref{eq:stdir_bound_bis}. Thus PFEM will be applied with an \emph{integration by parts} on the second line in this strategy, leading to~\eqref{eq:stdir_findim_2}, which rewrites with the current variables:
\begin{equation}\label{eq:PFEMheat}
\begin{bmatrix}
\mathbf{M}_s & \bm{0} & \bm{0} & \bm{0} \\
\bm{0} & \mathbf{M}_S & \bm{0} & \bm{0} \\
\bm{0} & \bm{0} & \mathbf{M}_\sigma & \bm{0} \\
\bm{0} & \bm{0} & \bm{0} & \mathbf{M}_\partial
\end{bmatrix}
\begin{pmatrix}
\diff{}{t}\underline{\alpha}_s \\ \mathbf{f}_S \\ \mathbf{f}_\sigma \\ -\mathbf{y}_\partial
\end{pmatrix}
=
\begin{bmatrix}
\bm{0} & \mathbf{D}_{-\div} & -\mathbf{M}_\sigma & \bm{0} \\
-\mathbf{D}_{-\div}^\top & \bm{0} & \bm{0} & \mathbf{B}_0 \\
\mathbf{M}_\sigma & \bm{0} & \bm{0} & \bm{0} \\
\bm{0} & -\mathbf{B}_0^\top & \bm{0} & \bm{0}
\end{bmatrix}
\begin{pmatrix}
\mathbf{e}_s \\ \mathbf{e}_S \\ \mathbf{e}_\sigma \\ \mathbf{u}_\partial
\end{pmatrix},
\end{equation}
where:
$$
\begin{array}{rcl c rcl}
M_s^{ij} &:=& \inner[L^2(\Omega,\mathbb{R})]{\bm{\varphi}_1^i}{\bm{\varphi}_1^j}, 
& M_S^{kl} &:=& \inner[L^2(\Omega,\mathbb{R}^N)]{\bm{\varphi}_S^k}{\bm{\varphi}_S^l}, \\
M_\sigma^{{\widetilde k}{\widetilde l}} &:=& \inner[L^2(\Omega,\mathbb{R}^N)]{\bm{\varphi}_\sigma^{\widetilde k}}{\bm{\varphi}_\sigma^{\widetilde l}}, 
& B_0^{km} &:=& -\inner[L^2(\partial\Omega, \mathbb R)]{\Gamma_\perp\bm{\varphi}_S^k}{\bm{\psi}_{\partial}^m},\\
\end{array}
$$
with $\bm{\varphi}_2^k := \begin{pmatrix}
\bm{\varphi}_S^k \\ \bm{0}
\end{pmatrix}$ if $1 \le k \le N_S$, and $\bm{\varphi}_2^k := \begin{pmatrix}
\bm{0} \\ \bm{\varphi}_\sigma^{k-N_S}
\end{pmatrix}$ if $N_S + 1 \le k \le N_S + N_\sigma$.

To be compatible, the discretizations of the constitutive relations are given as follows:
\begin{itemize}
\item
Dulong-Petit model reads:
$
\mathbf{M}_s \underline{\alpha}_s = \mathbf{M}_{\rho C_V} \mathbf{e}_s,
$
with $M_{\rho C_V}^{ij} := \inner[L^2(\Omega,\mathbb{R})]{\bm{\varphi}_1^i}{\rho ~ C_V ~ \bm{\varphi}_1^j}$;
\item
Following Remark~\ref{rem:Fourier}, Fourier's law reads:
$
\mathbf{\Lambda} ~ \mathbf{f}_S = \mathbf{M}_{e_s} \mathbf{e}_S,
$

with $\bm{\Lambda}^{ij} := \inner[L^2(\Omega,\mathbb{R}^N)]{\bm{\varphi}_S^i}{\bm{\lambda} ~ \bm{\varphi}_S^j}$ and $M_{e_s}^{ij} := \inner[L^2(\Omega,\mathbb{R}^N)]{\bm{\varphi}_S^i}{e_s ~ \bm{\varphi}_S^j}$;
\item
The constitutive law coming from the introduction of the irreversible entropy production, as explained in Remark~\ref{rem:ConstRelEntropy}, is taken into account by:
\begin{equation}\label{eq:discreteEntropyProductionCR}
(\mathbf{e}_S)^\top ~ \mathbf{M}_S \mathbf{f}_S + (\mathbf{e}_\sigma)^\top ~ \mathbf{M}_\sigma \mathbf{f}_\sigma = 0.
\end{equation}
\end{itemize}

\begin{remark}
In Fourier's law, the mass matrix $\mathbf{M}_{e_s}$ depends on the co-energy variable $e_s$. This will rise difficulties for the numerical solution in time.
\end{remark}

To conclude, the structure-preserving property can be appreciated in the following result.
\begin{proposition}
Let $H^d(\underline{\alpha}_s) := H(\alpha_s^d)$ be the discrete Hamiltonian, where $\alpha_s^d$ is the discretization of the energy variable in the basis $\bm{\varphi}_1$. It holds:
$$
\diff{}{t}H^d = \mathbf{u}_\partial^\top ~ \mathbf{M}_\partial \mathbf{y}_\partial,
$$
that is the first law of thermodynamics at the discrete level.
\end{proposition}

\begin{proof}
Thanks to the compatible discretization of the Dulong-Petit model, Proposition~\ref{prop:PBparabolic} gives:
$$
\diff{}{t} H^d = - \mathbf{e}_2^\top ~ \mathbf{M}_2 \mathbf{f}_2 - \mathbf{e}_\partial^\top ~ \mathbf{M}_\partial \mathbf{f}_\partial.
$$
By definition of $\mathbf{f}_2$, $\mathbf{e}_2$, $\mathbf{M}_2$, $\mathbf{e}_\partial$, and $\mathbf{f}_\partial$ one computes:
$$
\begin{array}{rl}
\diff{}{t} H^d &= - \mathbf{e}_2^\top ~ \mathbf{M}_2 \mathbf{f}_2 - \mathbf{e}_\partial^\top ~ \mathbf{M}_\partial \mathbf{f}_\partial, \\
&= - \begin{pmatrix} \mathbf{e}_S \\ \mathbf{e}_\sigma \end{pmatrix}^\top ~ 
\begin{bmatrix}
\mathbf{M}_S & \bm{0} \\
\bm{0} & \mathbf{M}_\sigma 
\end{bmatrix} \begin{pmatrix} \mathbf{f}_S \\ \mathbf{f}_\sigma \end{pmatrix} + \mathbf{u}_\partial^\top ~ \mathbf{M}_\partial \mathbf{y}_\partial, \\
&= \mathbf{u}_\partial^\top ~ \mathbf{M}_\partial \mathbf{y}_\partial,
\end{array}
$$
thanks to the constitutive relation~\eqref{eq:discreteEntropyProductionCR} coming from the irreversible entropy production.
\end{proof}

\begin{remark}
Fourier's law does not contribute to the power balance of the internal energy. Nevertheless, such a constitutive relation is needed for the problem to be well-defined.
\end{remark}


\begin{remark}
The methodology detailed so far is certainly not limited to the previous three examples. Indeed higher-order differential~\cite{brugnoli2019kirchhoff}, $\text{curl}$ operator for Maxwell's equations~\cite{PayMatHai20}, nonlinear system~\cite{CarMatLef19}, and different Hamiltonian choices can be handled as well. For instance, in the case of the heat equation, the entropy or the classical $L^2$-norm of the temperature can be alternatively considered as Hamiltonian functional~\cite{SerMatHai19b,SerMatHai19c}. In addition, mixed boundary conditions can be incorporated either by introducing Lagrange multipliers or by employing a virtual domain decomposition method \cite{brugnoli2020wc}. As already mentioned, dissipation (both in the domain and on the boundary) can also be considered in this strategy~\cite{SerMatHai19d,SerMatHai19a}. Hence a very wide class of (nonlinear) multiphysics systems can be discretized in a structure-preserving manner (with well-represented exchanges of energy between the subsystems).
\end{remark}

In the next section we present an ongoing project which has been initiated to prove the efficiency of the PFEM methodology, leveraging well-established and robust software tools for the finite element discretization of partial differential equations and time integration.

\section{SCRIMP: Simulation and ContRol of Interactions in Multi-Physics \label{sec:scrimp}}

In this section the main features related to the numerical simulation of pHs in the framework of the ongoing project named {\sc{SCRIMP}} (Simulation and ContRol of Interactions in Multi-Physics) are detailed. The aim is to provide a flexible prototype {\sc{Python}} code for the numerical simulation of pHs both for research and educational purposes. In addition to numerical experiments proposed later in Section \ref{sec:simus}, the reader is referred to interactive companion {\sc{Jupyter}} notebooks \cite{bhsv:20z} to learn how to numerically solve the model problems introduced in Section \ref{sec:pfem} with {\sc{SCRIMP}}. In the following, the key ideas behind {\sc{SCRIMP}} are mentioned and then a specific emphasis on both space and time discretizations is given. 

\subsection{\label{sec:scrimp:key} Key ideas behind {\sc{SCRIMP}}}

In short, the key ideas related to the design of {\sc{SCRIMP}} are provided:
\begin{itemize}
\item The {\sc{Python}} dynamic programming language has been selected due to its expressiveness and the availability of high-level interfaces to scientific computing software libraries \cite{lila:20}; 
\item {\sc{SCRIMP}} assumes to rely on open-source, external software for the finite dimensional discretization of partial differential equations;
\item {\sc{SCRIMP}} encapsulates the finite dimensional objects related to the finite element discretization in space (e.g. matrices) to deduce the resulting linear or nonlinear pHs in a generic pHODE/pHDAE form as proposed in \cite{beattie2018linear};
\item For multiphysics problems, this design offers the advantage that discretization in space may be handled by different software components depending on the discipline or on the modelling. The modularity and the object-oriented nature of {\sc{Python}} thus offer the flexibility to easily combine the different pHs to deduce the global interconnected system. This is much in line with the mathematical theory of pHs \cite{SchJel14}. Furthermore we note that interconnections of different systems (with e.g. the transformer or gyrator transformations  \cite{SchJel14}) can be easily incorporated. 
\end{itemize}
The design of {\sc{SCRIMP}} is based on procedural and object-oriented paradigms and thus follows the standard ideas governing most of the numerical PDE software. Whereas a detailed exposition of the design patterns of {\sc{SCRIMP}} and its performance will be published elsewhere, concrete illustrations of most of these key ideas can be found in the companion {\sc{Jupyter}} notebooks \cite{bhsv:20z}. The description of the current numerical methods related to space and time discretizations available in {\sc{SCRIMP}} is given.

\subsection{\label{sec:scrimp:space} Semi-discretization in space}

As outlined in Section \ref{sec:pfem}, PFEM relies on an abstract variational formulation written in appropriate finite element spaces. 

To perform the semi-discretization in space, we rely on {\sc{FEniCS}} \cite{AlnaesBlechta2015a}, an open-source {\sc{C++}} scientific software library that provides a high-level {\sc{Python}} interface. The {\sc{FEniCS}} Project is mainly based on a collection of software components targeting the automated solution of partial differential equations via the finite element method. Its core components notably include the Unified Form Language ({\sc{UFL}}) \cite{AlnaesEtAl2012}, the {\sc{FEniCS}} Form Compiler ({\sc{FFC}}) \cite{KirbyLogg2007a} and the finite element library {\sc{DOLFIN}} \cite{LoggWells2010a}, which contains various types of conforming finite element methods, e.g., nodal Lagrangian finite elements for grad-conforming approximations or non non-nodal finite elements (e.g., Raviart-Thomas spaces for div-conforming approximations) 
as well. These families of finite elements are notably required to tackle the discretization in space of our core problems.

A key point to facilitate the generic implementation of PFEM is the use of {\sc{UFL}}. {\sc{UFL}} is indeed an expressive domain-specific language for abstractly representing (finite element) variational formulations of differential equations. In particular, this language defines a syntax for the integration of variational forms over various domains. This simply leads to an expressive implementation that is close to the abstract mathematical formulations presented in Section \ref{sec:pfem}. The {\sc{FEniCS}} Form Compiler {\sc{FFC}} then generates specialized {\sc{C++}} code from the symbolic {\sc{UFL}} representation of variational forms and finite element spaces. The combination of these core elements makes {\sc{FEniCS}} a versatile and efficient software for the finite element approximation of partial differential equations as outlined in \cite{LoggMardalEtAl2012a}. Additionally, {\sc{FEniCS}} also provides an interface for state-of-the-art linear solvers and preconditioners from freely available third-party libraries such as {\sc{PETSc}} \cite{petsc-user-ref}. 
This last feature may be especially useful to handle the numerical simulation of large-scale pHs.

\subsection{\label{sec:scrimp:time} Time integration methods}

As outlined in Section \ref{sec:pfem}, the  semi-discretization in space of the resulting pHs leads to systems of either ordinary differential equations (ODE) or differential algebraic equations (DAE). Hence reliable and accurate time integration methods must be provided.

To offer a large panel of numerical methods, a high-level interface to well-established time integration libraries is provided in {\sc{SCRIMP}}. Concerning the numerical solution of ODEs, we provide light interfaces to the {\sc{Assimulo}} library \cite{Andersson2015} and to the {\sc{SciPy}} time integration method scipy.integrate.solve\_ivp\footnote{\scriptsize https://docs.scipy.org/doc/scipy/reference/generated/scipy.integrate.solve\_ivp.html\#scipy.integrate.solve\_ivp} that both include standard multistep and one-step methods for stiff and non-stiff ordinary differential equations given in explicit form $y^{'} = f(t, y)$ with $y(t_0) = y_0$ where $t_0$ and $y_0$ denote the initial time and initial condition, respectively. This formulation requires the solution of linear systems of equations involving sparse finite element mass matrices. State-of-the-art sparse direct solvers based on Gaussian factorization are used for that purpose. Through {\sc{Assimulo}}, the popular {\sc{CVODE}}\footnote{\scriptsize https://computing.llnl.gov/sites/default/files/public/cv\_guide.pdf} solver from {\sc{Sundials}} \cite{hindmarsh2005sundials} is also accessible. For nonstiff problems, {\sc{CVODE}} relies on the Adams-Moulton formulas, with the order varying between $1$ and $12$. For stiff problems, {\sc{CVODE}} includes schemes based on Backward Differentiation Formulas (BDFs) with order varying between $1$ and $5$. In addition, we also propose standalone implementations of symplectic time integration methods such as the second order accurate Stormer-Verlet method.

The interface to {\sc{Assimulo}} also allows one to handle the numerical solution of linear DAEs through the use of the {\sc{Sundials}} {\sc{IDA}} solver\footnote{https://computing.llnl.gov/projects/sundials/ida}. {\sc{IDA}} is a package for the solution of differential algebraic equation systems written in the form $F(t, y, y^{'})=0$ with $y(t_0) = y_0$. The integration method in {\sc{IDA}} is based on variable-order, variable-coefficient BDF in fixed-leading-coefficient form, where the method order varies between $1$ and $5$. We note that setting the initial conditions properly is of utmost importance for a DAE solver. To do so, we rely on the {\sc{IDA\_YA\_YDP\_INIT}} method to find consistent initial conditions for the time integration. We refer the reader to \cite[Section 2.3]{hindmarsh2005sundials} for additional details on {\sc{IDA}}. As standalone methods, we have considered the second-order accurate Stormer-Verlet and fourth-order accurate Runge-Kutta (RK4) methods for the solution of linear DAEs.

To the best of our knowledge, open-source libraries for the solution of general nonlinear differential algebraic equations with high-level {\sc{Python}} interfaces are not yet available. Hence a simple forward in time integration method for the solution of the nonlinear pHDAE related to the energy formulation of the heat equation problem has been provided; see \cite{serh:20} for illustrations and discussion. As a future direction, we plan to investigate the potential of the {\sc{PETSc}}'s time stepping library {\sc{TS}} \cite{abhyankar2018petsc} to be able to tackle the solution of large-scale pHDAE systems.

\subsection{\label{sec:scrimp:modred} Model reduction of port-Hamiltonian systems}

Structure-preserving model reduction is of significant importance for stability analysis, optimization or control of problems related to pHs. Hence structure-preserving model reduction algorithms have been implemented in {\sc{SCRIMP}}. In particular, the structure-preserving model reduction algorithm (Algorithm 1) proposed in \cite{chbg:16} has been selected in the pHODE case. We refer the reader to \cite{bhsv:20z} for an illustration, where the model reduction of the pHs related to the wave equation problem is considered. While for linear pHDAE systems consolidated methodologies have been proposed (see, e.g., \cite{eklm:18}), structure-preserving model reduction for general nonlinear differential algebraic
systems remains to be explored, to the best of our knowledge. This is a significant research direction to be considered within {\sc{SCRIMP}} in a near future. 
%
%

\section{Numerical simulations}
\label{sec:simus}

 In this section, PFEM is applied to the pHs presented in Section~\ref{sec:pfem}. We specifically learn
how to define and solve those problems with {\sc{SCRIMP}}. These tutorials introduce the methodology step-by-step and are supposed to be self-contained and independent from the others. We refer the reader to the companion {\sc{Jupyter}} notebooks \cite{bhsv:20z} for additional information.

\let\clearpage\relax

%
%

    \hypertarget{wave:anisotropic-heterogeneous-wave-equation}{%
\subsection{Anisotropic heterogeneous wave
equation}\label{wave:anisotropic-heterogeneous-wave-equation}}

    We first recall the continuous problem related to the anisotropic heterogeneous wave equation, enriched with internal and boundary damping, and tackle the
semi-discretization in space of the port-Hamiltonian system through the
PFEM methodology. This discretization leads to a pHODE formulation as explained in Section \ref{sec:wave}. After time discretization, we perform a numerical simulation to obtain
an approximation of the space-time solution. 

    \hypertarget{wave:problem-statement}{%
\subsubsection{Problem statement}\label{wave:problem-statement}}

We consider the two-dimensional heterogeneous anisotropic wave equation
with impedance boundary condition defined for all $(t\ge 0)$ as:
\begin{eqnarray*}
    \rho(\boldsymbol x)\,\displaystyle \frac{\partial^2}{\partial t^2} w(t,\boldsymbol x) - \text{div}\Big(\bm{\mathcal{T}}(\boldsymbol x)\cdot\textbf{grad} \ w(t,\boldsymbol x)\Big) 	&=& -\epsilon(\boldsymbol x) \, \partial_t w(t,\boldsymbol x), \quad \boldsymbol x \in \Omega, \\
    Z(\boldsymbol x)~(\bm{\mathcal{T}}(\boldsymbol x)\cdot\textbf{grad} \ w(t,\boldsymbol x)) \cdot {\bf{n}} + 
    {\partial_t} w(t,\boldsymbol x) & = & 0, \quad \boldsymbol x \in \partial \Omega, \\ 
    w(0, \boldsymbol x) & = & w_0(x), \quad \boldsymbol x \in \Omega, t=0 \\
    {\partial_t} w(0,\boldsymbol x) & = & w_1(x), \quad \boldsymbol x \in \Omega, t=0,
\end{eqnarray*}
with \(\Omega \subset \mathbb{R}^2\) an open bounded spatial domain with
Lipschitz-continuous boundary \(\partial \Omega\). We
consider here a rectangular shaped domain for \(\Omega\).
\(w(t,\boldsymbol x)\) denotes the deflection from the equilibrium
position at point $\boldsymbol x \in \Omega $ and time \(t\).
\(\rho \in L^{\infty}(\Omega)\) (positive and bounded from below)
denotes the mass density,
\(\bm{\mathcal{T}} \in L^{\infty}(\Omega)^{2\times 2}\) (symmetric
and coercive) the Young's elasticity modulus, $\epsilon$ a positive viscous damping parameter and \(Z(\boldsymbol x)\)
is the positive impedance function defined on \(\partial \Omega\), respectively.

    \hypertarget{wave:setup}{%
\subsubsection{Setup}\label{wave:setup}}

    We initialize here the Python object related to the \verb|Wave_2D| class of
{\sc{SCRIMP}}. This object will be used throughout this section.
    \begin{tcolorbox}[breakable, size=fbox, boxrule=1pt, pad at break*=1mm,colback=cellbackground, colframe=cellborder, enlarge top by=0.25em, enlarge bottom by=0.5em, enlarge top by=0.25em, enlarge bottom by=0.5em]
\begin{Verbatim}[commandchars=\\\{\}]
\PY{n}{W} \PY{o}{=} \PY{n}{SCRIMP}\PY{o}{.}\PY{n}{Wave\PYZus{}2D}\PY{p}{(}\PY{p}{)}
\end{Verbatim}
\end{tcolorbox}

    \hypertarget{wave:constants}{%
\subsubsection{Constants}\label{wave:constants}}

We define the constants related to the rectangular
domain \(\Omega\). The coordinates of the bottom left (\(x_0, y_0\)) and
top right (\(x_L, y_L\)) corners of the rectangle are required.

\begin{tcolorbox}[breakable, size=fbox, boxrule=1pt, pad at break*=1mm,colback=cellbackground, colframe=cellborder, enlarge top by=0.25em, enlarge bottom by=0.5em]
\begin{Verbatim}[commandchars=\\\{\}]
\PY{n}{x0}\PY{p}{,} \PY{n}{xL}\PY{p}{,} \PY{n}{y0}\PY{p}{,} \PY{n}{yL} \PY{o}{=} \PY{l+m+mf}{0.}\PY{p}{,} \PY{l+m+mf}{2.}\PY{p}{,} \PY{l+m+mf}{0.}\PY{p}{,} \PY{l+m+mf}{1.}
\end{Verbatim}
\end{tcolorbox}
We then define the time interval related to the time discretization.
\(t_i, t_f\) denote the initial and final time instants respectively.

\begin{tcolorbox}[breakable, size=fbox, boxrule=1pt, pad at break*=1mm,colback=cellbackground, colframe=cellborder, enlarge top by=0.25em, enlarge bottom by=0.5em]
\begin{Verbatim}[commandchars=\\\{\}]
\PY{n}{ti}\PY{p}{,} \PY{n}{tf}  \PY{o}{=} \PY{l+m+mf}{0.}\PY{p}{,} \PY{l+m+mf}{5.}
\end{Verbatim}
\end{tcolorbox}
We specify that we choose the {\sc{Assimulo}} external library to be used
later for the time integration of the resulting ODE and provide the value of the time step. This should be
considered as a reference value since adaptative methods in time can be
used later.

\begin{tcolorbox}[breakable, size=fbox, boxrule=1pt, pad at break*=1mm,colback=cellbackground, colframe=cellborder, enlarge top by=0.25em, enlarge bottom by=0.5em]
\begin{Verbatim}[commandchars=\\\{\}]
\PY{n}{dt}           \PY{o}{=} \PY{l+m+mf}{1.e\PYZhy{}3}
\PY{n}{ode\PYZus{}library}  \PY{o}{=} \PY{l+s+s1}{\PYZsq{}}\PY{l+s+s1}{ODE:Assimulo}\PY{l+s+s1}{\PYZsq{}}
\end{Verbatim}
\end{tcolorbox}

    \hypertarget{wave:fenics-expressions-definition}{%
\subsubsection{FEniCS expressions
definition}\label{wave:fenics-expressions-definition}}

    For the finite element discretization of the pHs, the {\sc{FEniCS}} library is used in the \verb|Wave_2D| class of {\sc{SCRIMP}}.
Hence to properly use {\sc{FEniCS}} expression definition, we provide the
definition of the different variables in {\sc{C++}} code given in strings. We
first specify the mass density as a function depending on the space
coordinates. Hence in this expression, $x{[}0{]}$ corresponds to the first
spatial variable and $x{[}1{]}$ to the second one, respectively.
\begin{tcolorbox}[breakable, size=fbox, boxrule=1pt, pad at break*=1mm,colback=cellbackground, colframe=cellborder, enlarge top by=0.25em, enlarge bottom by=0.5em]
\begin{Verbatim}[commandchars=\\\{\}]
\PY{n}{Rho}    \PY{o}{=} \PY{l+s+s1}{\PYZsq{}}\PY{l+s+s1}{x[0]*x[0] * (2.\PYZhy{}x[0])+ 1}\PY{l+s+s1}{\PYZsq{}}
\end{Verbatim}
\end{tcolorbox}
We specify the Young's elasticity modulus tensor. Three components are
only required due to the symmetry property of this tensor.

    \begin{tcolorbox}[breakable, size=fbox, boxrule=1pt, pad at break*=1mm,colback=cellbackground, colframe=cellborder, enlarge top by=0.25em, enlarge bottom by=0.5em]
\begin{Verbatim}[commandchars=\\\{\}]
\PY{n}{T11}    \PY{o}{=} \PY{l+s+s1}{\PYZsq{}}\PY{l+s+s1}{x[0]*x[0]+1}\PY{l+s+s1}{\PYZsq{}}
\PY{n}{T12}    \PY{o}{=} \PY{l+s+s1}{\PYZsq{}}\PY{l+s+s1}{x[1]}\PY{l+s+s1}{\PYZsq{}}
\PY{n}{T22}    \PY{o}{=} \PY{l+s+s1}{\PYZsq{}}\PY{l+s+s1}{x[0]+2}\PY{l+s+s1}{\PYZsq{}}
\end{Verbatim}
\end{tcolorbox}
We finally set the impedance function \(Z\) defined on the boundary
of the domain. Here a constant value is used on \(\partial \Omega\). We
also provide the viscous damping parameter (\verb|eps|).

    \begin{tcolorbox}[breakable, size=fbox, boxrule=1pt, pad at break*=1mm,colback=cellbackground, colframe=cellborder, enlarge top by=0.25em, enlarge bottom by=0.5em]
\begin{Verbatim}[commandchars=\\\{\}]
\PY{n}{Z}   \PY{o}{=} \PY{l+s+s1}{\PYZsq{}}\PY{l+s+s1}{0.1}\PY{l+s+s1}{\PYZsq{}}
\PY{n}{eps} \PY{o}{=} \PY{l+s+s1}{\PYZsq{}}\PY{l+s+s1}{25 * x[0] * (xL\PYZhy{}x[0]) * x[1] * (yL\PYZhy{}x[1])}\PY{l+s+s1}{\PYZsq{}}
\end{Verbatim}
\end{tcolorbox}
Finally we specify the initial conditions of the problem related to the energy variables and to the deflection.

    \begin{tcolorbox}[breakable, size=fbox, boxrule=1pt, pad at break*=1mm,colback=cellbackground, colframe=cellborder, enlarge top by=0.25em, enlarge bottom by=0.5em]
\begin{Verbatim}[commandchars=\\\{\}]
\PY{n}{Aq\PYZus{}0\PYZus{}1}\PY{o}{=} \PY{l+s+s1}{\PYZsq{}}\PY{l+s+s1}{0}\PY{l+s+s1}{\PYZsq{}}
\PY{n}{Aq\PYZus{}0\PYZus{}2}\PY{o}{=} \PY{l+s+s1}{\PYZsq{}}\PY{l+s+s1}{0}\PY{l+s+s1}{\PYZsq{}}
\PY{n}{Ap\PYZus{}0}  \PY{o}{=} \PY{l+s+s1}{\PYZsq{}}\PY{l+s+s1}{0}\PY{l+s+s1}{\PYZsq{}}
\PY{n}{W\PYZus{}0}   \PY{o}{=} \PY{l+s+s1}{\PYZsq{}}\PY{l+s+s1}{0}\PY{l+s+s1}{\PYZsq{}}
\end{Verbatim}
\end{tcolorbox}

    \hypertarget{wave:problem-at-the-continuous-level}{%
\subsubsection{Problem at the continuous
level}\label{wave:problem-at-the-continuous-level}}

    We are now able to completely define the problem at the continuous
level. We start by specifying that the computational domain \(\Omega\)
is of rectangular shape. To do so, we provide the
coordinates of the bottom left and top right corners to the \verb|Wave_2D|
object.

    \begin{tcolorbox}[breakable, size=fbox, boxrule=1pt, pad at break*=1mm,colback=cellbackground, colframe=cellborder, enlarge top by=0.25em, enlarge bottom by=0.5em]
\begin{Verbatim}[commandchars=\\\{\}]
\PY{n}{W}\PY{o}{.}\PY{n}{Set\PYZus{}Rectangular\PYZus{}Domain}\PY{p}{(}\PY{n}{x0}\PY{p}{,} \PY{n}{xL}\PY{p}{,} \PY{n}{y0}\PY{p}{,} \PY{n}{yL}\PY{p}{)}\PY{p}{;}
\end{Verbatim}
\end{tcolorbox}
\begin{remark}
General {\sc{Gmsh}} meshes can be imported by the user. However, for the time being, the library does not allow the treatment of mixed boundary conditions on generic meshes.
\end{remark}
We provide next the time integration interval.

    \begin{tcolorbox}[breakable, size=fbox, boxrule=1pt, pad at break*=1mm,colback=cellbackground, colframe=cellborder, enlarge top by=0.25em, enlarge bottom by=0.5em]
\begin{Verbatim}[commandchars=\\\{\}]
\PY{n}{W}\PY{o}{.}\PY{n}{Set\PYZus{}Initial\PYZus{}Final\PYZus{}Time}\PY{p}{(}\PY{n}{ti}\PY{p}{,} \PY{n}{tf}\PY{p}{)}\PY{p}{;}
\end{Verbatim}
\end{tcolorbox}
We then provide the physical parameters related to the wave equation:
the mass density, the Young's elasticity modulus tensor and the
impendance function, respectively.

    \begin{tcolorbox}[breakable, size=fbox, boxrule=1pt, pad at break*=1mm,colback=cellbackground, colframe=cellborder, enlarge top by=0.25em, enlarge bottom by=0.5em]
\begin{Verbatim}[commandchars=\\\{\}]
\PY{n}{W}\PY{o}{.}\PY{n}{Set\PYZus{}Physical\PYZus{}Parameters}\PY{p}{(}\PY{n}{Rho}\PY{p}{,} \PY{n}{T11}\PY{p}{,} \PY{n}{T12}\PY{p}{,} \PY{n}{T22}\PY{p}{)}\PY{p}{;}
\end{Verbatim}
\end{tcolorbox}
We then specify the complete modelling for the damping and thus provide
information related to the impedance function and viscous damping parameter, respectively.

    \begin{tcolorbox}[breakable, size=fbox, boxrule=1pt, pad at break*=1mm,colback=cellbackground, colframe=cellborder, enlarge top by=0.25em, enlarge bottom by=0.5em]
\begin{Verbatim}[commandchars=\\\{\}]
\PY{n}{W}\PY{o}{.}\PY{n}{Set\PYZus{}Damping}\PY{p}{(}\PY{n}{damp}\PY{o}{=}\PY{p}{[}\PY{l+s+s1}{\PYZsq{}}\PY{l+s+s1}{impedance}\PY{l+s+s1}{\PYZsq{}}\PY{p}{,} \PY{l+s+s1}{\PYZsq{}}\PY{l+s+s1}{fluid}\PY{l+s+s1}{\PYZsq{}}\PY{p}{]}\PY{p}{,} \PY{n}{Z}\PY{o}{=}\PY{n}{Z}\PY{p}{,} \PY{n}{eps}\PY{o}{=}\PY{n}{eps}\PY{p}{)}\PY{p}{;}
\end{Verbatim}
\end{tcolorbox}
The user has to provide the temporal and spatial parts of the boundary control function (\verb|Ub_tm0| and \verb|Ub_sp0|, respectively).

    \begin{tcolorbox}[breakable, size=fbox, boxrule=1pt, pad at break*=1mm,colback=cellbackground, colframe=cellborder, enlarge top by=0.25em, enlarge bottom by=0.5em]
\begin{Verbatim}[commandchars=\\\{\}]
\PY{n}{W}\PY{o}{.}\PY{n}{Set\PYZus{}Boundary\PYZus{}Control}\PY{p}{(}\PY{n}{Ub\PYZus{}tm0}\PY{o}{=}\PY{k}{lambda} \PY{n}{t}\PY{p}{:}  \PY{n}{np}\PY{o}{.}\PY{n}{sin}\PY{p}{(} \PY{l+m+mi}{2} \PY{o}{*} \PY{l+m+mi}{2}\PY{o}{*}\PY{n}{pi}\PY{o}{/}\PY{n}{tf} \PY{o}{*}\PY{n}{t}\PY{p}{)} \PY{o}{*} \PY{l+m+mi}{25} \PY{p}{,} \PY{n}{Ub\PYZus{}sp0}\PY{o}{=}\PY{l+s+s1}{\PYZsq{}}\PY{l+s+s1}{x[0] * x[1] * (1\PYZhy{}x[1])}\PY{l+s+s1}{\PYZsq{}}\PY{p}{)}\PY{p}{;}
\end{Verbatim}
\end{tcolorbox}
Finally we provide the initial conditions for the ODE.

    \begin{tcolorbox}[breakable, size=fbox, boxrule=1pt, pad at break*=1mm,colback=cellbackground, colframe=cellborder, enlarge top by=0.25em, enlarge bottom by=0.5em]
\begin{Verbatim}[commandchars=\\\{\}]
\PY{n}{W}\PY{o}{.}\PY{n}{Set\PYZus{}Initial\PYZus{}Data}\PY{p}{(}\PY{n}{Aq\PYZus{}0\PYZus{}1}\PY{o}{=}\PY{l+s+s1}{\PYZsq{}}\PY{l+s+s1}{0}\PY{l+s+s1}{\PYZsq{}}\PY{p}{,} \PY{n}{Aq\PYZus{}0\PYZus{}2}\PY{o}{=}\PY{l+s+s1}{\PYZsq{}}\PY{l+s+s1}{0}\PY{l+s+s1}{\PYZsq{}}\PY{p}{,} \PY{n}{Ap\PYZus{}0}\PY{o}{=}\PY{l+s+s1}{\PYZsq{}}\PY{l+s+s1}{0}\PY{l+s+s1}{\PYZsq{}}\PY{p}{,} \PY{n}{W\PYZus{}0}\PY{o}{=}\PY{l+s+s1}{\PYZsq{}}\PY{l+s+s1}{0}\PY{l+s+s1}{\PYZsq{}}\PY{p}{)}\PY{p}{;}
\end{Verbatim}
\end{tcolorbox}

    \hypertarget{wave:problem-at-the-discrete-level-in-space-and-time}{%
\subsubsection{Problem at the discrete level in space and
time}\label{wave:problem-at-the-discrete-level-in-space-and-time}}

    We start by selecting the computational mesh which is generated with
{\sc{Gmsh}}\footnote{https://gmsh.info/} and saved as a \verb|.xml| file. Here the parameter $rfn\_num$ corresponds to a mesh refinement
parameter.

    \begin{tcolorbox}[breakable, size=fbox, boxrule=1pt, pad at break*=1mm,colback=cellbackground, colframe=cellborder, enlarge top by=0.25em, enlarge bottom by=0.5em]
\begin{Verbatim}[commandchars=\\\{\}]
\PY{n}{W}\PY{o}{.}\PY{n}{Set\PYZus{}Gmsh\PYZus{}Mesh}\PY{p}{(}\PY{l+s+s1}{\PYZsq{}}\PY{l+s+s1}{rectangle.xml}\PY{l+s+s1}{\PYZsq{}}\PY{p}{,} \PY{n}{rfn\PYZus{}num}\PY{o}{=}\PY{l+m+mi}{2}\PY{p}{)}\PY{p}{;}
\end{Verbatim}
\end{tcolorbox}
To perform the discretization in space, we must first specify the
conforming finite element approximation spaces to be used (see \cite{HaiMatSer20}). Concerning
the energy variables associated with the strain, we select the
Raviart-Thomas finite element family known as \(RT_k\) 
consisting of vector functions with a continuous normal component across
the interfaces between the elements of a mesh. For the energy variables
associated with the linear momentum and the boundary variables, we
choose the classical \(P_k\) finite element approximation. The given
combination of parameters \verb|rt_order=0, p_order=1, b_order=1|
corresponds to the \(RT_0 P_1 P_1\) family.

    \begin{tcolorbox}[breakable, size=fbox, boxrule=1pt, pad at break*=1mm,colback=cellbackground, colframe=cellborder, enlarge top by=0.25em, enlarge bottom by=0.5em]
\begin{Verbatim}[commandchars=\\\{\}]
\PY{n}{W}\PY{o}{.}\PY{n}{Set\PYZus{}Finite\PYZus{}Element\PYZus{}Spaces}\PY{p}{(}\PY{n}{family\PYZus{}q}\PY{o}{=}\PY{l+s+s1}{\PYZsq{}}\PY{l+s+s1}{RT}\PY{l+s+s1}{\PYZsq{}}\PY{p}{,} \PY{n}{family\PYZus{}p}\PY{o}{=}\PY{l+s+s1}{\PYZsq{}}\PY{l+s+s1}{P}\PY{l+s+s1}{\PYZsq{}}\PY{p}{,} \PY{n}{family\PYZus{}b}\PY{o}{=}\PY{l+s+s1}{\PYZsq{}}\PY{l+s+s1}{P}\PY{l+s+s1}{\PYZsq{}}\PY{p}{,}\PY{n}{rq}\PY{o}{=}\PY{l+m+mi}{0}\PY{p}{,} \PY{n}{rp}\PY{o}{=}\PY{l+m+mi}{1}\PY{p}{,} \PY{n}{rb}\PY{o}{=}\PY{l+m+mi}{1}\PY{p}{)}\PY{p}{;}
\end{Verbatim}
\end{tcolorbox}
We then perform the semi-discretization in space of the weak formulation
with PFEM. At the end of this stage, the complete formulation of the
pHODE is obtained. The different matrices related to the pHODE system
are constructed in the Assembly method of the \verb|Wave_2D| class of {\sc{SCRIMP}}
and are directly accessible through the object of the \verb|Wave_2D| class. The finite element assembly relies on the
variational formulation of PFEM and exploits the level of abstraction
provided by the UFL used in {\sc{FEniCS}},
leading to a code that is close to the mathematical formulation. The divergence based 
formulation is selected leading to a pHODE system. In other words, the integration by parts will be performed on the second line of~\eqref{eq:pHsWave}.

\begin{tcolorbox}[breakable, size=fbox, boxrule=1pt, pad at break*=1mm,colback=cellbackground, colframe=cellborder, enlarge top by=0.25em, enlarge bottom by=0.5em]
\begin{Verbatim}[commandchars=\\\{\}]
\PY{n}{W}\PY{o}{.}\PY{n}{Assembly}\PY{p}{(}\PY{n}{formulation}\PY{o}{=}\PY{l+s+s1}{\PYZsq{}}\PY{l+s+s1}{Div}\PY{l+s+s1}{\PYZsq{}}\PY{p}{)}\PY{p}{;}
\end{Verbatim}
\end{tcolorbox}
To perform the time integration of the pHODE, we first need to
interpolate both the control function on the boundary and the initial
data on the appropriate finite element spaces.

    \begin{tcolorbox}[breakable, size=fbox, boxrule=1pt, pad at break*=1mm,colback=cellbackground, colframe=cellborder, enlarge top by=0.25em, enlarge bottom by=0.5em]
\begin{Verbatim}[commandchars=\\\{\}]
\PY{n}{W}\PY{o}{.}\PY{n}{Project\PYZus{}Boundary\PYZus{}Control}\PY{p}{(}\PY{p}{)}
\PY{n}{W}\PY{o}{.}\PY{n}{Project\PYZus{}Initial\PYZus{}Data}\PY{p}{(}\PY{p}{)}\PY{p}{;}
\end{Verbatim}
\end{tcolorbox}
Then we specify the parameters related to the time discretization.

\begin{tcolorbox}[breakable, size=fbox, boxrule=1pt, pad at break*=1mm,colback=cellbackground, colframe=cellborder, enlarge top by=0.25em, enlarge bottom by=0.5em]
\begin{Verbatim}[commandchars=\\\{\}]
\PY{n}{W}\PY{o}{.}\PY{n}{Set\PYZus{}Time\PYZus{}Setting}\PY{p}{(}\PY{n}{dt}\PY{p}{)}\PY{p}{;}
\end{Verbatim}
\end{tcolorbox}

    \hypertarget{wave:numerical-approximation-of-the-space-time-solution}{%
\subsubsection{Numerical approximation of the space-time
solution}\label{wave:numerical-approximation-of-the-space-time-solution}}

    We are now able to perform the time integration of the resulting pHODE system and
deduce the behaviour of both the energy variables and the Hamiltonian with respect to the time and space variables, respectively. Detailed
information from the {\sc{Assimulo}} library is included after time
integration.

    \begin{tcolorbox}[breakable, size=fbox, boxrule=1pt, pad at break*=1mm,colback=cellbackground, colframe=cellborder, enlarge top by=0.25em, enlarge bottom by=0.5em]
\begin{Verbatim}[commandchars=\\\{\}]
\PY{n}{A}\PY{p}{,} \PY{n}{Hamiltonian} \PY{o}{=} \PY{n}{W}\PY{o}{.}\PY{n}{Time\PYZus{}Integration}\PY{p}{(}\PY{n}{ode\PYZus{}library}\PY{p}{)}
\end{Verbatim}
\end{tcolorbox}

    \begin{Verbatim}[commandchars=\\\{\}]
ODE Integration using assimulo built-in functions:
Final Run Statistics: ---

 Number of steps                                 : 614
 Number of function evaluations                  : 800
 Number of Jacobian*vector evaluations           : 2977
 Number of function eval. due to Jacobian eval.  : 0
 Number of error test failures                   : 0
 Number of nonlinear iterations                  : 797
 Number of nonlinear convergence failures        : 51

Solver options:

 Solver                   : CVode
 Linear multistep method  : BDF
 Nonlinear solver         : Newton
 Linear solver type       : SPGMR
 Maximal order            : 3
 Tolerances (absolute)    : 1e-05
 Tolerances (relative)    : 1e-05

Simulation interval    : 0.0 - 5.0 seconds.
Elapsed simulation time: 0.9727537930002654 seconds.

    \end{Verbatim}

    \hypertarget{wave:post-processing}{%
\subsubsection{Post-processing}\label{wave:post-processing}}

    We represent the two-dimensional mesh with corresponding degrees of
freedom for each variable in Figure \ref{fig:wave:mesh}.

    \begin{tcolorbox}[breakable, size=fbox, boxrule=1pt, pad at break*=1mm,colback=cellbackground, colframe=cellborder, enlarge top by=0.25em, enlarge bottom by=0.5em]
\begin{Verbatim}[commandchars=\\\{\}]
\PY{n}{W}\PY{o}{.}\PY{n}{notebook} \PY{o}{=} \PY{k+kc}{True}
\PY{n}{W}\PY{o}{.}\PY{n}{Plot\PYZus{}Mesh\PYZus{}with\PYZus{}DOFs}\PY{p}{(}\PY{p}{)}    
\end{Verbatim}
\end{tcolorbox}

    \begin{figure}[!htbp]
    \centering
    \includegraphics[width=0.55\textwidth]{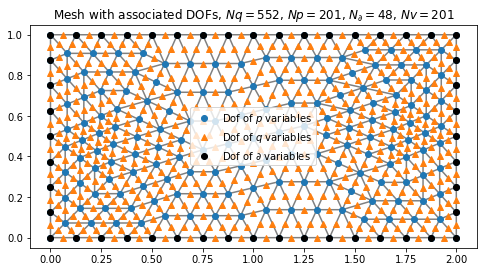}
    \caption{Two-dimensional triangular mesh with corresponding degrees of
freedom for each variable for the anisotropic wave problem with damping.}
    \label{fig:wave:mesh}
    \end{figure}
We plot the Hamiltonian function versus time in Figure \ref{fig:wave:ham}. Here $tspan$ represents the collection of discrete times due to the possibly
    adaptative time procedure used in the time integration library.

\begin{tcolorbox}[breakable, size=fbox, boxrule=1pt, pad at break*=1mm,colback=cellbackground, colframe=cellborder, enlarge top by=0.25em, enlarge bottom by=0.5em]
\begin{Verbatim}[commandchars=\\\{\}]
\PY{n}{W}\PY{o}{.}\PY{n}{Plot\PYZus{}Hamiltonian}\PY{p}{(}\PY{n}{W}\PY{o}{.}\PY{n}{tspan}\PY{p}{,}\PY{n}{Hamiltonian}\PY{p}{,} \PY{n}{marker}\PY{o}{=}\PY{l+s+s1}{\PYZsq{}}\PY{l+s+s1}{o}\PY{l+s+s1}{\PYZsq{}}\PY{p}{)}
\end{Verbatim}
\end{tcolorbox}

    \begin{figure}[!htbp]
    \centering
    \includegraphics[width=0.55\textwidth]{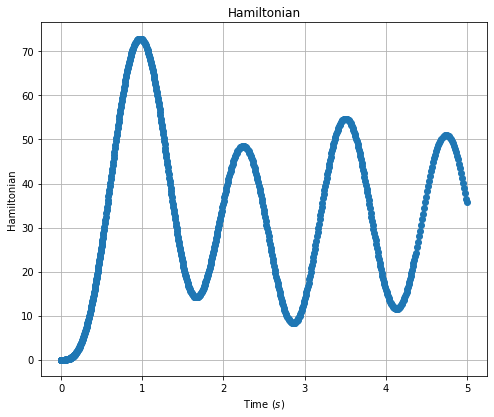}
    \caption{Hamiltonian function versus time for the anisotropic, heterogeneous and boundary controlled wave problem with damping.}
    \label{fig:wave:ham}
    \end{figure} 
The behaviour of the deflection is graphically represented at a
given time. Here we simply plot the deflection at the final time of the
simulation in Figure \ref{fig:wave:def}.

    \begin{tcolorbox}[breakable, size=fbox, boxrule=1pt, pad at break*=1mm,colback=cellbackground, colframe=cellborder, enlarge top by=0.25em, enlarge bottom by=0.5em]
\begin{Verbatim}[commandchars=\\\{\}]
\PY{n}{W}\PY{o}{.}\PY{n}{Plot\PYZus{}3D}\PY{p}{(}\PY{n}{W}\PY{o}{.}\PY{n}{Get\PYZus{}Deflection}\PY{p}{(}\PY{n}{A}\PY{p}{)}\PY{p}{,} \PY{n}{tf}\PY{p}{,} \PY{l+s+s1}{\PYZsq{}}\PY{l+s+s1}{Deflection at t=}\PY{l+s+s1}{\PYZsq{}}\PY{o}{+}\PY{n+nb}{str}\PY{p}{(}\PY{n}{tf}\PY{p}{)}\PY{p}{)}
\end{Verbatim}
\end{tcolorbox}

    \begin{figure}[!htbp]
    \centering
    \includegraphics[width=0.55\textwidth]{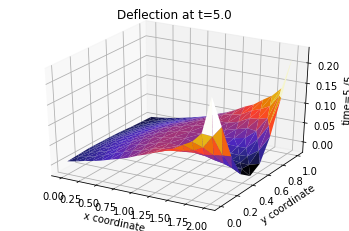}
    \caption{Deflection at the final time for the anisotropic, heterogeneous and boundary controlled wave problem with damping.}
    \label{fig:wave:def}
    \end{figure}
The related {\sc{Jupyter}} notebook \cite{bhsv:20z} further illustrates how to obtain a structure-preserving reduced model of this port-Hamiltonian system. After application of the model reduction algorithm proposed in \cite{chbg:16}, a pHODE of reduced size has to be integrated to obtain an approximate solution of the wave propagation problem. This is further illustrated on the simple application detailed in this section. In addition, a supplementary notebook illustrates the numerical simulation of the wave equation problem, when mixed boundary conditions (i.e. Dirichlet and Neumann conditions) on the boundary control function are imposed by Lagrange multipliers \cite{brugnoli2020wc}.

\hypertarget{min:the-mindlin-plate-problem}{%
\subsection{The Mindlin plate problem}\label{the-mindlin-plate-problem}}

    We first recall the considered continuous problem related to the Mindlin plate and tackle the semi-discretization in space of the pHs by PFEM. After transformation and time discretization, we perform a numerical simulation to obtain
an approximation of the space-time solution. As in Section \ref{wave:anisotropic-heterogeneous-wave-equation}, the procedure is described step-by-step and detailed explanations and numerical illustrations are provided. 

    \hypertarget{min:problem-statement}{%
\subsubsection{Problem statement}\label{problem-statement}}

Consider the Mindlin plate problem defined for all \(t\ge 0\) as:
\begin{equation}
\left\{\begin{aligned}
   \displaystyle \rho b \partial_{tt} w &= \mathrm{div}(\boldsymbol{q}), \quad \boldsymbol x \in \Omega = \{0, L_x\}\times \{0, L_y\},  \\
\displaystyle \frac{\rho b^3}{12} \partial_{tt}{\boldsymbol \theta} &= \boldsymbol{q} + \mathrm{Div}(\boldsymbol{M}) \\
\end{aligned}\right.
\end{equation}
with initial conditions:
\begin{equation}
\begin{aligned}
    w(0, \boldsymbol x) &= 0, \\
    {\partial_t} w(0,\boldsymbol x) &= w_1(x),  \\
\end{aligned} \qquad
\begin{aligned}
   \boldsymbol{\theta}(0, \boldsymbol x) & = \boldsymbol{0}, \\
  {\partial_t} \boldsymbol{\theta}(0,\boldsymbol x) &= \boldsymbol{0}, \\
\end{aligned} \qquad
\begin{aligned}
\boldsymbol x &\in \Omega, t=0,\\
\boldsymbol x &\in \Omega, t=0,\\
\end{aligned}
  \end{equation}
and boundary conditions:
\begin{equation}
\begin{aligned}
    w\vert_{\Gamma_D}&=u_D(t), \\
    \boldsymbol{q} \cdot \boldsymbol{n}\vert_{\Gamma_N}&=u_N(t), \\
\end{aligned} \qquad
\begin{aligned}
    \boldsymbol{\theta}\vert_{\Gamma_D}&=0,  \\ 
     \quad \boldsymbol{M} \cdot \boldsymbol{n}\vert_{\Gamma_N}&=0,  \\ 
\end{aligned} \qquad
\begin{aligned}
\Gamma_D&=\{x=0 \}, \\
\Gamma_N&=\{x=L_x, y=0, y=L_y\}.
\end{aligned}
\end{equation}
Mixed boundary conditions are considered in this example. The subsets
\(\Gamma_N, \; \Gamma_D\) represent the subsets of the boundary where
Neumann and Dirichlet conditions hold respectively. The Dirichlet
conditions are enforced using Lagrange multipliers. The PFEM discretization then leads
to a pHDAE, as explained in \cite{brugnoli2020wc} for the wave equation. The following expressions have been considered for the initial and boundary conditions:
\begin{equation}
    w_1 = xy, \qquad u_D = 0.01(\cos\left(2\pi \frac{t}{t_f}\right)-1), \qquad u_N = 10^5\sin\left(2\pi\frac{x}{L_x}\right)\left(1-\exp^{-10\frac{t}{t_f}}\right).
\end{equation}

    \hypertarget{min:setup}{%
\subsubsection{Setup}\label{setup}}

    We initialize here the {\sc{Python}} object related to the \verb|Mindlin| class of
{\sc{SCRIMP}}. This object will be used throughout this section.

    \begin{tcolorbox}[breakable, size=fbox, boxrule=1pt, pad at break*=1mm,colback=cellbackground, colframe=cellborder, enlarge top by=0.25em, enlarge bottom by=0.5em]
\begin{Verbatim}[commandchars=\\\{\}]
\PY{n}{Min} \PY{o}{=} \PY{n}{SCRIMP}\PY{o}{.}\PY{n}{Mindlin}\PY{p}{(}\PY{p}{)}
\end{Verbatim}
\end{tcolorbox}

    \hypertarget{min:constants}{%
\subsubsection{Constants}\label{constants}}

We define the constants related to the rectangular
domain \(\Omega\). The coordinates of the bottom left
\((x_0, y_0) =(0, 0)\) and the top right \((x_L, y_L) = (L_x, L_y)\)
corners of the rectangle are required.

    \begin{tcolorbox}[breakable, size=fbox, boxrule=1pt, pad at break*=1mm,colback=cellbackground, colframe=cellborder, enlarge top by=0.25em, enlarge bottom by=0.5em]
\begin{Verbatim}[commandchars=\\\{\}]
\PY{n}{x0}\PY{p}{,} \PY{n}{xL}\PY{p}{,} \PY{n}{y0}\PY{p}{,} \PY{n}{yL} \PY{o}{=} \PY{l+m+mf}{0.}\PY{p}{,} \PY{l+m+mf}{2.}\PY{p}{,} \PY{l+m+mf}{0.}\PY{p}{,} \PY{l+m+mf}{1.}
\end{Verbatim}
\end{tcolorbox}
As in the previous example, the time interval related to the time discretization is defined as follows:

    \begin{tcolorbox}[breakable, size=fbox, boxrule=1pt, pad at break*=1mm,colback=cellbackground, colframe=cellborder, enlarge top by=0.25em, enlarge bottom by=0.5em]
\begin{Verbatim}[commandchars=\\\{\}]
\PY{n}{ti}\PY{p}{,} \PY{n}{tf}  \PY{o}{=} \PY{l+m+mf}{0.}\PY{p}{,} \PY{l+m+mf}{0.01}
\end{Verbatim}
\end{tcolorbox}
A Runge-Kutta method for the time integration of the system is prescribed. This method is conditionally stable, so the time-step has to be set
accurately to avoid numerical instabilities.

\begin{tcolorbox}[breakable, size=fbox, boxrule=1pt, pad at break*=1mm,colback=cellbackground, colframe=cellborder, enlarge top by=0.25em, enlarge bottom by=0.5em]
\begin{Verbatim}[commandchars=\\\{\}]
\PY{n}{dt}           \PY{o}{=} \PY{l+m+mf}{1.e\PYZhy{}6}
\PY{n}{dae\PYZus{}library}  \PY{o}{=} \PY{l+s+s1}{\PYZsq{}}\PY{l+s+s1}{DAE:RK4\PYZus{}Augmented}\PY{l+s+s1}{\PYZsq{}}
\end{Verbatim}
\end{tcolorbox}

    \hypertarget{min:fenics-expressions-definition}{%
\subsubsection{{\sc{FEniCS}} expressions
definition}\label{fenics-expressions-definition}}

    The {\sc{FEniCS}} library is also used in the \verb|Mindlin| class of {\sc{SCRIMP}}.
The coefficients related to the physical parameters of the isotropic plate can
be provided as either real numbers or {\sc{FEniCS}} expressions.

    \begin{tcolorbox}[breakable, size=fbox, boxrule=1pt, pad at break*=1mm,colback=cellbackground, colframe=cellborder, enlarge top by=0.25em, enlarge bottom by=0.5em]
\begin{Verbatim}[commandchars=\\\{\}]
\PY{n}{E}   \PY{o}{=} \PY{l+m+mi}{7}\PY{o}{*}\PY{l+m+mi}{10}\PY{o}{*}\PY{o}{*}\PY{l+m+mi}{10}
\PY{n}{rho} \PY{o}{=} \PY{l+m+mi}{2700}
\PY{n}{nu}  \PY{o}{=} \PY{l+m+mf}{0.3}
\PY{n}{h}   \PY{o}{=} \PY{l+m+mf}{0.1}
\PY{n}{k}   \PY{o}{=} \PY{l+m+mi}{5}\PY{o}{/}\PY{l+m+mi}{6}
\end{Verbatim}
\end{tcolorbox}
Similarly the initial vertical condition
\(w_1\) can be defined as a string. It represents a {\sc{C++}} code that will be compiled by the {\sc{Dolfin}} library of {\sc{FEniCS}}.

    \begin{tcolorbox}[breakable, size=fbox, boxrule=1pt, pad at break*=1mm,colback=cellbackground, colframe=cellborder, enlarge top by=0.25em, enlarge bottom by=0.5em]
\begin{Verbatim}[commandchars=\\\{\}]
\PY{n}{ew\PYZus{}0} \PY{o}{=} \PY{l+s+s1}{\PYZsq{}}\PY{l+s+s1}{x[0]*x[1]}\PY{l+s+s1}{\PYZsq{}}
\end{Verbatim}
\end{tcolorbox}
This means that the initial velocity satisfies \(w_1= xy\). Note that
the initial condition has to be compatible with the boundary conditions. The other
initial conditions will be set to zero.

    \hypertarget{min:problem-at-the-continuous-level}{%
\subsubsection{Problem at the continuous
level}\label{problem-at-the-continuous-level}}

    We are now able to completely define the problem at the continuous
level. We start by specifying that the computational domain \(\Omega\)
is of rectangular shape. To define \(\Omega\), we provide the
coordinates of the bottom left and top right corners to the \verb|Mindlin|
object.

    \begin{tcolorbox}[breakable, size=fbox, boxrule=1pt, pad at break*=1mm,colback=cellbackground, colframe=cellborder, enlarge top by=0.25em, enlarge bottom by=0.5em]
\begin{Verbatim}[commandchars=\\\{\}]
\PY{n}{Min}\PY{o}{.}\PY{n}{Set\PYZus{}Rectangular\PYZus{}Domain}\PY{p}{(}\PY{n}{x0}\PY{p}{,} \PY{n}{xL}\PY{p}{,} \PY{n}{y0}\PY{p}{,} \PY{n}{yL}\PY{p}{)}\PY{p}{;}
\end{Verbatim}
\end{tcolorbox}
The time integration interval is then given.

    \begin{tcolorbox}[breakable, size=fbox, boxrule=1pt, pad at break*=1mm,colback=cellbackground, colframe=cellborder, enlarge top by=0.25em, enlarge bottom by=0.5em]
\begin{Verbatim}[commandchars=\\\{\}]
\PY{n}{Min}\PY{o}{.}\PY{n}{Set\PYZus{}Initial\PYZus{}Final\PYZus{}Time}\PY{p}{(}\PY{n}{ti}\PY{p}{,} \PY{n}{tf}\PY{p}{)}\PY{p}{;}
\end{Verbatim}
\end{tcolorbox}
The physical parameters related to the Mindlin plate are set.

    \begin{tcolorbox}[breakable, size=fbox, boxrule=1pt, pad at break*=1mm,colback=cellbackground, colframe=cellborder, enlarge top by=0.25em, enlarge bottom by=0.5em]
\begin{Verbatim}[commandchars=\\\{\}]
\PY{n}{Min}\PY{o}{.}\PY{n}{Set\PYZus{}Physical\PYZus{}Parameters}\PY{p}{(}\PY{n}{rho}\PY{p}{,} \PY{n}{h}\PY{p}{,} \PY{n}{E}\PY{p}{,} \PY{n}{nu}\PY{p}{,} \PY{n}{k}\PY{p}{,} \PY{n}{init\PYZus{}by\PYZus{}value}\PY{o}{=}\PY{k+kc}{True}\PY{p}{)}\PY{p}{;}
\end{Verbatim}
\end{tcolorbox}
Finally the initial conditions in terms of co-energy variables are also set.

    \begin{tcolorbox}[breakable, size=fbox, boxrule=1pt, pad at break*=1mm,colback=cellbackground, colframe=cellborder, enlarge top by=0.25em, enlarge bottom by=0.5em]
\begin{Verbatim}[commandchars=\\\{\}]
\PY{n}{Min}\PY{o}{.}\PY{n}{Set\PYZus{}Initial\PYZus{}Data}\PY{p}{(}\PY{n}{W\PYZus{}0}\PY{o}{=}\PY{l+s+s1}{\PYZsq{}}\PY{l+s+s1}{0}\PY{l+s+s1}{\PYZsq{}}\PY{p}{,} \PY{n}{Th1\PYZus{}0}\PY{o}{=}\PY{l+s+s1}{\PYZsq{}}\PY{l+s+s1}{0}\PY{l+s+s1}{\PYZsq{}}\PY{p}{,} \PY{n}{Th2\PYZus{}0}\PY{o}{=}\PY{l+s+s1}{\PYZsq{}}\PY{l+s+s1}{0}\PY{l+s+s1}{\PYZsq{}}\PY{p}{,}\PYZbs{}
                     \PY{n}{ew\PYZus{}0}\PY{o}{=}\PY{n}{ew\PYZus{}0}\PY{p}{,} \PY{n}{eth1\PYZus{}0}\PY{o}{=}\PY{l+s+s1}{\PYZsq{}}\PY{l+s+s1}{0}\PY{l+s+s1}{\PYZsq{}}\PY{p}{,} \PY{n}{eth2\PYZus{}0}\PY{o}{=}\PY{l+s+s1}{\PYZsq{}}\PY{l+s+s1}{0}\PY{l+s+s1}{\PYZsq{}}\PY{p}{,}\PYZbs{}
                     \PY{n}{Ekap11\PYZus{}0}\PY{o}{=}\PY{l+s+s1}{\PYZsq{}}\PY{l+s+s1}{0}\PY{l+s+s1}{\PYZsq{}}\PY{p}{,} \PY{n}{Ekap12\PYZus{}0}\PY{o}{=}\PY{l+s+s1}{\PYZsq{}}\PY{l+s+s1}{0}\PY{l+s+s1}{\PYZsq{}}\PY{p}{,} \PY{n}{Ekap22\PYZus{}0}\PY{o}{=}\PY{l+s+s1}{\PYZsq{}}\PY{l+s+s1}{0}\PY{l+s+s1}{\PYZsq{}}\PY{p}{,}\PYZbs{}
                     \PY{n}{egam1\PYZus{}0}\PY{o}{=}\PY{l+s+s1}{\PYZsq{}}\PY{l+s+s1}{0}\PY{l+s+s1}{\PYZsq{}}\PY{p}{,} \PY{n}{egam2\PYZus{}0}\PY{o}{=}\PY{l+s+s1}{\PYZsq{}}\PY{l+s+s1}{0}\PY{l+s+s1}{\PYZsq{}}\PY{p}{)}\PY{p}{;}
\end{Verbatim}
\end{tcolorbox}

    \hypertarget{min:problem-at-the-discrete-level-in-space-and-time}{%
\subsubsection{Problem at the discrete level in space and
time}\label{problem-at-the-discrete-level-in-space-and-time}}

    We start by selecting the computational mesh which is generated with
{\sc{FEniCS}} inner mesh utilities. The first parameter corresponds to a mesh refinement parameter.

    \begin{tcolorbox}[breakable, size=fbox, boxrule=1pt, pad at break*=1mm,colback=cellbackground, colframe=cellborder, enlarge top by=0.25em, enlarge bottom by=0.5em]
\begin{Verbatim}[commandchars=\\\{\}]
\PY{n}{Min}\PY{o}{.}\PY{n}{Generate\PYZus{}Mesh}\PY{p}{(}\PY{l+m+mi}{10}\PY{p}{,} \PY{n}{structured\PYZus{}mesh}\PY{o}{=}\PY{k+kc}{True}\PY{p}{)}\PY{p}{;}
\end{Verbatim}
\end{tcolorbox}
To perform the discretization in space, the conforming finite element approximation spaces to be used has to be specified. The finite
element for the linear and angular velocity are Lagrange polynomials of
order $r$. The momenta tensor and shear stress are chosen as Discontinuous
 Galerkin elements of order $r-1$ \cite{brug:20}. This choice of finite elements is similar to the one proposed in \cite{cohen2007mindlin}, but a simplicial mesh is used instead of a quadrilateral one. By default, the boundary variables are approximated as Lagrange polynomials of order $1$. This can be easily changed
through the option \verb|family_b|~for the family, and \verb|rb| for the degree.

    \begin{tcolorbox}[breakable, size=fbox, boxrule=1pt, pad at break*=1mm,colback=cellbackground, colframe=cellborder, enlarge top by=0.25em, enlarge bottom by=0.5em]
\begin{Verbatim}[commandchars=\\\{\}]
\PY{n}{Min}\PY{o}{.}\PY{n}{Set\PYZus{}Finite\PYZus{}Elements\PYZus{}Spaces}\PY{p}{(}\PY{n}{r}\PY{o}{=}\PY{l+m+mi}{1}\PY{p}{)}\PY{p}{;}
\end{Verbatim}
\end{tcolorbox}
We then perform the semi-discretization in space of the weak formulation
with PFEM. At the end of this stage, the complete formulation
of the pHDAE is obtained. The different matrices related to the pHDAE
system are constructed in the \verb|Assembly_Mixed_BC| method of the Mindlin
class of {\sc{SCRIMP}} and are directly accessible through the object of the
Mindlin class. The subsets named \verb|G1|, \verb|G2|, \verb|G3|, \verb|G4|, denote the left, bottom, right and top sides of the rectangle,
respectively.

\noindent In {\sc{SCRIMP}} the boundary control $\boldsymbol{u}_\partial$ is assumed to take the form: 
\begin{verbatim}
    Ub_tm0(t) * Ub_sp0(x) + Ub_tm1(t) + Ub_sp1(x)
\end{verbatim}
Its derivative $\dot{\boldsymbol{u}}_\partial$ is expressed as:
\begin{verbatim}
    Ub_tm0_dir(t) * Ub_sp0(x) + Ub_tm1_dir(t) 
\end{verbatim}

To integrate in time we need to provide the derivative of
the boundary condition. This information is provided by the variables
\(\verb|Ub_tm0_dir|, \; \verb|Ub_tm1_dir|\), respectively. This is needed to reduce the resulting DAE of index $2$ to index $1$.

    \begin{tcolorbox}[breakable, size=fbox, boxrule=1pt, pad at break*=1mm,colback=cellbackground, colframe=cellborder, enlarge top by=0.25em, enlarge bottom by=0.5em]
\begin{Verbatim}[commandchars=\\\{\}]
\PY{n}{Min}\PY{o}{.}\PY{n}{Set\PYZus{}Mixed\PYZus{}Boundaries}\PY{p}{(}\PY{n}{Dir}\PY{o}{=}\PY{p}{[}\PY{l+s+s1}{\PYZsq{}}\PY{l+s+s1}{G1}\PY{l+s+s1}{\PYZsq{}}\PY{p}{]}\PY{p}{,} \PY{n}{Nor}\PY{o}{=}\PY{p}{[}\PY{l+s+s1}{\PYZsq{}}\PY{l+s+s1}{G2}\PY{l+s+s1}{\PYZsq{}}\PY{p}{,} \PY{l+s+s1}{\PYZsq{}}\PY{l+s+s1}{G3}\PY{l+s+s1}{\PYZsq{}}\PY{p}{,} \PY{l+s+s1}{\PYZsq{}}\PY{l+s+s1}{G4}\PY{l+s+s1}{\PYZsq{}}\PY{p}{]}\PY{p}{)}
\PY{n}{Min}\PY{o}{.}\PY{n}{Assembly\PYZus{}Mixed\PYZus{}BC}\PY{p}{(}\PY{p}{)} 
\PY{n}{Min}\PY{o}{.}\PY{n}{Set\PYZus{}Mixed\PYZus{}BC\PYZus{}Normal}\PY{p}{(}\PY{n}{Ub\PYZus{}tm0}\PY{o}{=}\PY{k}{lambda} \PY{n}{t}\PY{p}{:} \PY{n}{np}\PY{o}{.}\PY{n}{array}\PY{p}{(}\PY{p}{[}\PY{p}{(}\PY{l+m+mi}{1} \PY{o}{\PYZhy{}} \PY{n}{np}\PY{o}{.}\PY{n}{exp}\PY{p}{(}\PY{o}{\PYZhy{}}\PY{n}{t}\PY{o}{/}\PY{n}{tf}\PY{p}{)} \PY{p}{)}\PY{p}{,}\PY{l+m+mi}{0}\PY{p}{,}\PY{l+m+mi}{0}\PY{p}{]}\PY{p}{)} \PY{p}{,}\PYZbs{}
                        \PY{n}{Ub\PYZus{}sp0}\PY{o}{=}\PY{p}{(}\PY{l+s+s1}{\PYZsq{}}\PY{l+s+s1}{100000*sin(2*pi/xL*x[0])}\PY{l+s+s1}{\PYZsq{}}\PY{p}{,} \PY{l+s+s1}{\PYZsq{}}\PY{l+s+s1}{0.}\PY{l+s+s1}{\PYZsq{}}\PY{p}{,} \PY{l+s+s1}{\PYZsq{}}\PY{l+s+s1}{0.}\PY{l+s+s1}{\PYZsq{}}\PY{p}{)}\PY{p}{)}
\PY{n}{amp}   \PY{o}{=} \PY{l+m+mf}{0.01}
\PY{n}{omega} \PY{o}{=} \PY{l+m+mi}{2}\PY{o}{*}\PY{n}{pi}\PY{o}{/}\PY{n}{tf}
\PY{n}{Min}\PY{o}{.}\PY{n}{Set\PYZus{}Mixed\PYZus{}BC\PYZus{}Dirichlet}\PY{p}{(}\PY{n}{Ub\PYZus{}tm0}\PY{o}{=}\PY{k}{lambda} \PY{n}{t} \PY{p}{:} \PY{n}{np}\PY{o}{.}\PY{n}{array}\PY{p}{(}\PY{p}{[}\PY{o}{\PYZhy{}}\PY{n}{amp}\PY{o}{*}\PY{n}{omega}\PY{o}{*}\PY{n}{np}\PY{o}{.}\PY{n}{sin}\PY{p}{(}\PY{n}{omega}\PY{o}{*}\PY{n}{t}\PY{p}{)}\PY{p}{,}\PY{l+m+mi}{0}\PY{p}{,}\PY{l+m+mi}{0}\PY{p}{]}\PY{p}{)}\PY{p}{,} \PY{n}{Ub\PYZus{}sp0}\PY{o}{=}\PY{p}{(}\PY{l+s+s1}{\PYZsq{}}\PY{l+s+s1}{1.}\PY{l+s+s1}{\PYZsq{}}\PY{p}{,} \PY{l+s+s1}{\PYZsq{}}\PY{l+s+s1}{0.}\PY{l+s+s1}{\PYZsq{}}\PY{p}{,}\PY{l+s+s1}{\PYZsq{}}\PY{l+s+s1}{0.}\PY{l+s+s1}{\PYZsq{}}\PY{p}{)}\PY{p}{,}\PYZbs{}
                           \PY{n}{Ub\PYZus{}tm0\PYZus{}dir}\PY{o}{=}\PY{k}{lambda} \PY{n}{t} \PY{p}{:} \PY{n}{np}\PY{o}{.}\PY{n}{array}\PY{p}{(}\PY{p}{[}\PY{o}{\PYZhy{}}\PY{n}{amp}\PY{o}{*}\PY{n}{omega}\PY{o}{*}\PY{o}{*}\PY{l+m+mi}{2}\PY{o}{*}\PY{n}{np}\PY{o}{.}\PY{n}{cos}\PY{p}{(}\PY{n}{omega}\PY{o}{*}\PY{n}{t}\PY{p}{)}\PY{p}{,}\PY{l+m+mi}{0}\PY{p}{,}\PY{l+m+mi}{0}\PY{p}{]}\PY{p}{)}\PY{p}{)}\PY{p}{;}
\end{Verbatim}
\end{tcolorbox}
To perform the time integration of the pHDAE, we first need to
interpolate the boundary control function and the initial data on the appropriate finite element spaces.

    \begin{tcolorbox}[breakable, size=fbox, boxrule=1pt, pad at break*=1mm,colback=cellbackground, colframe=cellborder, enlarge top by=0.25em, enlarge bottom by=0.5em]
\begin{Verbatim}[commandchars=\\\{\}]
\PY{n}{Min}\PY{o}{.}\PY{n}{Project\PYZus{}Boundary\PYZus{}Control}\PY{p}{(}\PY{p}{)}
\PY{n}{Min}\PY{o}{.}\PY{n}{Project\PYZus{}Initial\PYZus{}Data}\PY{p}{(}\PY{p}{)}\PY{p}{;}
\end{Verbatim}
\end{tcolorbox}
Finally the specification of the parameters related to the time discretization is made.

    \begin{tcolorbox}[breakable, size=fbox, boxrule=1pt, pad at break*=1mm,colback=cellbackground, colframe=cellborder, enlarge top by=0.25em, enlarge bottom by=0.5em]
\begin{Verbatim}[commandchars=\\\{\}]
\PY{n}{Min}\PY{o}{.}\PY{n}{Set\PYZus{}Time\PYZus{}Setting}\PY{p}{(}\PY{n}{dt}\PY{p}{)}\PY{p}{;}
\end{Verbatim}
\end{tcolorbox}

    \hypertarget{min:numerical-approximation-of-the-space-time-solution}{%
\subsubsection{Numerical approximation of the space-time
solution}\label{numerical-approximation-of-the-space-time-solution}}

    For the numerical approximation of the solution of the pHDAE system, the
algebraic condition is differentiated. The integrator \verb|'DAE:RK4_Augmented'| takes as input a pHDAE. Then, it exploits a projection method to express the Lagrange multiplier in terms of the unknown  \cite{benner2015time}, thus reducing the original DAE system into a purely ODE one. This allows employing standard ODE solvers for the time integration, as discussed in Section \ref{constants}.\\

    \begin{tcolorbox}[breakable, size=fbox, boxrule=1pt, pad at break*=1mm,colback=cellbackground, colframe=cellborder, enlarge top by=0.25em, enlarge bottom by=0.5em]
\begin{Verbatim}[commandchars=\\\{\}]
\PY{n}{A}\PY{p}{,} \PY{n}{Hamiltonian} \PY{o}{=} \PY{n}{Min}\PY{o}{.}\PY{n}{Time\PYZus{}Integration}\PY{p}{(}\PY{n}{dae\PYZus{}library}\PY{p}{)}
\end{Verbatim}
\end{tcolorbox}

    \hypertarget{min:post-processing}{%
\subsubsection{Post-processing}\label{post-processing}}


    


    
\begin{figure}[!htbp]
    \centering
    \includegraphics[width=0.55\textwidth]{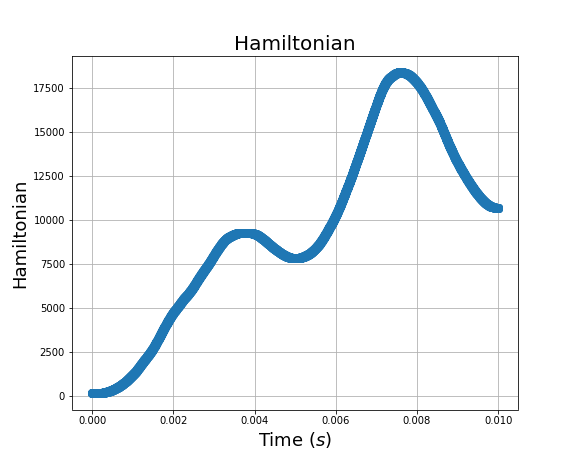}
    \caption{Hamiltonian versus time for the Mindlin plate problem.}
    \label{Hamiltonian:AB}
\end{figure} 

\begin{figure}[!htbp]
    \centering
    \begin{subfigure}[b]{0.46\columnwidth}%
    \centering
        \includegraphics[width=\textwidth]{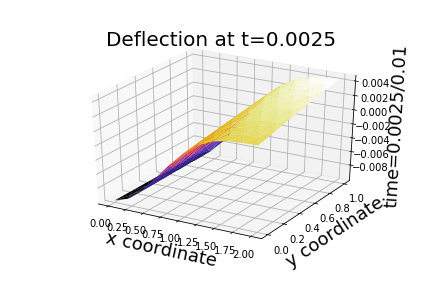}
        \subcaption{Deflection $w$ at $t=t_f/4$}%
    \end{subfigure}%
    \hspace{8pt}%
    \begin{subfigure}[b]{0.46\columnwidth}%
    \centering
        \includegraphics[width=\textwidth]{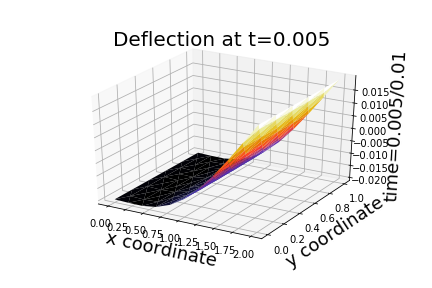}
        \subcaption{Deflection $w$ at $t=t_f/2$}%
    \end{subfigure} \\
    \begin{subfigure}[b]{0.46\columnwidth}%
    \centering
        \includegraphics[width=\textwidth]{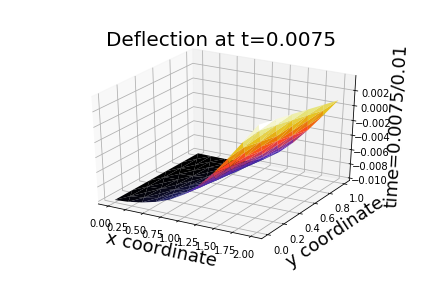}%
        \subcaption{Deflection $w$ at $t=3 t_f/4$}%
    \end{subfigure}%
    \hspace{8pt}%
    \begin{subfigure}[b]{0.46\columnwidth}%
    \centering
        \includegraphics[width=\textwidth]{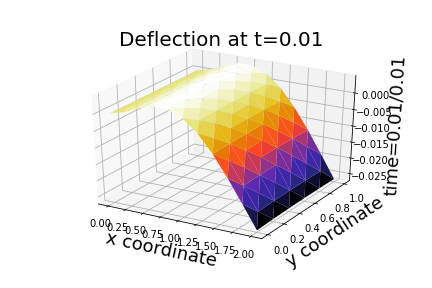}%
        \subcaption{Deflection $w$ at $t=t_f$}%
    \end{subfigure}%
    \caption{Snapshots of the vertical deflection of the Mindlin plate at different instants.}
    \label{Deflection:AB}
\end{figure}

Post-processing is performed similarly as in Section \ref{wave:post-processing}. Hence we omit the related {\sc{Python}} lines of code for sake of brevity. In Figure \ref{Hamiltonian:AB} the evolution of the Hamiltonian function is shown versus time. We note that the Dirichlet condition causes an increase in energy. In Figure 5 snapshots of the vertical deflection at different instants are shown. We remark that the Neumann boundary condition causes the plate to bend asymmetrically.

%
%

    \hypertarget{heath:anisotropic-heterogeneous-heat-equation}{%
\subsection{Anisotropic heterogeneous heat
equation}\label{heat:anisotropic-heterogeneous-heat-equation}}

    This third tutorial aims at illustrating PFEM to discretize the pHs presented in Section~\ref{sec:heat}, modelling the heat equation. We specifically learn how to define and solve this problem with {\sc{SCRIMP}}.
    We first define the continuous problem by using a specific
class of {\sc{SCRIMP}} related to the heat equation in two dimensions. Then we
tackle the discretization in space of the pHs
through PFEM. The discretization of the energy
formulation leads to a \emph{nonlinear} pHDAE formulation. After time discretization, we
perform a numerical simulation to obtain an approximation of the
space-time solution. Finally a simple post-processing is provided.

    \hypertarget{heath:problem-statement}{%
\subsubsection{Problem statement}\label{heat:problem-statement}}

We consider the two-dimensional heterogeneous anisotropic heat equation
defined for all \(t\ge 0\) as

\begin{eqnarray*}
    \rho(\boldsymbol x)\, C_V(\boldsymbol x) \displaystyle \frac{\partial}{\partial t} T(t,\boldsymbol x) &=& \text{div}\Big(\bm{\lambda}(\boldsymbol x)\cdot\textbf{grad} \ T(t,\boldsymbol x)\Big), \quad \boldsymbol x \in \Omega, \\
    T(t,\boldsymbol x) & = & v_{\partial}(t,\boldsymbol x) ,      \quad \boldsymbol x \in \partial \Omega, \\ 
    T(0, \boldsymbol x)             & = & T_0(x), \quad \boldsymbol x \in \Omega, t=0 \\
\end{eqnarray*}
with \(\Omega \subset \mathbb{R}^2\) an open bounded spatial domain with
Lipschitz-continuous boundary \(\partial \Omega\).
\(v_{\partial}(t,\boldsymbol x)\) represents the boundary control
function on the temperature (the notation $u$ is kept for the internal energy density). \(T(t,\boldsymbol x)\) denotes the temperature at point $\boldsymbol x
\in \Omega $ and time \(t\). \(\rho \in L^{\infty}(\Omega)\) (positive
and bounded from below) denotes the mass density,
\(\bm{\lambda} \in L^{\infty}(\Omega)^{2\times 2}\)
(symmetric and coercive) the thermal conductivity. In the following
\(\Omega\) is assumed to be of rectangular shape.

    \hypertarget{port-hamiltonian-formulation}{%
\subsubsection{Port-Hamiltonian formulation}\label{port-hamiltonian-formulation}}

We refer to \cite{SerMatHai19b,SerMatHai19c} for the
modeling and discretization of various port-Hamiltonian formulations of
this problem. The authors consider quadratic Lyapunov functional, entropy or internal
energy as Hamiltonian, respectively. We will consider the PFEM discretization of the internal energy functional
formulation as proposed in Section \ref{sec:heat}, which will lead to a nonlinear pHDAE. Our goal in this tutorial
is to show how a pHDAE system can be formulated and solved with
{\sc{SCRIMP}}. 

    \hypertarget{heath:setup}{%
\subsubsection{Setup}\label{heat:setup}}

    We initialize here the Python object related to the energy formulation
of the \verb|Heat_2D| class of {\sc{SCRIMP}}, that is assumed to be imported. This object will be used throughout
this tutorial. 
    \begin{tcolorbox}[breakable, size=fbox, boxrule=1pt, pad at break*=1mm,colback=cellbackground, colframe=cellborder, enlarge top by=0.25em, enlarge bottom by=0.5em]
\begin{Verbatim}[commandchars=\\\{\}]
\PY{n}{H} \PY{o}{=} \PY{n}{Energy}\PY{p}{(}\PY{p}{)}
\end{Verbatim}
\end{tcolorbox}
\verb|Energy| corresponds to a class inherited from the \verb|Heat_2D| base class. This base class contains implementations of the Lyapunov and entropy formulations as well. 
    \hypertarget{heath:constants}{%
\subsubsection{Constants}\label{heat:constants}}

The same lines of code as for the \verb|Wave_2D| and \verb|Mindlin| classes are used to define the constants related to the rectangular mesh.
    \begin{tcolorbox}[breakable, size=fbox, boxrule=1pt, pad at break*=1mm,colback=cellbackground, colframe=cellborder, enlarge top by=0.25em, enlarge bottom by=0.5em]
\begin{Verbatim}[commandchars=\\\{\}]
\PY{n}{x0}\PY{p}{,} \PY{n}{xL}\PY{p}{,} \PY{n}{y0}\PY{p}{,} \PY{n}{yL} \PY{o}{=} \PY{l+m+mf}{0.}\PY{p}{,} \PY{l+m+mf}{2.}\PY{p}{,} \PY{l+m+mf}{0.}\PY{p}{,} \PY{l+m+mf}{1.}
\end{Verbatim}
\end{tcolorbox}
The time interval related to the time discretization is specified similarly.

    \begin{tcolorbox}[breakable, size=fbox, boxrule=1pt, pad at break*=1mm,colback=cellbackground, colframe=cellborder, enlarge top by=0.25em, enlarge bottom by=0.5em]
\begin{Verbatim}[commandchars=\\\{\}]
\PY{n}{ti}\PY{p}{,} \PY{n}{tf}  \PY{o}{=} \PY{l+m+mf}{0.}\PY{p}{,} \PY{l+m+mf}{5.}
\end{Verbatim}
\end{tcolorbox}
We provide the time step for the time discretization of the pHDAE as well.

    \begin{tcolorbox}[breakable, size=fbox, boxrule=1pt, pad at break*=1mm,colback=cellbackground, colframe=cellborder, enlarge top by=0.25em, enlarge bottom by=0.5em]
\begin{Verbatim}[commandchars=\\\{\}]
\PY{n}{dt} \PY{o}{=} \PY{l+m+mf}{1.e\PYZhy{}3}
\end{Verbatim}
\end{tcolorbox}

    \hypertarget{heath:fenics-expressions-definition}{%
\subsubsection{FEniCS expressions
definition}\label{heat:fenics-expressions-definition}}

Using {\sc{FEniCS}} expressions, the physical parameters related to our model problem are defined.
    \begin{tcolorbox}[breakable, size=fbox, boxrule=1pt, pad at break*=1mm,colback=cellbackground, colframe=cellborder, enlarge top by=0.25em, enlarge bottom by=0.5em]
\begin{Verbatim}[commandchars=\\\{\}]
\PY{n}{rho}    \PY{o}{=} \PY{l+s+s1}{\PYZsq{}}\PY{l+s+s1}{x[0]*(xL\PYZhy{}x[0]) + x[1]*(yL\PYZhy{}x[1]) + 2}\PY{l+s+s1}{\PYZsq{}}
\PY{n}{Lambda11} \PY{o}{=} \PY{l+s+s1}{\PYZsq{}}\PY{l+s+s1}{5. + x[0]*x[1]}\PY{l+s+s1}{\PYZsq{}}
\PY{n}{Lambda12} \PY{o}{=} \PY{l+s+s1}{\PYZsq{}}\PY{l+s+s1}{(x[0]\PYZhy{}x[1])*(x[0]\PYZhy{}x[1])}\PY{l+s+s1}{\PYZsq{}}
\PY{n}{Lambda22} \PY{o}{=} \PY{l+s+s1}{\PYZsq{}}\PY{l+s+s1}{3.+x[1]/(x[0]+1)}\PY{l+s+s1}{\PYZsq{}}
\PY{n}{CV}      \PY{o}{=} \PY{l+s+s1}{\PYZsq{}}\PY{l+s+s1}{3.}\PY{l+s+s1}{\PYZsq{}}
\end{Verbatim}
\end{tcolorbox}
The initial conditions of the problem related to the
temperature and to the flow and effort variables are then given. The temperature
follows a Gaussian behaviour for which we specify related parameters.

    \begin{tcolorbox}[breakable, size=fbox, boxrule=1pt, pad at break*=1mm,colback=cellbackground, colframe=cellborder, enlarge top by=0.25em, enlarge bottom by=0.5em]
\begin{Verbatim}[commandchars=\\\{\}]
\PY{n}{ampl}\PY{p}{,} \PY{n}{sX}\PY{p}{,} \PY{n}{sY}\PY{p}{,} \PY{n}{X0}\PY{p}{,} \PY{n}{Y0}    \PY{o}{=} \PY{l+m+mi}{1000}\PY{p}{,} \PY{n}{xL}\PY{o}{/}\PY{l+m+mi}{6}\PY{p}{,} \PY{n}{yL}\PY{o}{/}\PY{l+m+mi}{6}\PY{p}{,} \PY{n}{xL}\PY{o}{/}\PY{l+m+mi}{2}\PY{p}{,} \PY{n}{yL}\PY{o}{/}\PY{l+m+mi}{2} 
\PY{n}{au\PYZus{}0} \PY{o}{=} \PY{l+s+s1}{\PYZsq{}}\PY{l+s+s1}{3000}\PY{l+s+s1}{\PYZsq{}}
\PY{n}{eu\PYZus{}0} \PY{o}{=} \PY{l+s+s1}{\PYZsq{}}\PY{l+s+s1}{ ampl * exp(\PYZhy{} pow( (x[0]\PYZhy{}X0)/sX, 2) \PYZhy{} pow( (x[1]\PYZhy{}Y0)/sY, 2) ) + 1000}\PY{l+s+s1}{\PYZsq{}}
\end{Verbatim}
\end{tcolorbox}
The spatial part of the boundary control function is defined next.

    \begin{tcolorbox}[breakable, size=fbox, boxrule=1pt, pad at break*=1mm,colback=cellbackground, colframe=cellborder, enlarge top by=0.25em, enlarge bottom by=0.5em]
\begin{Verbatim}[commandchars=\\\{\}]
\PY{n}{Ub\PYZus{}sp0}  \PY{o}{=} \PY{l+s+s1}{\PYZsq{}\PYZsq{}\PYZsq{}}\PY{l+s+s1}{ }
\PY{l+s+s1}{           ( abs(x[0]) \PYZlt{}= DOLFIN\PYZus{}EPS ?  1. * (yL\PYZhy{}x[1])*x[1] : 0 )}
\PY{l+s+s1}{        + ( abs(x[1])  \PYZlt{}= DOLFIN\PYZus{}EPS ?  1. * (xL\PYZhy{}x[0])*x[0] : 0 )            }
\PY{l+s+s1}{        + ( abs(xL \PYZhy{} x[0]) \PYZlt{}= DOLFIN\PYZus{}EPS ?  \PYZhy{}15. * (yL\PYZhy{}x[1])*x[1] : 0 )}
\PY{l+s+s1}{        + ( abs(yL \PYZhy{} x[1]) \PYZlt{}= DOLFIN\PYZus{}EPS  ? 1. * (xL\PYZhy{}x[0])*x[0] : 0 )      }
\PY{l+s+s1}{        }\PY{l+s+s1}{\PYZsq{}\PYZsq{}\PYZsq{}}
\end{Verbatim}
\end{tcolorbox}
Finally we define the time-dependent part of the boundary control
as a pure Python function. The whole boundary control function
is then given as the product of the two quantities  (\verb|Ub_sp0| and
\verb|Ub_tm0|, respectively).

    \begin{tcolorbox}[breakable, size=fbox, boxrule=1pt, pad at break*=1mm,colback=cellbackground, colframe=cellborder, enlarge top by=0.25em, enlarge bottom by=0.5em]
\begin{Verbatim}[commandchars=\\\{\}]
\PY{k}{def} \PY{n+nf}{Ub\PYZus{}tm0}\PY{p}{(}\PY{n}{t}\PY{p}{)}\PY{p}{:}
    \PY{k}{if} \PY{n}{t}\PY{o}{\PYZlt{}}\PY{o}{=}\PY{l+m+mi}{2}\PY{p}{:}
        \PY{k}{return} \PY{l+m+mi}{500} \PY{o}{*} \PY{n}{sin}\PY{p}{(}\PY{l+m+mi}{2} \PY{o}{*} \PY{n}{pi} \PY{o}{*} \PY{n}{t}\PY{p}{)} 
    \PY{k}{else}\PY{p}{:} \PY{k}{return} \PY{l+m+mi}{0}
\end{Verbatim}
\end{tcolorbox}

    \hypertarget{heath:problem-at-the-continuous-level}{%
\subsubsection{Problem at the continuous
level}\label{heat:problem-at-the-continuous-level}}

    We are now able to completely define the problem at the continuous
level.

    \begin{tcolorbox}[breakable, size=fbox, boxrule=1pt, pad at break*=1mm,colback=cellbackground, colframe=cellborder, enlarge top by=0.25em, enlarge bottom by=0.5em]
\begin{Verbatim}[commandchars=\\\{\}]
\PY{n}{H}\PY{o}{.}\PY{n}{Set\PYZus{}Rectangular\PYZus{}Domain}\PY{p}{(}\PY{n}{x0}\PY{p}{,} \PY{n}{xL}\PY{p}{,} \PY{n}{y0}\PY{p}{,} \PY{n}{yL}\PY{p}{)}\PY{p}{;}
\PY{n}{H}\PY{o}{.}\PY{n}{Set\PYZus{}Initial\PYZus{}Final\PYZus{}Time}\PY{p}{(}\PY{n}{initial\PYZus{}time}\PY{o}{=}\PY{n}{ti}\PY{p}{,} \PY{n}{final\PYZus{}time}\PY{o}{=}\PY{n}{tf}\PY{p}{)}\PY{p}{;}
\PY{n}{H}\PY{o}{.}\PY{n}{Set\PYZus{}Physical\PYZus{}Parameters}\PY{p}{(}\PY{n}{rho}\PY{o}{=}\PY{n}{rho}\PY{p}{,} \PY{n}{Lambda11}\PY{o}{=}\PY{n}{Lambda11}\PY{p}{,} \PY{n}{Lambda12}\PY{o}{=}\PY{n}{Lambda12}\PY{p}{,} \PY{n}{Lambda22}\PY{o}{=}\PY{n}{Lambda22}\PY{p}{,} \PY{n}{CV}\PY{o}{=}\PY{n}{CV}\PY{p}{)}\PY{p}{;}
\end{Verbatim}
\end{tcolorbox}

    \hypertarget{heath:problem-at-the-discrete-level-in-space-and-time}{%
\subsubsection{Problem at the discrete level in space and
time}\label{heat:problem-at-the-discrete-level-in-space-and-time}}

    The structure-preserving discretization of the
infinite-dimensional pHs with PFEM is
described in detail in \cite{SerMatHai19c}. This leads to the
pHDAE given in \eqref{eq:PFEMheat}. The definition of the system at the discrete level follows the same steps as for the two previous examples.

    \begin{tcolorbox}[breakable, size=fbox, boxrule=1pt, pad at break*=1mm,colback=cellbackground, colframe=cellborder, enlarge top by=0.25em, enlarge bottom by=0.5em]
\begin{Verbatim}[commandchars=\\\{\}]
\PY{n}{H}\PY{o}{.}\PY{n}{Set\PYZus{}Gmsh\PYZus{}Mesh}\PY{p}{(}\PY{n}{xmlfile}\PY{o}{=}\PY{l+s+s1}{\PYZsq{}}\PY{l+s+s1}{rectangle.xml}\PY{l+s+s1}{\PYZsq{}}\PY{p}{,} \PY{n}{rfn\PYZus{}num}\PY{o}{=}\PY{l+m+mi}{1}\PY{p}{)}\PY{p}{;}
\PY{n}{H}\PY{o}{.}\PY{n}{Set\PYZus{}Finite\PYZus{}Element\PYZus{}Spaces}\PY{p}{(}\PY{n}{family\PYZus{}q}\PY{o}{=}\PY{l+s+s1}{\PYZsq{}}\PY{l+s+s1}{RT}\PY{l+s+s1}{\PYZsq{}}\PY{p}{,} \PY{n}{family\PYZus{}p}\PY{o}{=}\PY{l+s+s1}{\PYZsq{}}\PY{l+s+s1}{P}\PY{l+s+s1}{\PYZsq{}}\PY{p}{,} \PY{n}{family\PYZus{}b}\PY{o}{=}\PY{l+s+s1}{\PYZsq{}}\PY{l+s+s1}{P}\PY{l+s+s1}{\PYZsq{}}\PY{p}{,}\PY{n}{rq}\PY{o}{=}\PY{l+m+mi}{0}\PY{p}{,} \PY{n}{rp}\PY{o}{=}\PY{l+m+mi}{1}\PY{p}{,} \PY{n}{rb}\PY{o}{=}\PY{l+m+mi}{1}\PY{p}{)}\PY{p}{;}
\PY{n}{H}\PY{o}{.}\PY{n}{Assembly}\PY{p}{(}\PY{p}{)}\PY{p}{;}
\end{Verbatim}
\end{tcolorbox}
To perform the time integration of the pHDAE, we first need to set
and interpolate the initial data and the boundary control function on
the appropriate finite element spaces. Then, the time step is specified.

\begin{tcolorbox}[breakable, size=fbox, boxrule=1pt, pad at break*=1mm,colback=cellbackground, colframe=cellborder, enlarge top by=0.25em, enlarge bottom by=0.5em]
\begin{Verbatim}[commandchars=\\\{\}]
\PY{n}{H}\PY{o}{.}\PY{n}{Set\PYZus{}Initial\PYZus{}Data}\PY{p}{(}\PY{n}{au\PYZus{}0}\PY{o}{=}\PY{n}{au\PYZus{}0}\PY{p}{,} \PY{n}{eu\PYZus{}0}\PY{o}{=}\PY{n}{eu\PYZus{}0}\PY{p}{,} \PY{n}{ampl}\PY{o}{=}\PY{n}{ampl}\PY{p}{,} \PY{n}{sX}\PY{o}{=}\PY{n}{sX}\PY{p}{,} \PY{n}{sY}\PY{o}{=}\PY{n}{sY}\PY{p}{,} \PY{n}{X0}\PY{o}{=}\PY{n}{X0}\PY{p}{,} \PY{n}{Y0}\PY{o}{=}\PY{n}{Y0}\PY{p}{)}\PY{p}{;}
\PY{n}{H}\PY{o}{.}\PY{n}{Project\PYZus{}Initial\PYZus{}Data}\PY{p}{(}\PY{p}{)}\PY{p}{;}
\PY{n}{H}\PY{o}{.}\PY{n}{Set\PYZus{}Boundary\PYZus{}Control}\PY{p}{(}\PY{n}{Ub\PYZus{}tm0}\PY{o}{=}\PY{n}{Ub\PYZus{}tm0}\PY{p}{,} \PY{n}{Ub\PYZus{}sp0}\PY{o}{=}\PY{n}{Ub\PYZus{}sp0}\PY{p}{,} \PY{n}{Ub\PYZus{}tm1}\PY{o}{=}\PY{k}{lambda} \PY{n}{t} \PY{p}{:}\PY{l+m+mi}{0}\PY{p}{,} \PY{n}{Ub\PYZus{}sp1}\PY{o}{=}\PY{l+s+s1}{\PYZsq{}}\PY{l+s+s1}{1000}\PY{l+s+s1}{\PYZsq{}}\PY{p}{)}\PY{p}{;}
\PY{n}{H}\PY{o}{.}\PY{n}{Project\PYZus{}Boundary\PYZus{}Control}\PY{p}{(}\PY{p}{)}\PY{p}{;}
\PY{n}{H}\PY{o}{.}\PY{n}{Set\PYZus{}Time\PYZus{}Setting}\PY{p}{(}\PY{n}{dt}\PY{p}{)}\PY{p}{;}
\end{Verbatim}
\end{tcolorbox}

    \hypertarget{numerical-approximation-of-the-space-time-solution}{%
\subsubsection{Numerical approximation of the space-time
solution}\label{heat:numerical-approximation-of-the-space-time-solution}}

    Now we perform the time integration of the resulting pHDAE system and
deduce the behaviour of the energy variables, the Hamiltonian with
respect to the time and space variables, respectively. For the time discretization, we employ a fully explicit scheme, presented in \cite{serh:20} (Algorithm 2 of Section 4.4) as a first attempt.

    \begin{tcolorbox}[breakable, size=fbox, boxrule=1pt, pad at break*=1mm,colback=cellbackground, colframe=cellborder, enlarge top by=0.25em, enlarge bottom by=0.5em]
\begin{Verbatim}[commandchars=\\\{\}]
\PY{n}{H}\PY{o}{.}\PY{n}{Set\PYZus{}Formulation}\PY{p}{(}\PY{l+s+s1}{\PYZsq{}}\PY{l+s+s1}{div}\PY{l+s+s1}{\PYZsq{}}\PY{p}{)}
\PY{n}{alpha_s}\PY{p}{,} \PY{n}{fS}\PY{p}{,} \PY{n}{fsig}\PY{p}{,} \PY{n}{es}\PY{p}{,} \PY{n}{eS}\PY{p}{,} \PY{n}{esig}\PY{p}{,} \PY{n}{Hamiltonian} \PY{o}{=} \PY{n}{H}\PY{o}{.}\PY{n}{Integration\PYZus{}DAE}\PY{p}{(}\PY{p}{)}\PY{p}{;}
\end{Verbatim}
\end{tcolorbox}

    \hypertarget{heath:post-processing}{%
\subsubsection{Post-processing}\label{heat:post-processing}}
    
\begin{figure}[!htbp]%
    \centering%
    \includegraphics[width=0.55\textwidth]{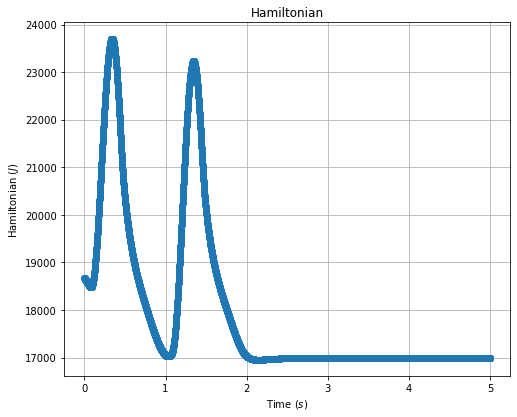}%
    \caption{Hamiltonian (internal energy) versus time for the heat equation.}%
    \label{fig:HamiltonianHeat}%
\end{figure}%
As an illustration, we plot the Hamiltonian function ({\it i.e.} the internal energy) versus time. 
The Hamiltonian function is constant after $2$ seconds, when the boundary control is switched off, as expected by the first law of thermodynamics.%

\section{Conclusions and perspectives}\label{sec:conclusions}

%
%

We have provided a general structure for the theoretical and numerical solution of infinite-dimensional port-Hamiltonian systems. This structure is particularly appealing since PFEM straightforwardly applies. Concerning the numerical solution, PFEM offers the advantage to leverage robust software components for the discretization of boundary controlled PDEs and time integration.

We have applied this strategy on abstract multidimensional linear hyperbolic and parabolic boundary controlled systems. We have notably shown that model problems based on the wave equation, Mindlin equation and heat equation fit within this unified theoretical framework.  Numerical simulations of infinite-dimensional pHs have been performed with the ongoing software project {\sc{SCRIMP}} that has been briefly introduced. Finally we have illustrated how to solve three case studies within this framework by carefully explaining the methodology, and have provided companion interactive {\sc Jupyter} notebooks.
%
%


Beside the generalization of the classes related to the heat and wave equation to the three-dimensional case, we plan to propose in {\sc{SCRIMP}} more advanced model problems based on the two-dimensional Shallow Water Equation (SWE) \cite{CarBruMatLef19,CarMatLef19}, the Kirchhoff model for thin plates \cite{brugnoli2019kirchhoff} and Maxwell's equations \cite{PayMatHai20}. Furthermore we will investigate both time integration methods that allow structure-preserving time discretization \cite{Kotyczka2018a} of finite dimensional pHs and more accurate time integrators for nonlinear pHDAE. In addition we plan to enrich the panel of structure-preserving model reduction algorithms to facilitate the simulation of large-scale port-Hamiltonian systems.  This is an essential prerequisite before first attempts related to control design.
Further developments foresee the comparisons with well-established algorithms for multi-physics problems leading to coupled systems of PDEs.
\section*{Acknowledgments}
This work is supported by the project ANR-16-CE92-0028,
entitled {\em Interconnected Infinite-Dimensional systems for Heterogeneous
Media}, INFIDHEM, financed by the French National Research Agency (ANR) and the Deutsche Forschungsgemeinschaft (DFG). Further information is available at {\url{https://websites.isae-supaero.fr/infidhem/the-project}}. \\
Moreover the authors would like to thank Michel Sala\"un and Denis Matignon for the fruitful and insightful discussions.

\bibliographystyle{plain}

\bibliography{main}

\newpage 

\end{document}